\documentclass[a4paper,11pt]{scrartcl}

\usepackage{amsthm}
\usepackage{bbm}
\usepackage{authblk}
\usepackage[pdflatex]{crop}
\usepackage[bookmarks]{hyperref}

\usepackage{graphbox}
\usepackage[utf8]{inputenc}
\usepackage[english]{babel}
\usepackage{amssymb}
\usepackage{amsmath}
\usepackage{latexsym}
\usepackage{eucal}
\usepackage{bbm}
\usepackage{mathtools}
\usepackage{tikz}
\usepackage{tikz-cd}
\usepackage{comment}
\allowdisplaybreaks

\setkomafont{disposition}{\normalfont\bfseries}
\newcommand{\auth}[0]{Tobias Fritz and Paolo Perrone}
\newcommand{\tit}[0]{Bimonoidal Structure of Probability Monads}
\newcommand{\kw}[0]{Monoidal monads, probabilistic powerdomain, strong monads, randomness, stochastic correlation, Kantorovich duality, Wasserstein distance.}

\hypersetup{
pdfauthor={\auth},%
pdftitle={\tit},%
colorlinks, linktocpage=true, pdfstartpage=1, pdfstartview=FitV,%
breaklinks=true, pdfpagemode=UseNone, pageanchor=true, pdfpagemode=UseOutlines,%
plainpages=false, bookmarksnumbered, bookmarksopen=true, bookmarksopenlevel=1,%
hypertexnames=true, pdfhighlight=/O,%
urlcolor=black, linkcolor=black, citecolor=black, 
}
\numberwithin{equation}{section}

\theoremstyle{plain}
\newtheorem{thm}{Theorem}[section]

\newtheorem{prop}[thm]{Proposition}
\newtheorem{cor}[thm]{Corollary}

\newtheorem{deph}[thm]{Definition}

\theoremstyle{definition}

\newcommand{\R}{\mathbb{R}}

\newcommand{\cat}[1]{{\mathsf{#1}}} 
\newcommand{\ar}[2][]{\arrow{#2}{#1}}
\newcommand{\uni}[2][]{\arrow[dashrightarrow]{#2}{#1}} 
\newcommand{\idar}[2][]{\arrow[equal]{#2}{#1}} 

\newcommand{\id}{\mathrm{id}} 

\let\originalleft\left
\let\originalright\right
\renewcommand{\left}{\mathopen{}\mathclose\bgroup\originalleft}
\renewcommand{\right}{\aftergroup\egroup\originalright}

\title{Bimonoidal Structure of\\Probability Monads}
\author[1]{Tobias Fritz\thanks{tfritz [at] pitp.ca}}
\affil[1]{\small Perimeter Institute for Theoretical Physics, Waterloo, ON (Canada)}
\author[2]{Paolo Perrone\thanks{pperrone [at] mit.edu}}
\affil[2]{Massachusetts Institute of Technology, Cambridge, MA (U.S.A.)}
\date{August 2018}

\begin{document}

\maketitle

\begin{abstract} 
\addcontentsline{toc}{section}{Abstract}

 We give a conceptual treatment of the notion of joints, marginals, and independence in the setting of categorical probability. This is achieved by endowing the usual probability monads (like the Giry monad) with a monoidal and an opmonoidal structure, mutually compatible (i.e.\ a bimonoidal structure). If the underlying monoidal category is cartesian monoidal, a bimonoidal structure is given uniquely by a commutative strength. However, if the underlying monoidal category is not cartesian monoidal, a strength is not enough to guarantee all the desired properties of joints and marginals. A bimonoidal structure is then the correct requirement for the more general case. 
 
 We explain the theory and the operational interpretation, with the help of the graphical calculus for monoidal categories. We give a definition of stochastic independence based on the bimonoidal structure, compatible with the intuition and with other approaches in the literature for cartesian monoidal categories. We then show as an example that the Kantorovich monad on the category of complete metric spaces is a bimonoidal monad for a non-cartesian monoidal structure.  
  
\end{abstract}

\newpage

\section{Introduction}\label{intro}

The standard way to treat randomness categorically is via a \emph{probability monad}, of which classic examples are the Giry monad \cite{giry} and the probabilistic powerdomain \cite{jones-plotkin}.
The interpretation is the following: let $\cat{C}$ be a category whose objects we think of as spaces of possible values that a variable may assume. A probability monad $P$ on $\cat{C}$ makes it possible to talk about random variables on objects $X\in\cat{C}$, or equivalently random elements of $X$: an element $p\in PX$ specifies the \emph{law} of a random variable on $X$. 

A central theme of probability theory is that random variables can form joints and marginals. For this to make sense in $\cat{C}$, we need $\cat{C}$ to be a monoidal category, and we need $P$ to interact well with the monoidal structure. We argue that this interaction is best modelled in terms of a bimonoidal structure.

A first structure which links a monad with the tensor product in a category is that of a \emph{strength}. A strength for a probability monad is a natural map $X\otimes PY\to P(X\otimes Y)$, whose interpretation is the following: an element of $X$ and a random element of $Y$ determine uniquely a random element of $X\otimes Y$ which has the correct marginals, and whose randomness is all in the $Y$ component. In the language of probability theory, $(x,q)\in X\otimes PY$ defines the product distribution of $\delta_x$ and $q$ on $X\otimes Y$. 
In the literature, the operational meaning of a strength for a monad, which includes the usage in probability, is well explained in \cite{comp-monads}, and in \cite{jones-plotkin} for the case of the probabilistic powerdomain. A compendium of probability monads appearing in the literature, with information about their strength, can be found in \cite{jacobs}. 

The monoidal structure can be thought of as a refinement of the idea of strength. The basic idea is that given two probability measures $p\in PX$ and $q\in PY$, one can canonically define a probability measure $p\otimes_\nabla q\in P(X\otimes Y)$, the ``product distribution''\footnote{Our reason for denoting it by $p\otimes_\nabla q$ rather than by $p\otimes q$ is that we want to interpret $p : 1\to PX$ and $q : 1\to PY$ as morphisms, so that $p\otimes q : 1\otimes 1 \to PX\otimes PY$ is not yet the product distribution. Rather, one needs to compose $p\otimes q$ with the monoidal structure $\nabla : PX\otimes PY \to P(X\otimes Y)$, which is the subject of the present paper, see Section \ref{probc}.}. This is not the only possible joint distribution that $p$ and $q$ have, but it can be obtained without additional knowledge (of their correlation). 
When a strength satisfies suitable symmetry conditions (commutative strength) it defines automatically a monoidal structure \cite{kock,logrelations}. 

An opmonoidal structure formalizes the dual intuition, namely that given a joint probability distribution $r\in P(X\otimes Y)$ we canonically have the marginals on $PX$ and $PY$ as well. A bimonoidal structure is a compatible way of combining the two structures, in a way consistent with the usual properties of products and marginals in probability.
When the underlying category is cartesian monoidal, then $P$ is automatically opmonoidal. In this case, we show that if $P$ carries a monoidal structure, then it is automatically bimonoidal. Therefore a commutative strong monad on a cartesian monoidal category is canonically bimonoidal. This is for example the case of the probabilistic powerdomain \cite{jones-plotkin}. We argue that the bimonoidal structure is the structure of relevance for probability theory: if the underlying category is not cartesian monoidal, or the strength is not commutative, then one cannot talk about joints and marginals in the usual way just by having a strong monad.

However, not every probability monad in the literature is bimonoidal, not even strong; a famous counterexample is in \cite{sato}. While a non-bimonoidal probability monad could be of use in measure theory to talk about spaces of measures, it would be far from applications to probability, since it would not permit talking about concepts like stochastic independence and correlation, which in probability theory play a central role. We thus want to argue that in order for a monad to \emph{really} count as a probability monad, it should be a bimonoidal monad.

In Section \ref{semicart} we describe the setting of semicartesian monoidal categories and affine monads, which we argue is the one of relevance for classical probability theory. In such a setting, we will represent the concepts using a graphical calculus analogous to that of \cite{mellies}, presented in \ref{graphicalcalculus}.
In Section \ref{pmon} we will sketch the basic theory and interpretation of a bimonoidal structure for probability monads, using the graphical calculus. The same definitions in terms of commutative diagrams can be found in Appendix \ref{monoidalstuff}. 
In \ref{rv}, we will show how this permits to talk about functions between products of random variables. 
In \ref{probc}, we show how to define a category of probability spaces from a probability monad, in such a way that the monoidal structure is inherited. This permits to connect with other treatments of stochastic independence in the literature. 
In \ref{strength} we will see in more detail why this formalism generalizes the strength of probability monads on cartesian monoidal categories. 
In Section \ref{independence}, we give a notion of stochastic independence based on the bimonoidal structure of the monad, and show that it satisfies some of the intuitively expected properties. In \ref{comparison} we show that, if the base category is cartesian monoidal, our definition agrees with the one given by Franz \cite{franz}, and it is compatible with the definition of independence structure given by Simpson \cite{simpson-independence}.
Finally, in Section \ref{bmp} we will give a nontrivial example of a bimonoidal monad, the Kantorovich monad on complete metric spaces \cite{breugel,ours_kantorovich}. The precise proofs and calculations of the statements of Section \ref{bmp} can be found in Appendix \ref{proofs}.

\section{Semicartesian monoidal categories and affine monads}\label{semicart}

By definition, a \emph{semicartesian monoidal category} is a monoidal category in which the monoidal unit $1$ is a terminal object. For probability theory, this is a very appealing feature of a category, because such an object can be interpreted as a trivial space, having only one possible state. In other words, the object $1$ would have the property that for every object $X$, $X\otimes 1 \cong X$ (monoidal unit), so that tensoring with $1$ does not increase the number of possible states, and moreover there is a unique map $!:X\to 1$ (terminal object), which we can think of as ``forgetting the state of $X$''.
Cartesian monoidal categories are in particular semicartesian. Not every monoidal category of interest in probability theory is cartesian, but most of them are semicartesian (in particular, all the ones listed in \cite{jacobs}). 

Semicartesian monoidal categories have another appealing feature for probability: every tensor product space comes equipped with natural projections onto its factors:
\begin{equation*}
 \begin{tikzcd}[row sep=tiny]
  X\otimes Y \ar{r}{\id \otimes !} & X\otimes 1 \ar{r}{\cong} & X, \\
  X\otimes Y \ar{r}{! \otimes \id} & 1\otimes Y \ar{r}{\cong} & Y,
 \end{tikzcd}
\end{equation*}
which satisfy the universal property of the product projections if and only if the category is cartesian monoidal. 
These maps are important in probability theory, because they give the \emph{marginals}. 
Since these projections are automatically natural in $X$ and $Y$, a semicartesian monoidal category is always a \emph{tensor category with projections} in the sense of \cite[Definition 3.3]{franz}; see~\cite{leinster_semicartesian} for more background.\footnote{Conversely, a tensor category equipped with natural projections is semicartesian whenever the projection maps $X\otimes 1\to X$ and $1\otimes X\to X$ coincide with the unitors for all objects $X$. See for example (the dual statement to) \cite[Theorem~3.5]{catlevy}.}

Suppose now that $P$ is a probability monad\footnote{In this work, ``probability monad'' is not a technical term: any monad could be in principle considered a probability monad. We merely use this term in order to indicate our intended interpretation in terms of randomness, as in the case of the Giry monad or the probabilistic powerdomain.} on a semicartesian monoidal category $\cat{C}$. Since we can interpret the unit $1$ has having only one possible (deterministic) state, it is tempting to say that just as well there should be only one possible random state: if there is only one possible outcome, then there is no real randomness. In other words, it is appealing to require that $P(1)\cong 1$. A monad with this condition is called \emph{affine}. Most monads of interest for probability are indeed affine (in particular, again, all the ones listed in \cite{jacobs}). 

Unless otherwise stated, we will always work in a symmetric semicartesian monoidal category with an affine probability monad. These conditions simplify the treatment a lot, while keeping most other conceptual aspects interesting. By the remarks above, they seem to be the right framework for classical probability theory. The definition of monoidal, opmonoidal, and bimonoidal monads can however be given for general braided monoidal categories: the interested reader can find them in Appendix \ref{monoidalstuff}.

\subsection{Graphical calculus}\label{graphicalcalculus}

Here we introduce a form of graphical calculus specializing that of Melliès \cite{mellies} to our setting. Let $\cat{C}$ be a strict symmetric semicartesian monoidal category, and $P$ an affine monad.
We can represent objects $X$ as vertical lines, and morphisms $f:X\to Y$ as boxes:
\begin{equation*}
 \includegraphics[align=c,scale=0.35,keepaspectratio=true]{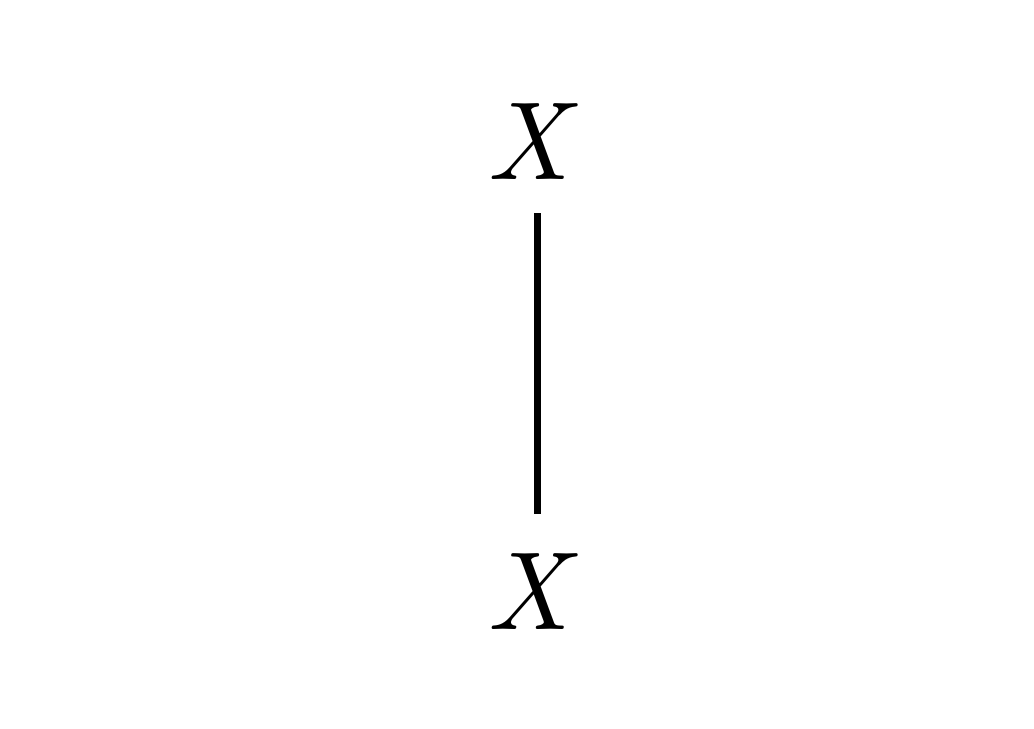}
 \quad\mbox{and}\quad
 \includegraphics[align=c,scale=0.35,keepaspectratio=true]{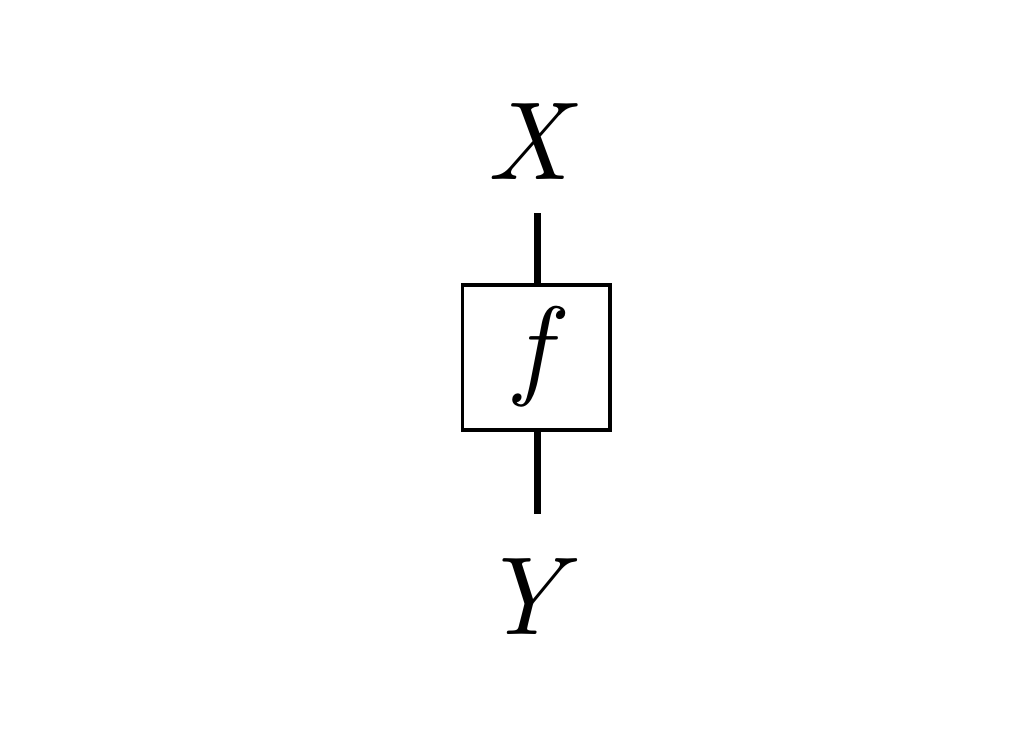}
\end{equation*}
which we read from top to bottom. 

Functor applications are represented by shadings. For example the image $PX$ of $X$ under a functor $P$ and the functor image $Pf:PX\to PY$ of $f$ are:
\begin{equation*}
 \includegraphics[align=c,scale=0.35,keepaspectratio=true]{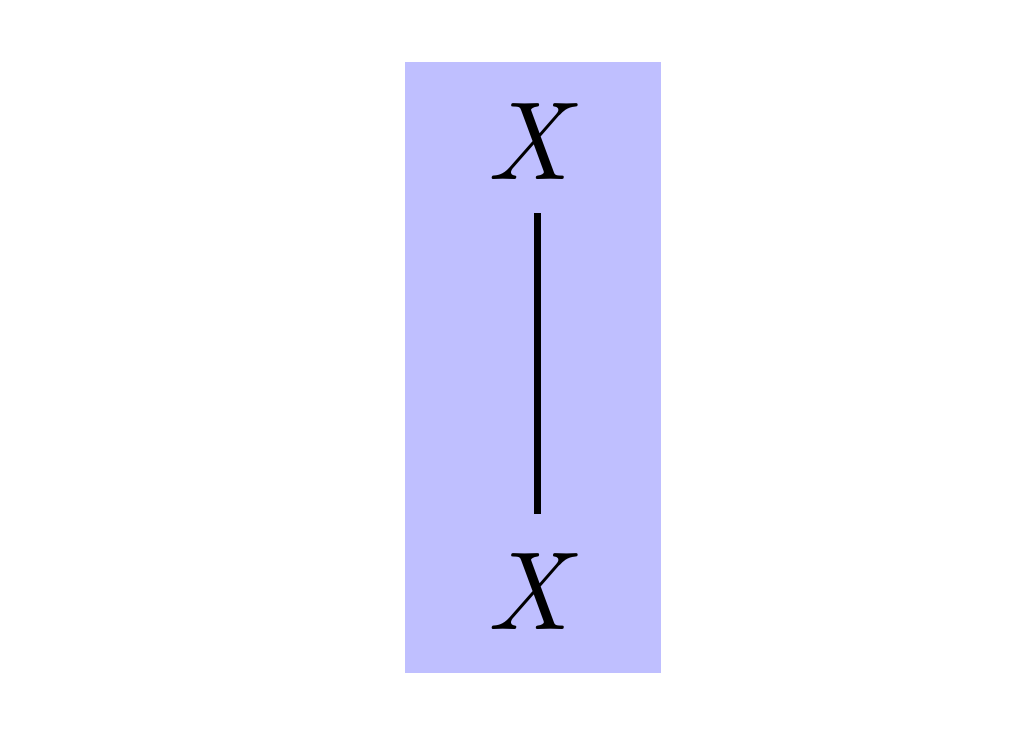}
 \quad\mbox{and}\quad
 \includegraphics[align=c,scale=0.35,keepaspectratio=true]{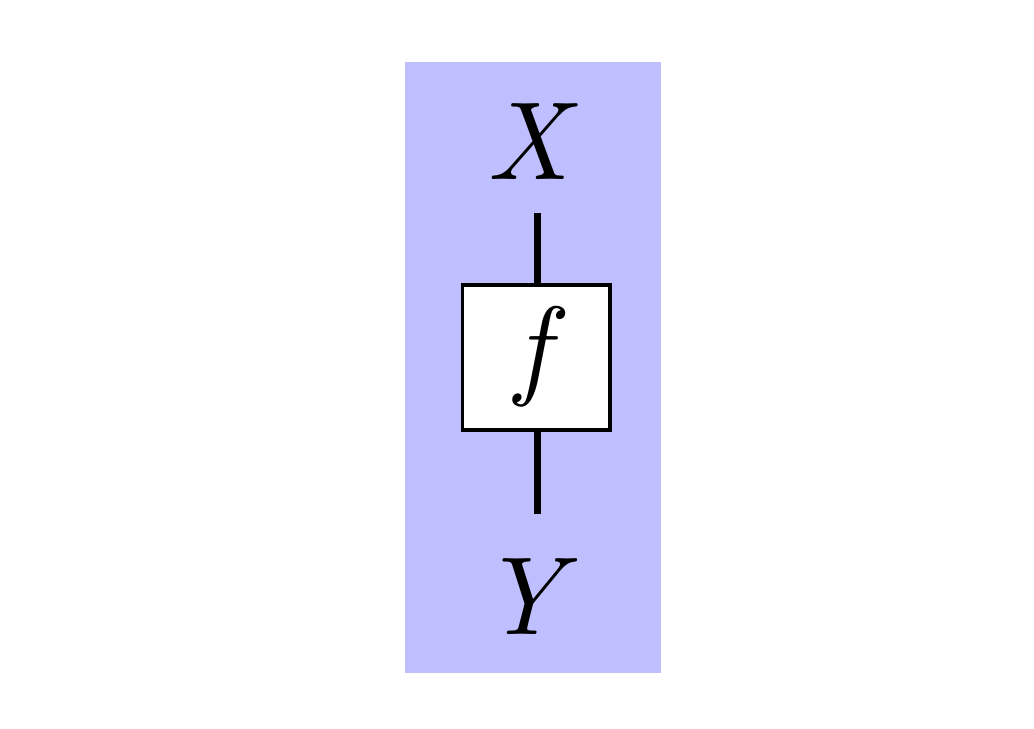}
\end{equation*}

We can represent monoidal products by horizontal juxtaposition. For example, the map $f\otimes g: X\otimes A \to Y\otimes B$ can be represented as: 
\begin{equation*}
 \includegraphics[align=c,scale=0.35,keepaspectratio=true]{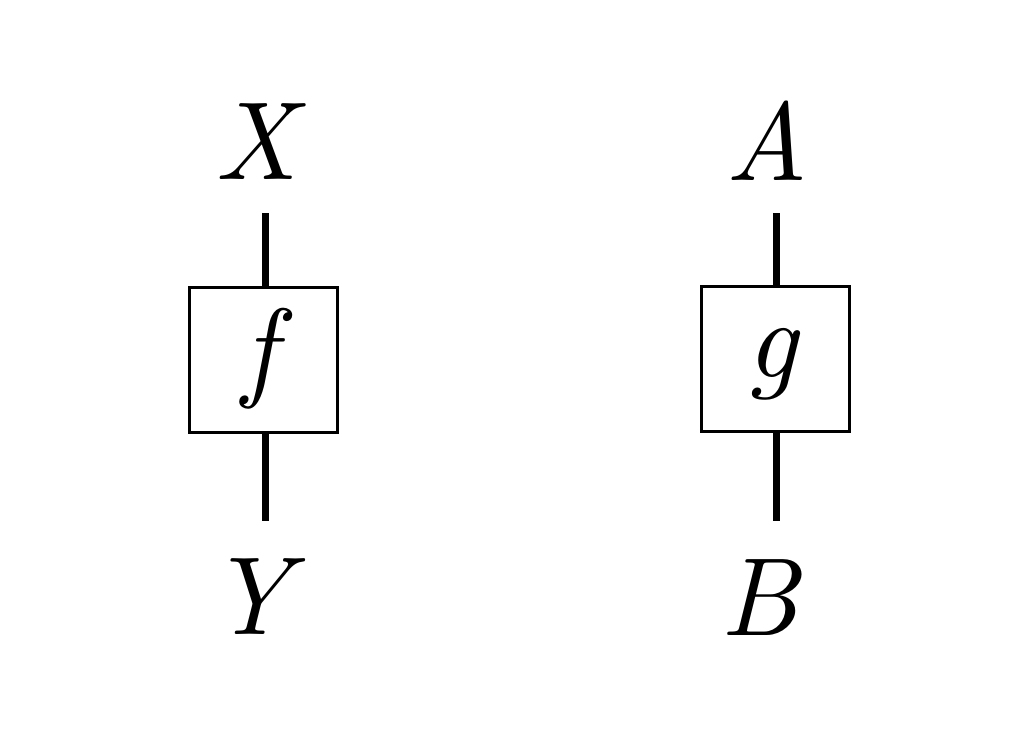}
\end{equation*}
The monoidal unit $1$ is better represented by \emph{nothing}, so that expressions like $X\otimes 1\cong 1\otimes X \cong X$ all have the same representation. However sometimes it is helpful to keep track of it, and in those cases we will draw it as a dotted line:
\begin{equation*}
 \includegraphics[align=c,scale=0.35,keepaspectratio=true]{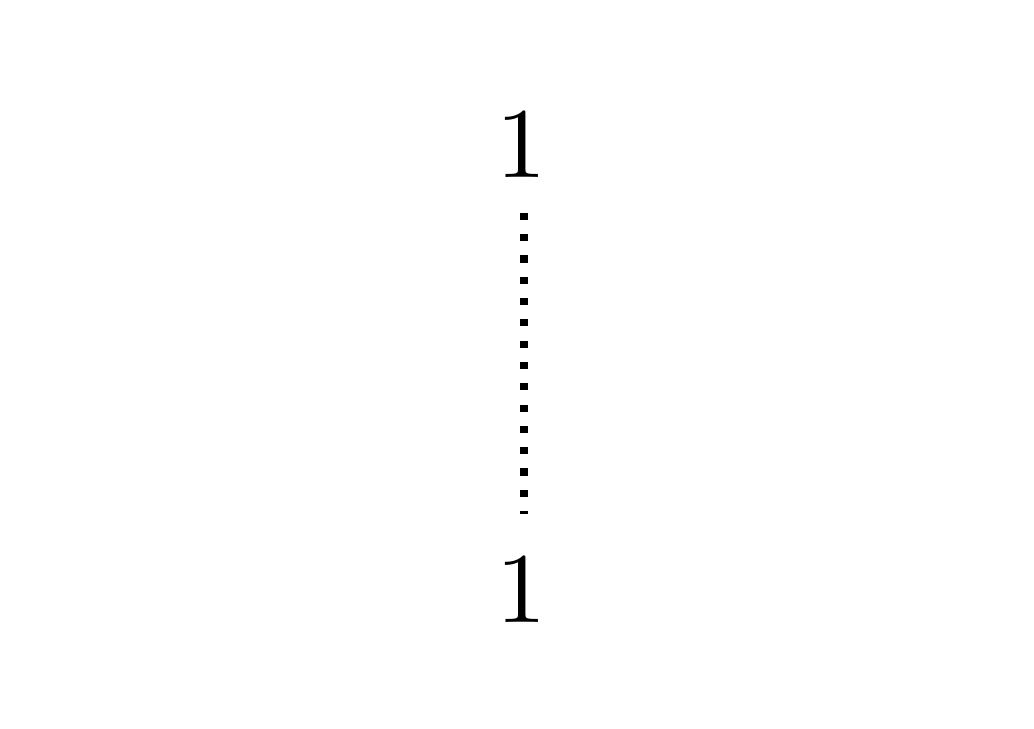}
\end{equation*}
For every object $X$ there is a unique map $!:X\to 1$, which we can interpret as ``forgetting the state of $X$''. We will represent such a map as a ``ground wire'', following the literature on quantum systems:
\begin{equation*}
 \includegraphics[align=c,scale=0.35,keepaspectratio=true]{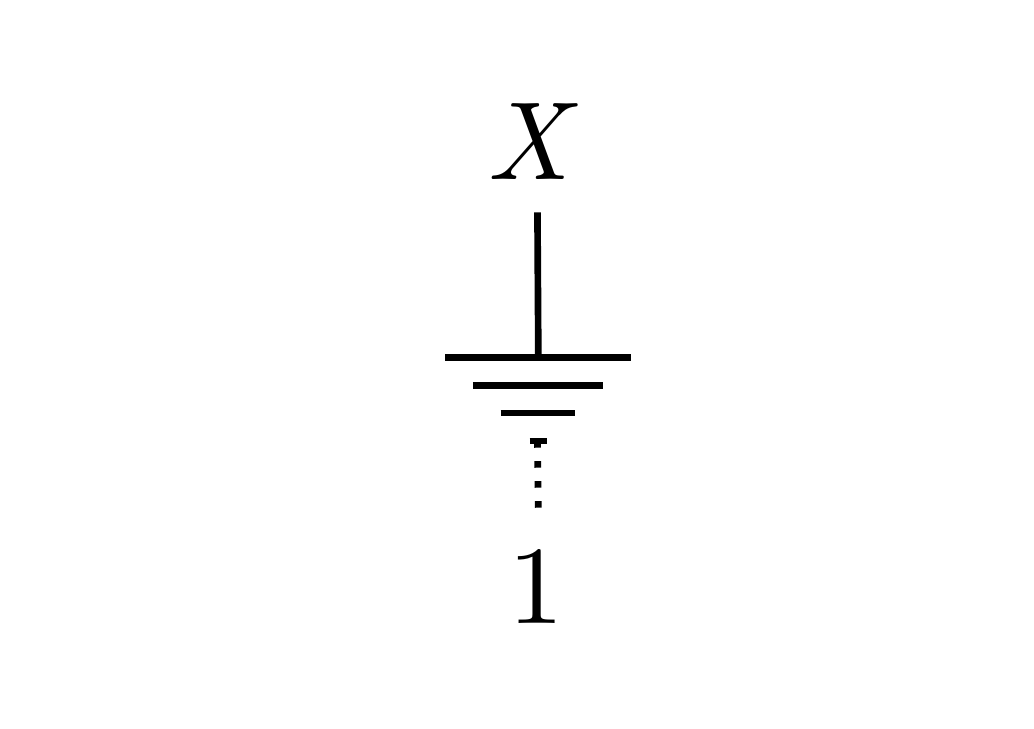}
 \quad\mbox{or, omitting the unit, simply:}\quad
 \includegraphics[align=c,scale=0.35,keepaspectratio=true]{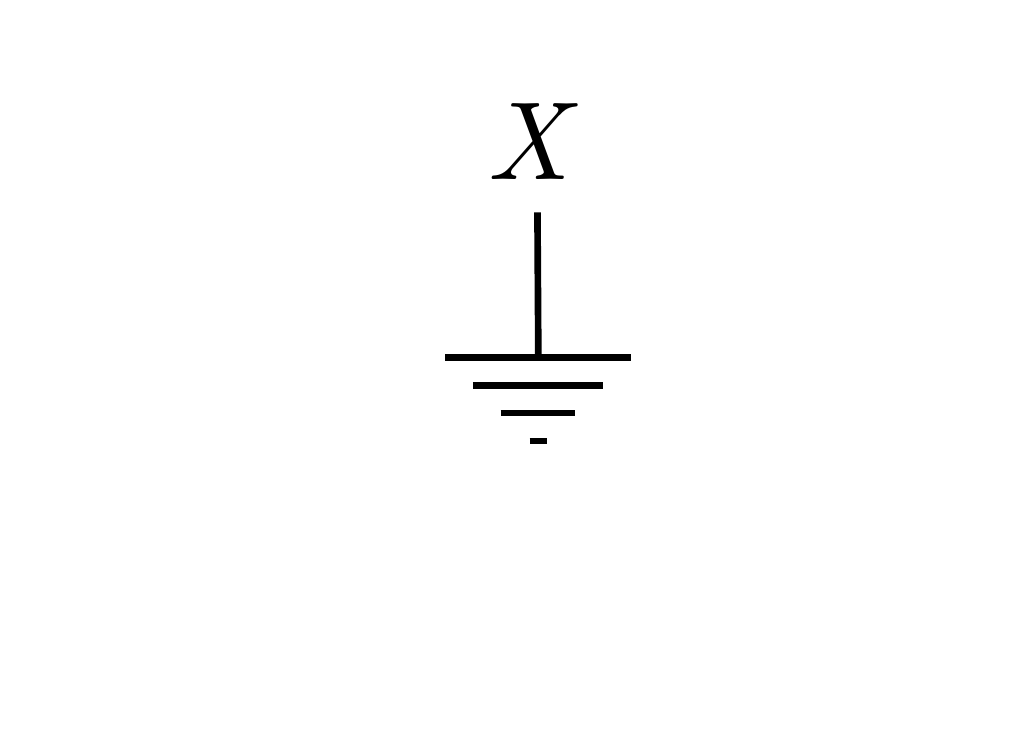}
\end{equation*}
The condition that $P$ is affine, in picture, is
\begin{equation*}
 \includegraphics[align=c,scale=0.35,keepaspectratio=true,clip=true,trim=70pt 0pt 70pt 0pt]{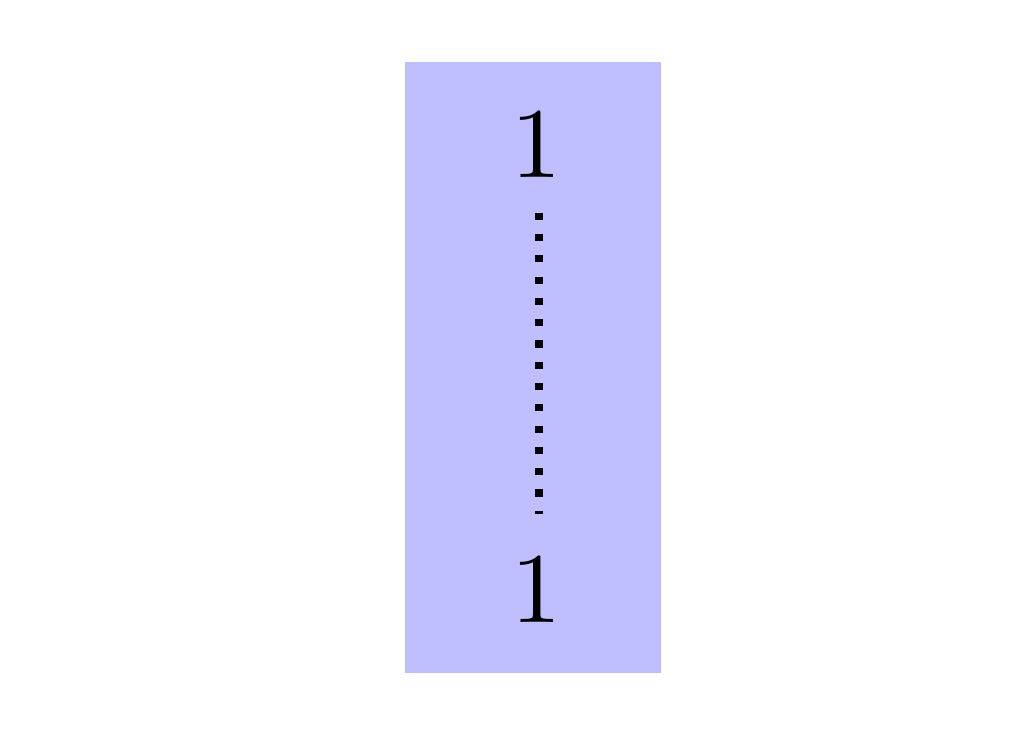}
 =
 \includegraphics[align=c,scale=0.35,keepaspectratio=true,clip=true,trim=70pt 0pt 70pt 0pt]{.//unit.pdf}
 \quad\mbox{or even more trivially:}\quad
 \includegraphics[align=c,scale=0.35,keepaspectratio=true,clip=true,trim=70pt 0pt 70pt 0pt]{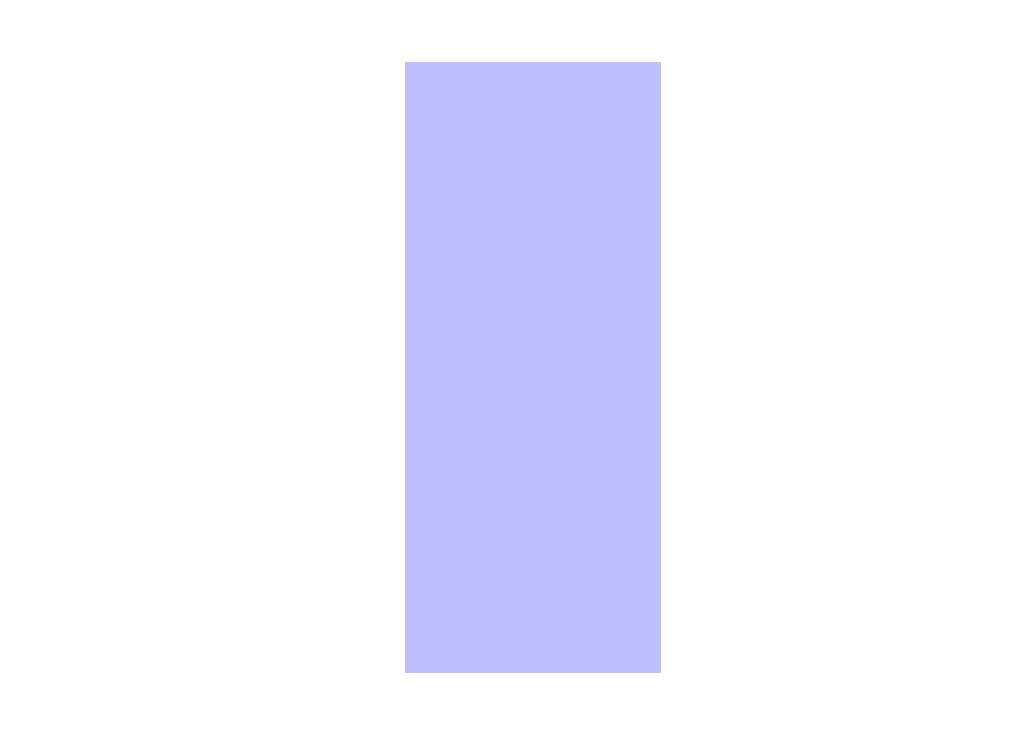}
 =
 \includegraphics[align=c,scale=0.35,keepaspectratio=true,clip=true,trim=70pt 0pt 70pt 0pt]{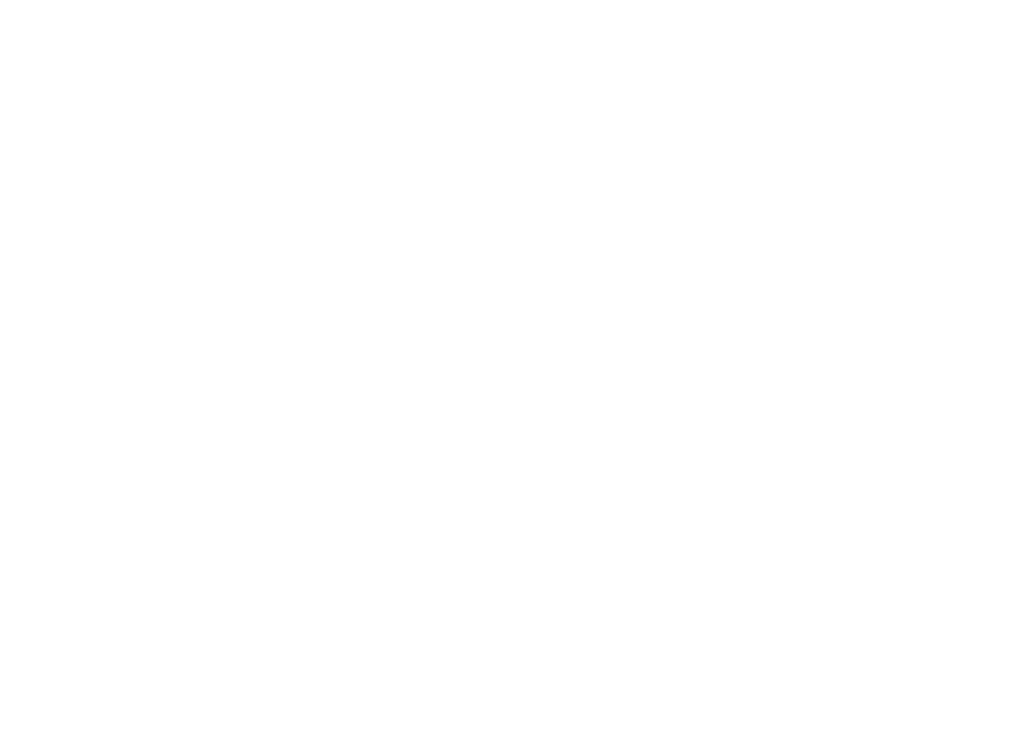}
\end{equation*}
Since we are in a symmetric monoidal category, there is a canonical \emph{braiding} isomorphism $X\otimes Y \to Y\otimes X$. We represent it as:
\begin{equation*}
 \includegraphics[align=c,scale=0.35,keepaspectratio=true]{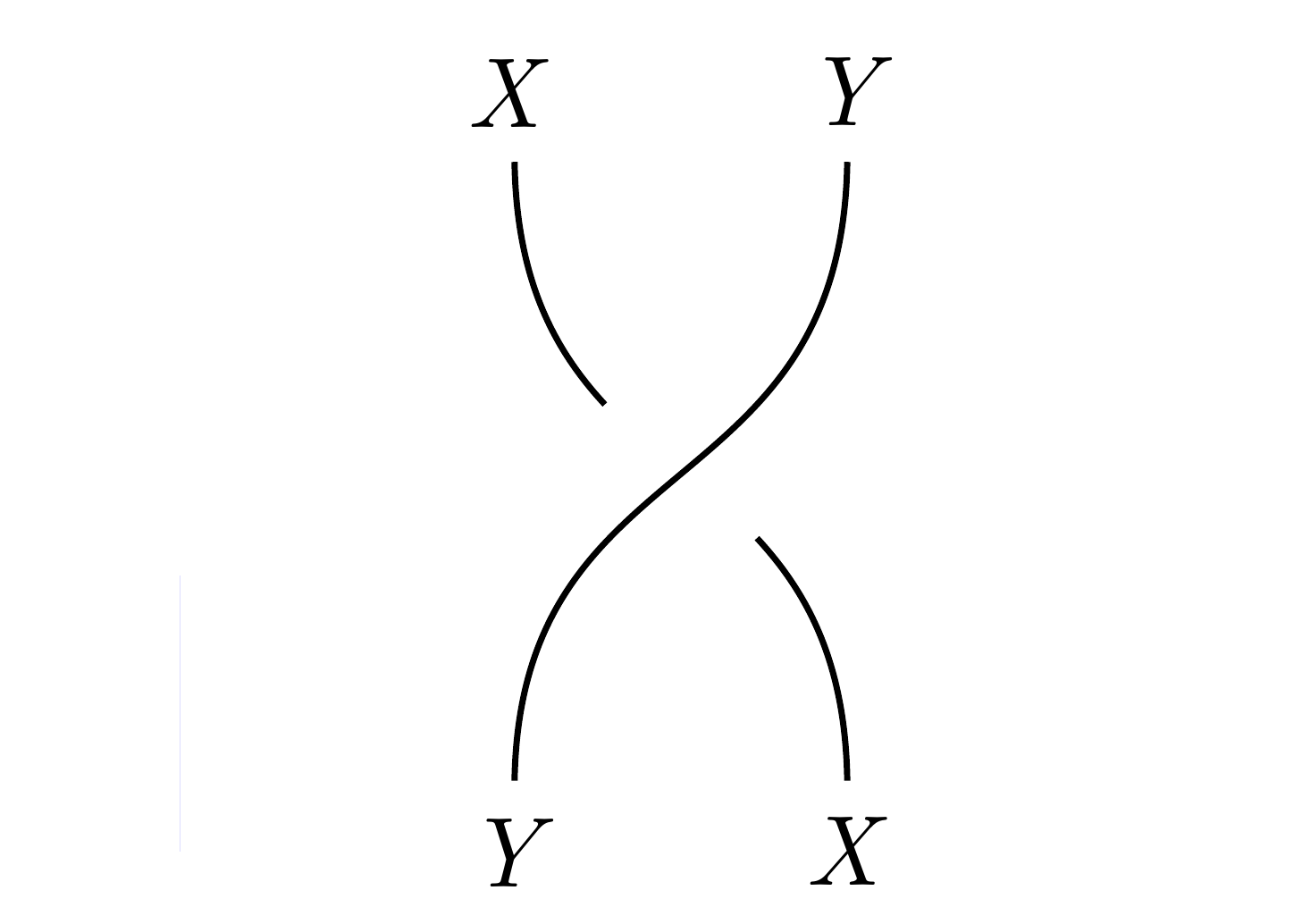}
\end{equation*}
which one can think of as ``swapping'' $X$ and $Y$. In a symmetric monoidal category, if we apply it twice, we obtain the identity.

We turn now to the monad structure of $P$. The monad unit $\delta:X\to PX$ is a natural transformation which ``puts $X$ into a shading'', while the multiplication $E:PPX\to PX$ goes from a double shading to a single shading:
\begin{equation*}
 \includegraphics[align=c,scale=0.35,keepaspectratio=true]{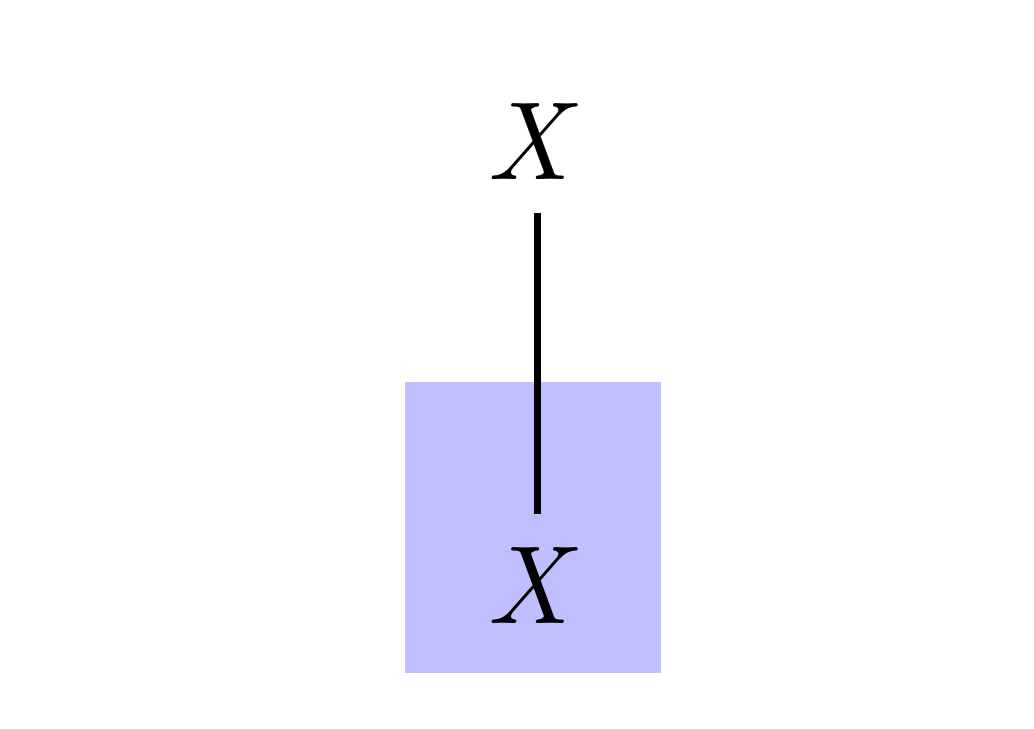}
 \quad\mbox{and}\quad
 \includegraphics[align=c,scale=0.35,keepaspectratio=true]{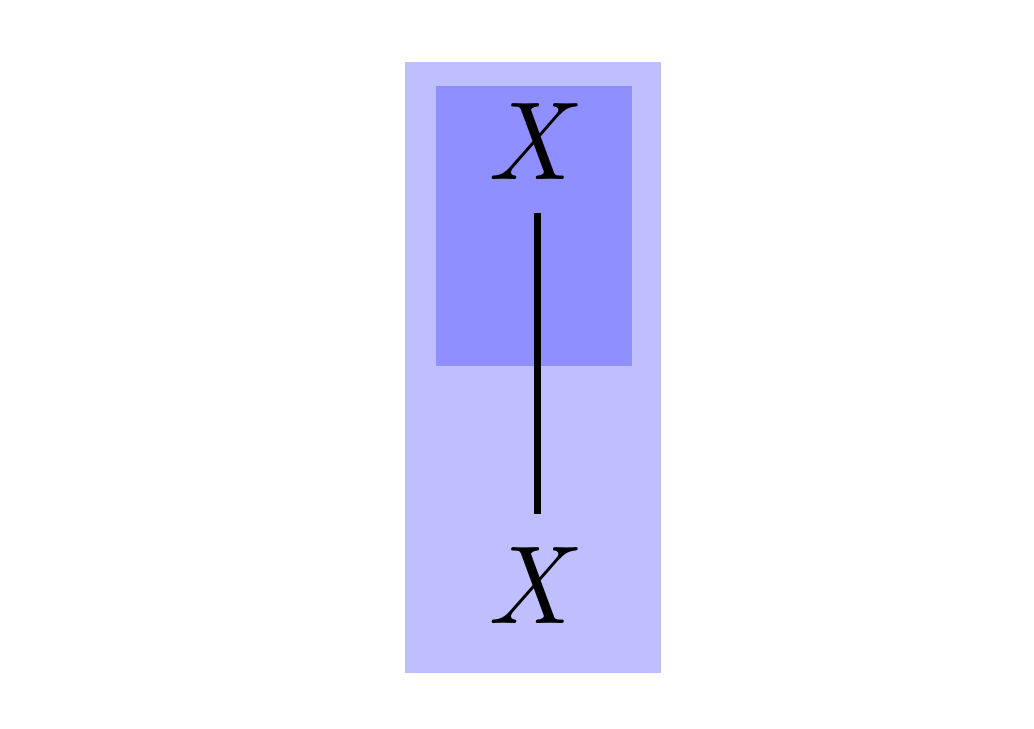}
\end{equation*}
We do not draw a box for these ``structure maps'': we consider them the canonical maps from their source to their target. The diagrams above will always denote $\delta$ and $E$, never other morphisms.

\section{Monoidal structure of probability monads}\label{pmon}

Let $P$ be an affine probability monad on a strict symmetric semicartesian monoidal category $\cat{C}$.
In this setting, a \emph{monoidal structure} for the functor $P$ amounts to a natural map $\nabla:PX\otimes PY\to P(X\otimes Y)$ with associativity and unitality conditions. In terms of graphical calculus, $\nabla$ is a way to pass from $PX\otimes PY$, i.e.:
\begin{equation*}
 \includegraphics[align=c,scale=0.35,keepaspectratio=true,clip=true,trim=0pt 50pt 0pt 50pt]{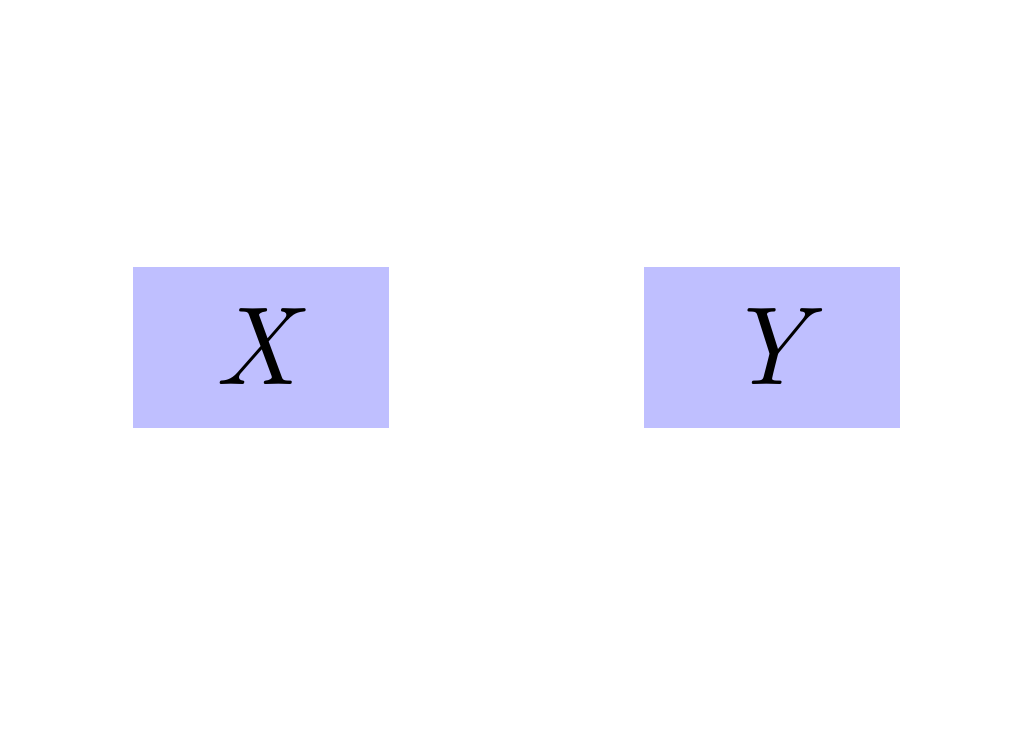}
\end{equation*}
to $P(X\otimes Y)$, i.e.:
\begin{equation*}
 \includegraphics[align=c,scale=0.35,keepaspectratio=true,clip=true,trim=0pt 50pt 0pt 50pt]{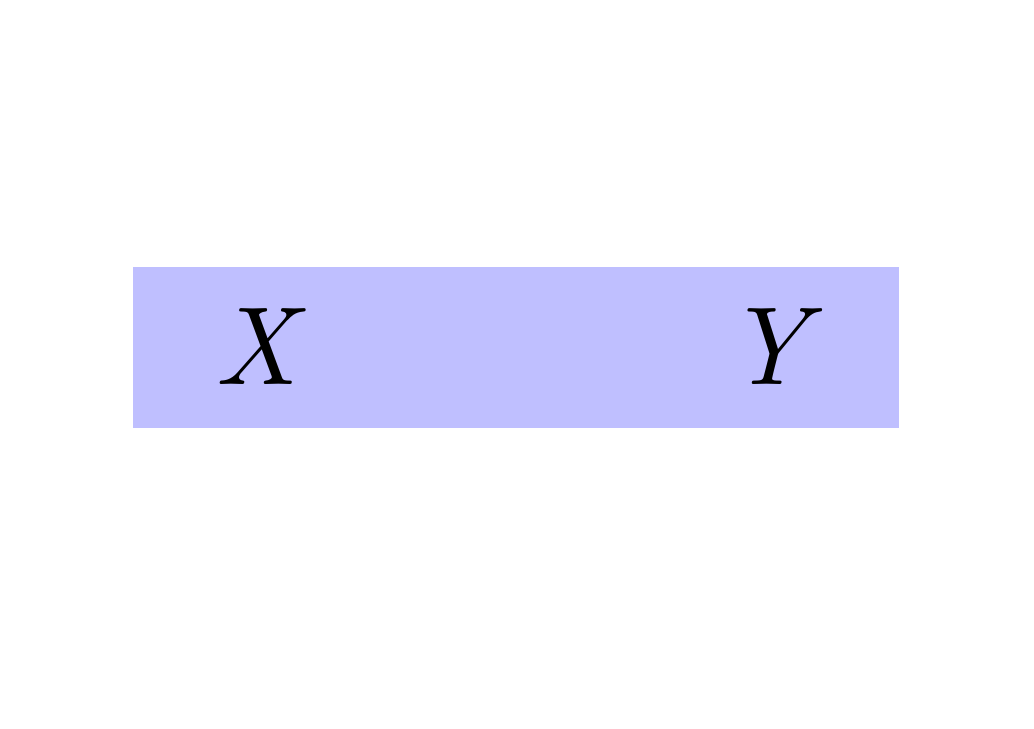}
\end{equation*}
so we can represent it as:
\begin{equation*}
 \includegraphics[align=c,scale=0.35,keepaspectratio=true]{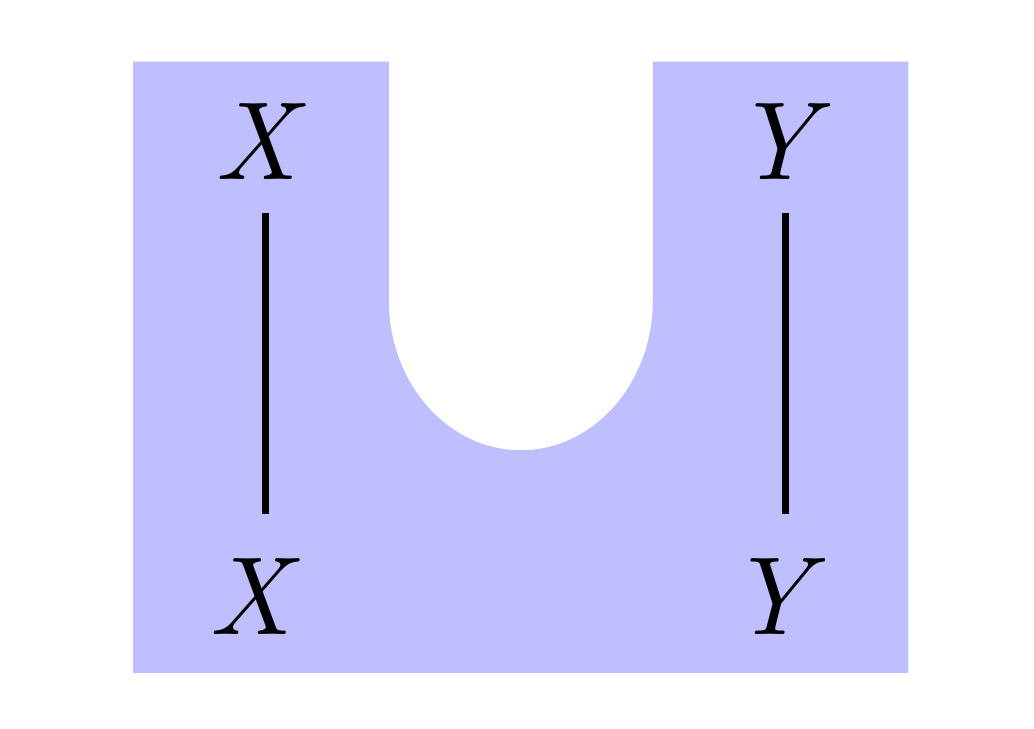}
\end{equation*}
We again do not put any box, as we consider it the canonical map of the form given by the diagram above.
The probabilistic interpretation is the following: given $p\in PX$ and $q\in PY$, there is a canonical (albeit not unique) way of obtaining a joint in $P(X\otimes Y)$, namely the product probability. 
Technically we also should need a map $1\to P(1)\cong 1$, i.e.
\begin{equation*}
 \includegraphics[align=c,scale=0.35,keepaspectratio=true]{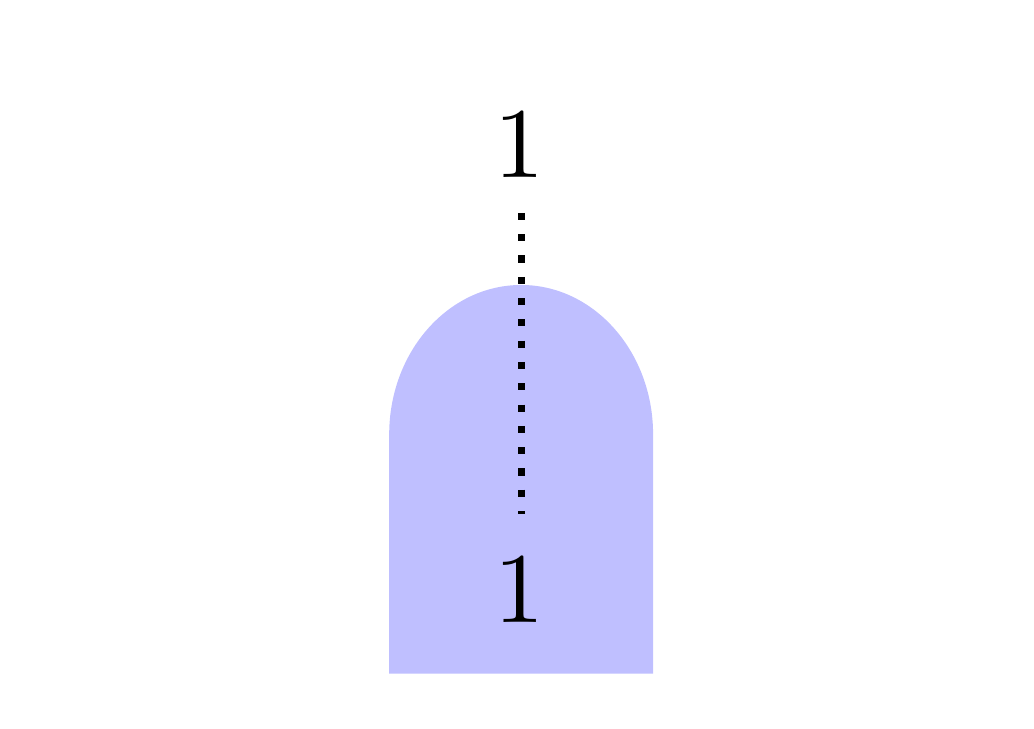}
\quad\mbox{or, omitting the unit, simply}\quad
 \includegraphics[align=c,scale=0.35,keepaspectratio=true]{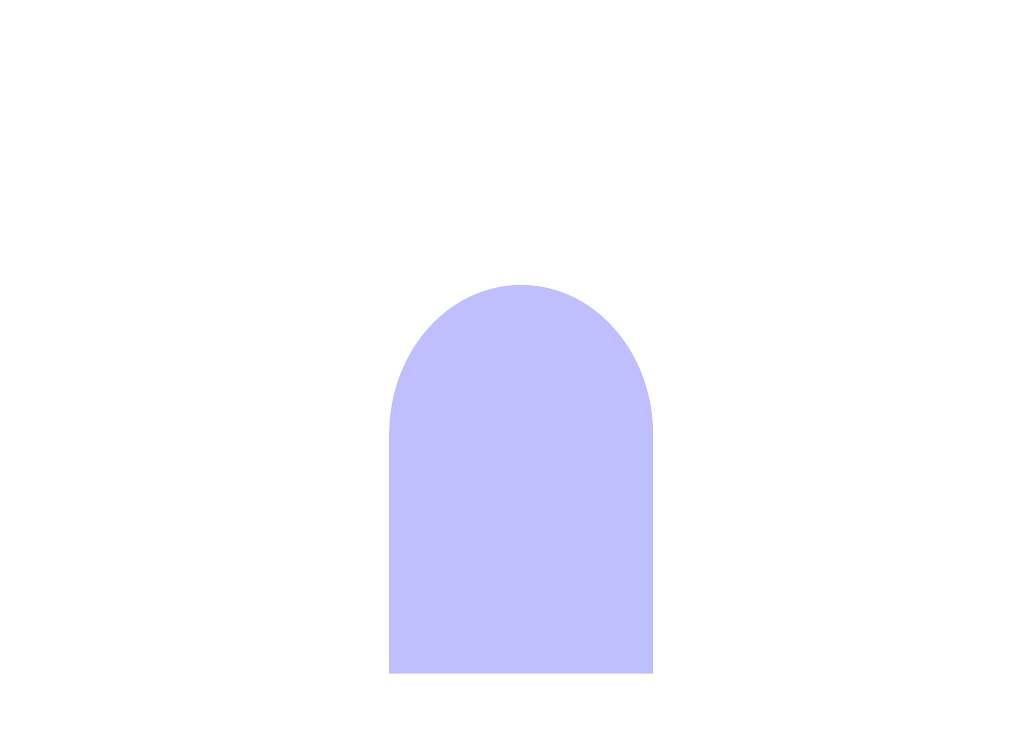}
\end{equation*}
But due to our affineness assumption, such a map can only be the identity.
The associativity condition now says that it does not matter in which way we multiply first:
\begin{equation*}
 \includegraphics[align=c,scale=0.35,keepaspectratio=true]{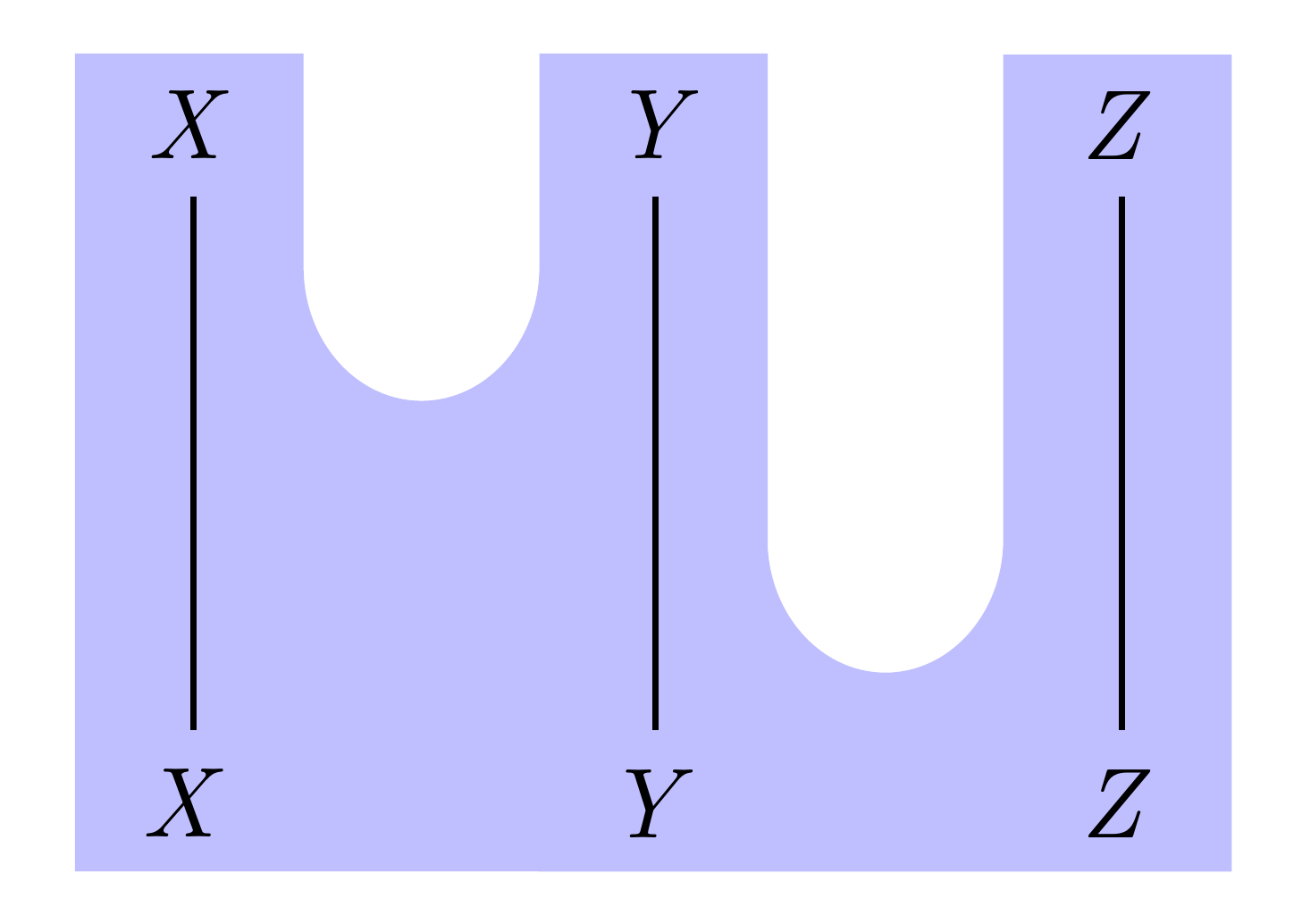}
 = 
 \includegraphics[align=c,scale=0.35,keepaspectratio=true]{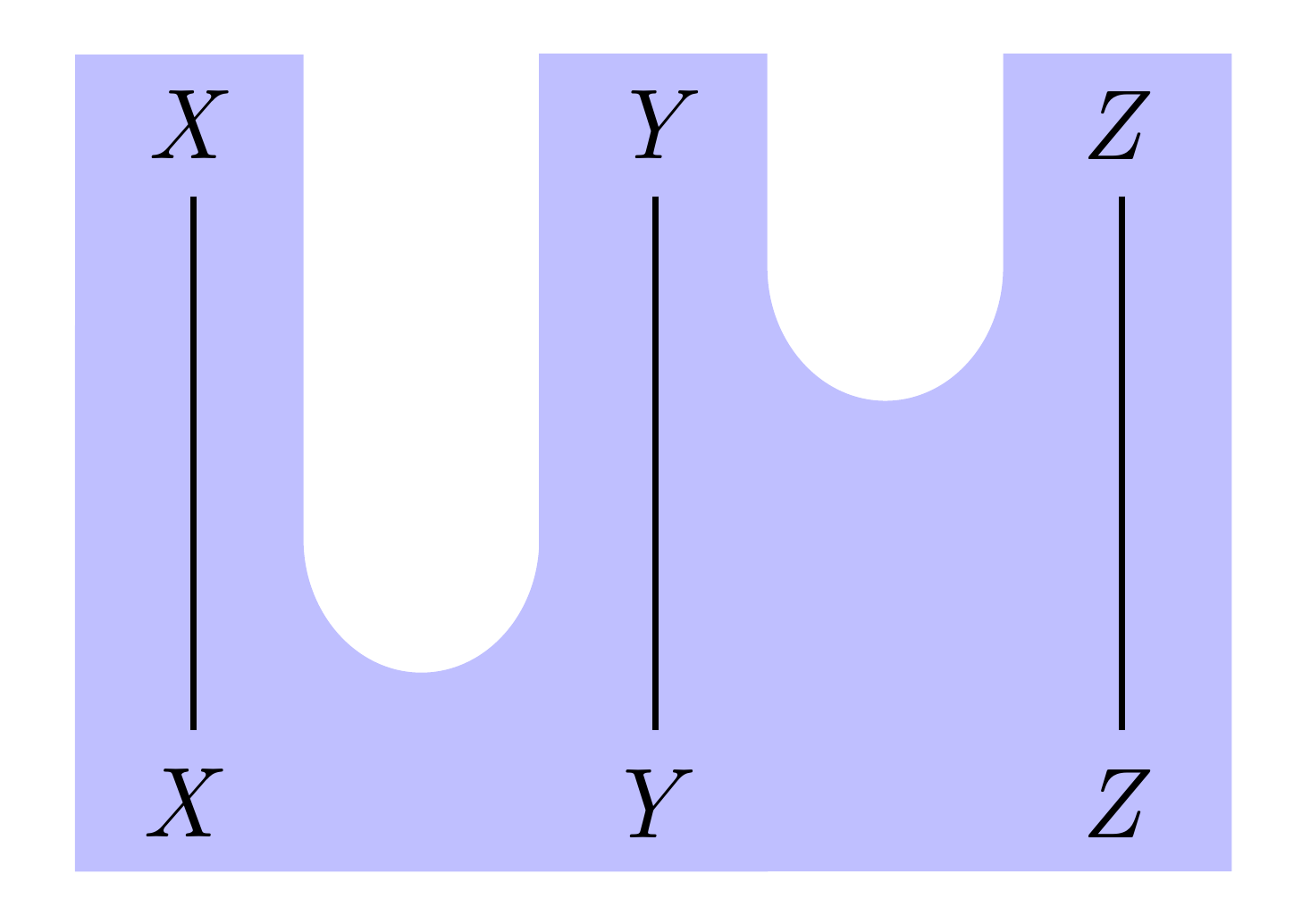}
\end{equation*}
so that there is really just one way of forming a product of three probability distributions. 
The left and right unitality conditions say that:
\begin{equation*}
 \includegraphics[align=c,scale=0.35,keepaspectratio=true]{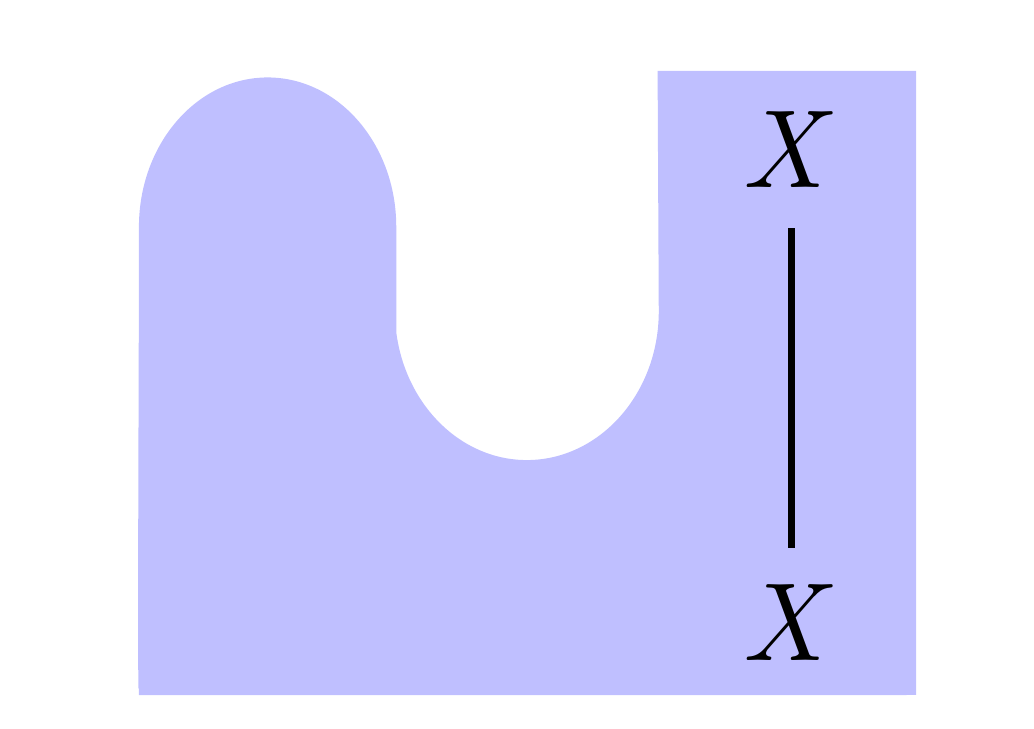}
 =
 \includegraphics[align=c,scale=0.35,keepaspectratio=true,clip=true,trim=10pt 0pt 10pt 0pt]{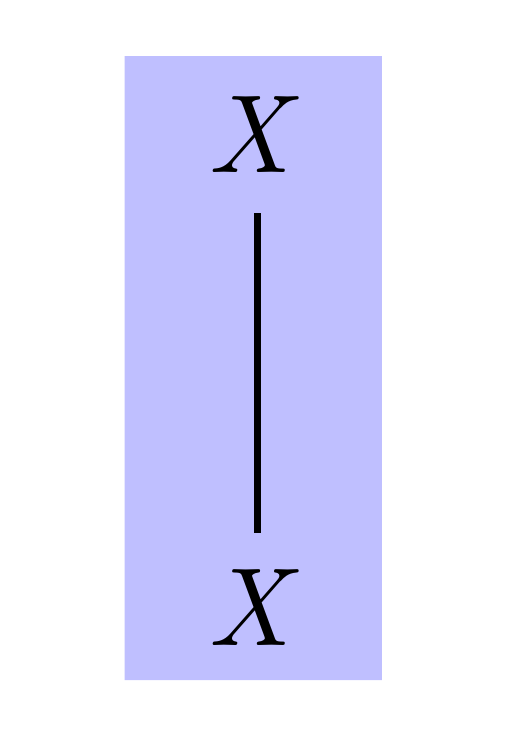}
 =
 \includegraphics[align=c,scale=0.35,keepaspectratio=true]{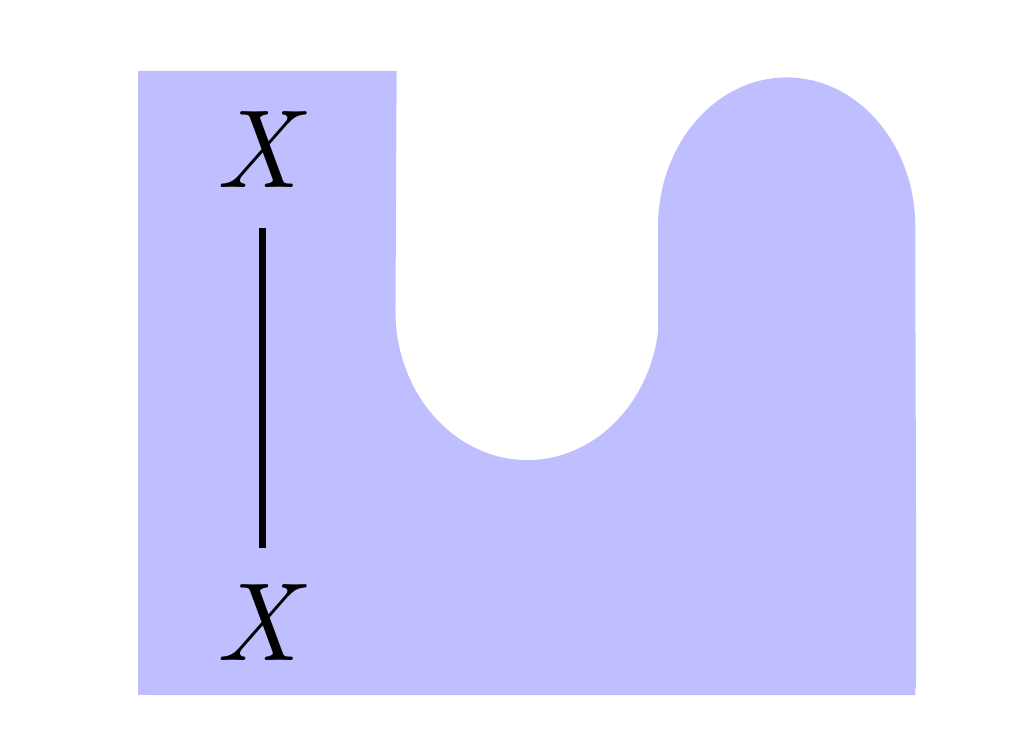}
\end{equation*}
which means that the product distribution of some $p\in PX$ with the unique measure on $1$ is the same as just $p$.

An \emph{opmonoidal structure} for the functor $P$ amounts to a natural map $\Delta:P(X\otimes Y)\to PX\otimes PY$, which we represent as:
\begin{equation*}
 \includegraphics[align=c,scale=0.35,keepaspectratio=true]{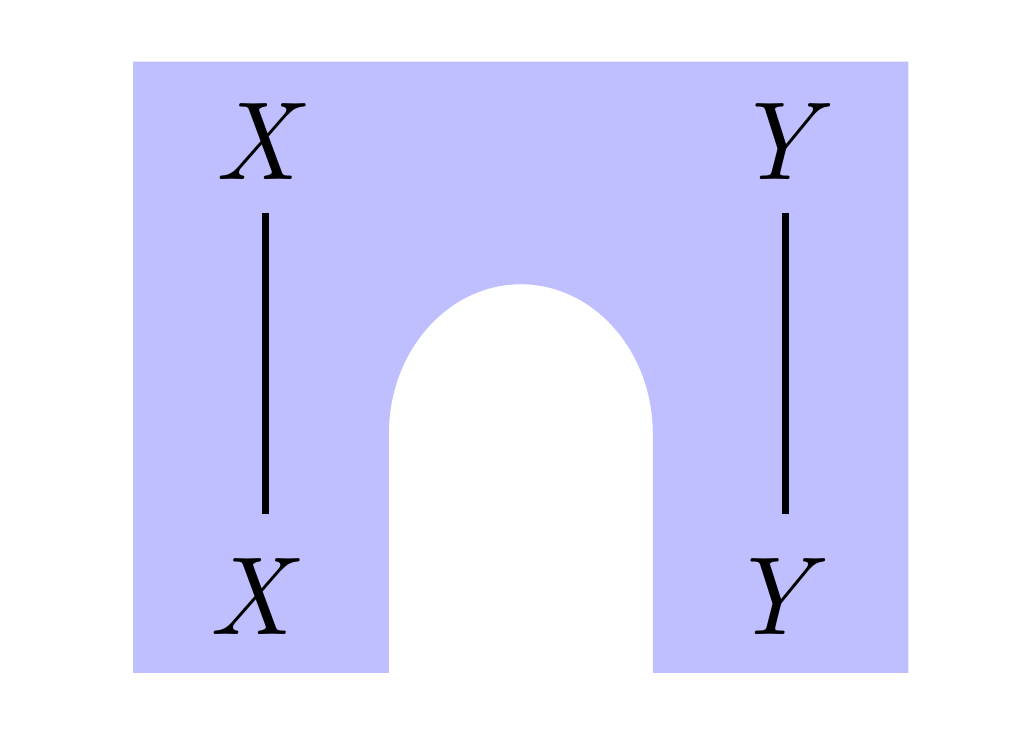}
\end{equation*}
and again a map $P(1)\to 1$, i.e.
\begin{equation*}
 \includegraphics[align=c,scale=0.35,keepaspectratio=true,clip=true,trim=0pt 0pt 0pt 70pt]{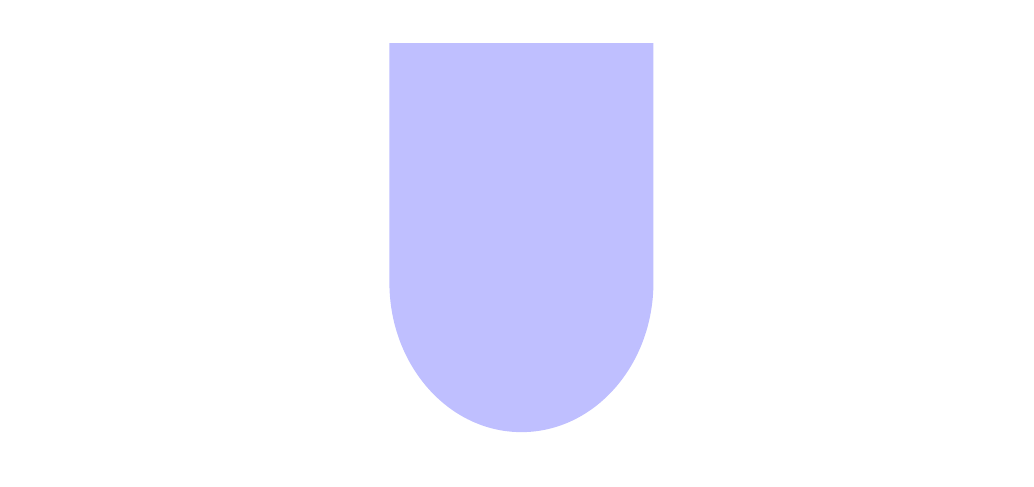}
\end{equation*}
which in this setting can only be the identity.
We have, dually, a coassociativity condition:
\begin{equation*}
 \includegraphics[align=c,scale=0.35,keepaspectratio=true]{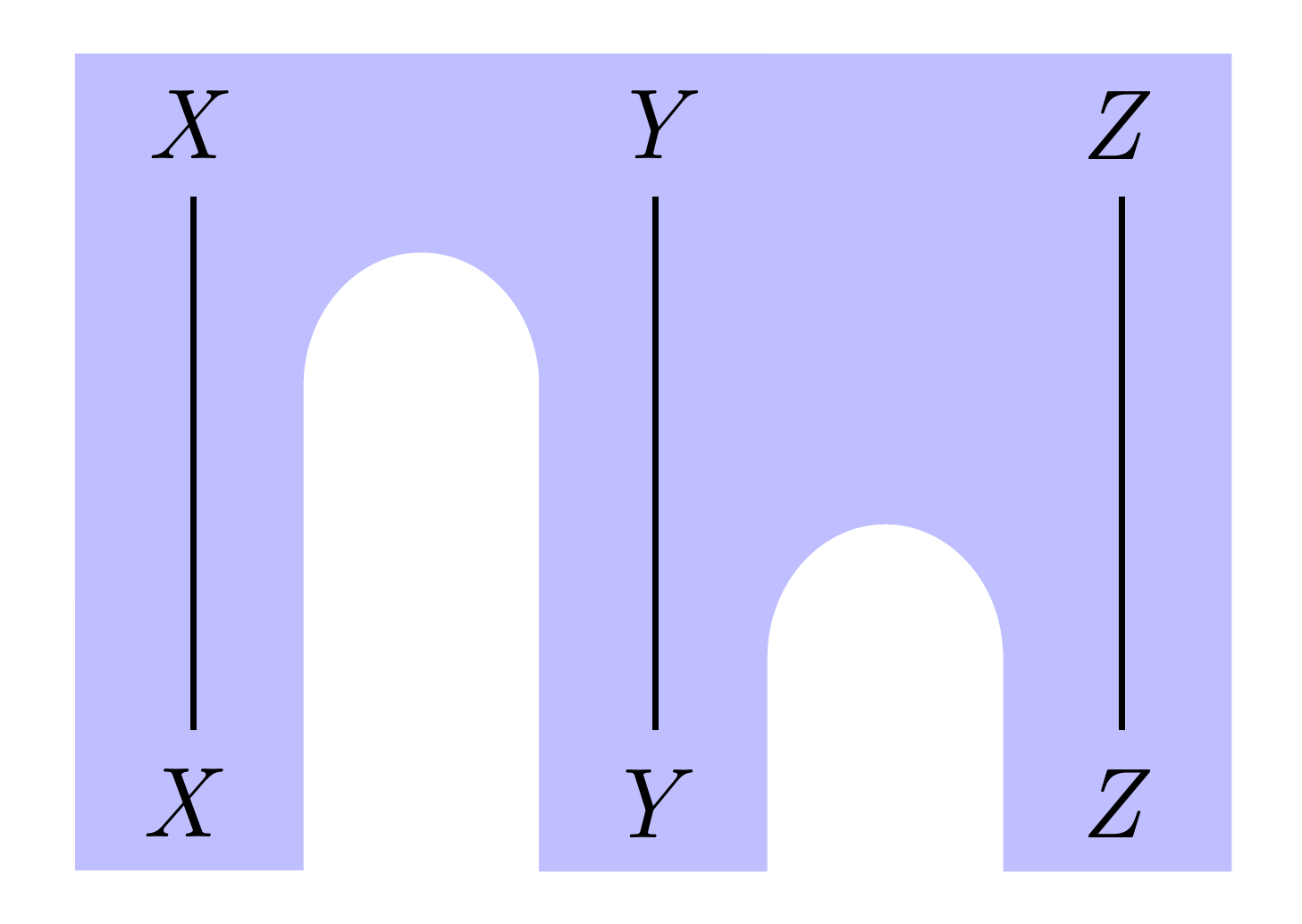}
 =
 \includegraphics[align=c,scale=0.35,keepaspectratio=true]{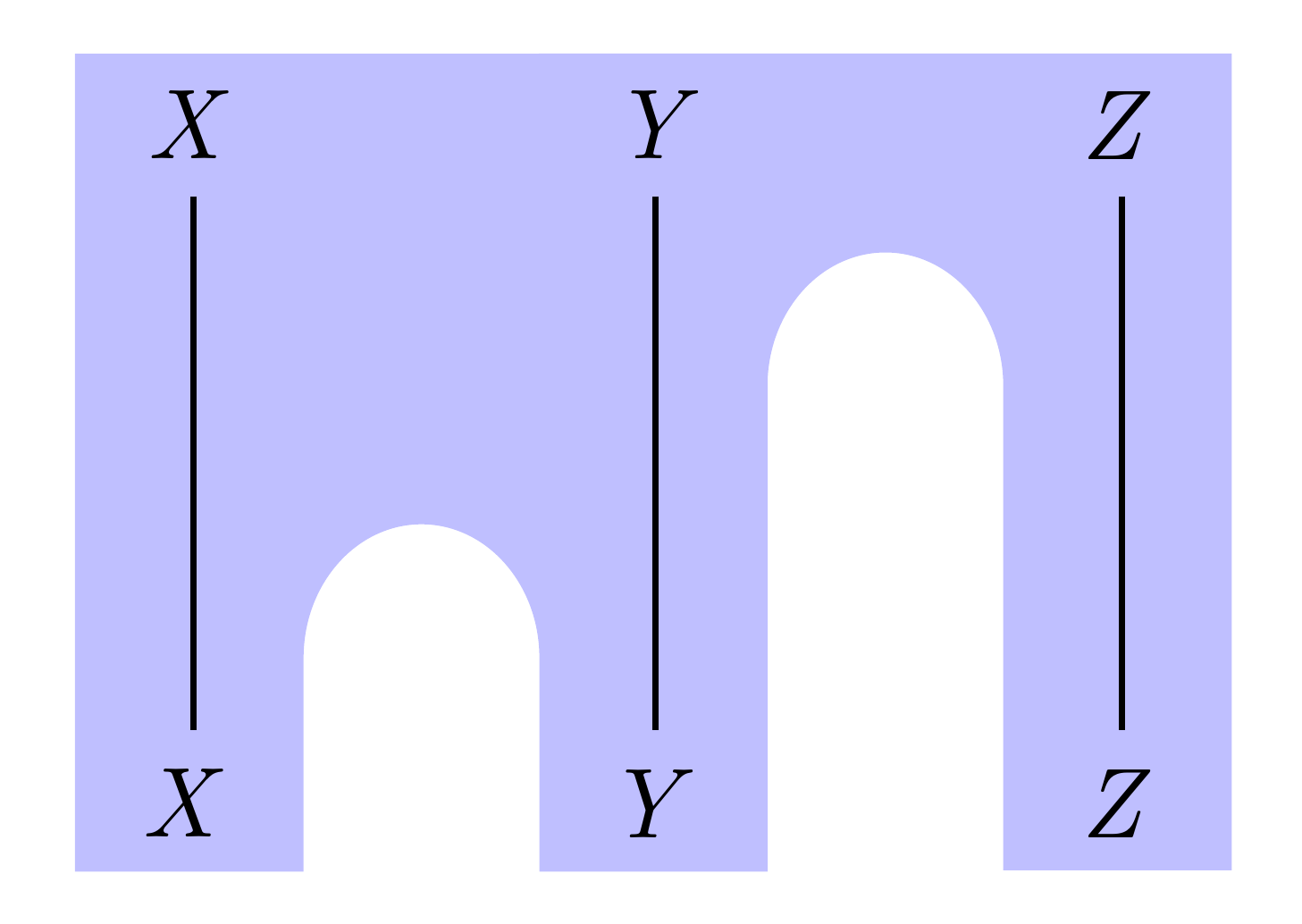}
\end{equation*}
The probabilistic interpretation is that given a joint probability distribution $r\in P(X\otimes Y)$, we can canonically obtain marginal distributions on $PX$ and $PY$, and again, if we have many factors, it does not matter in which order we take the marginals.
Analogously, we have left and right counitality conditions:
\begin{equation*}
 \includegraphics[align=c,scale=0.35,keepaspectratio=true]{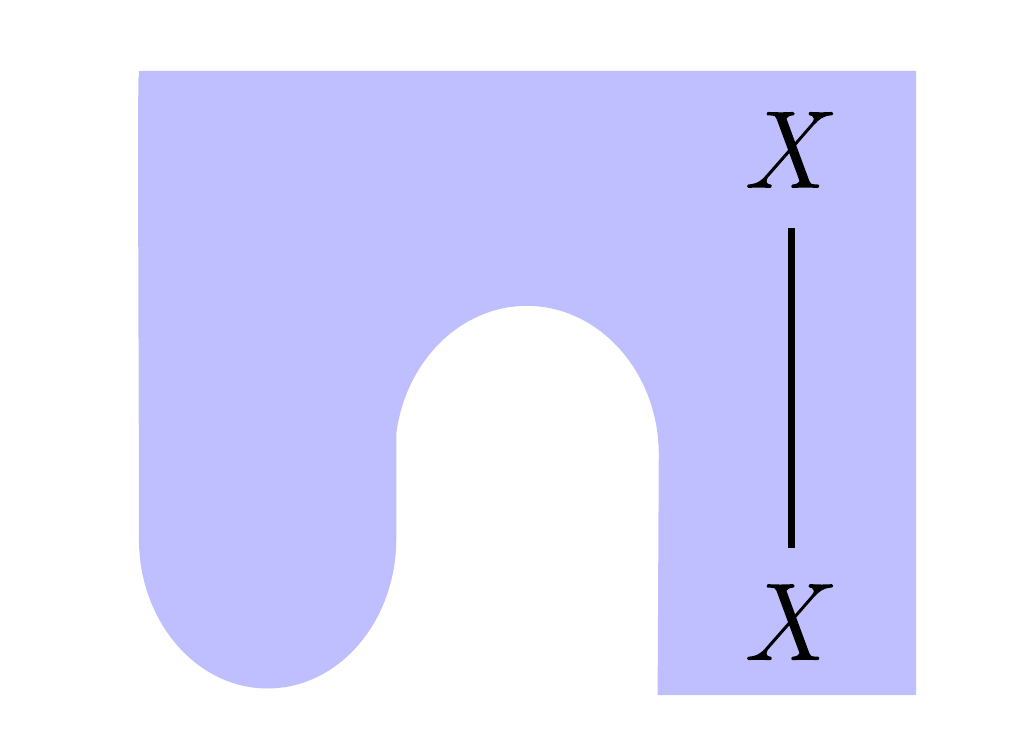}
 =
 \includegraphics[align=c,scale=0.35,keepaspectratio=true,clip=true,trim=10pt 0pt 10pt 0pt]{.//unital1.pdf}
 =
 \includegraphics[align=c,scale=0.35,keepaspectratio=true]{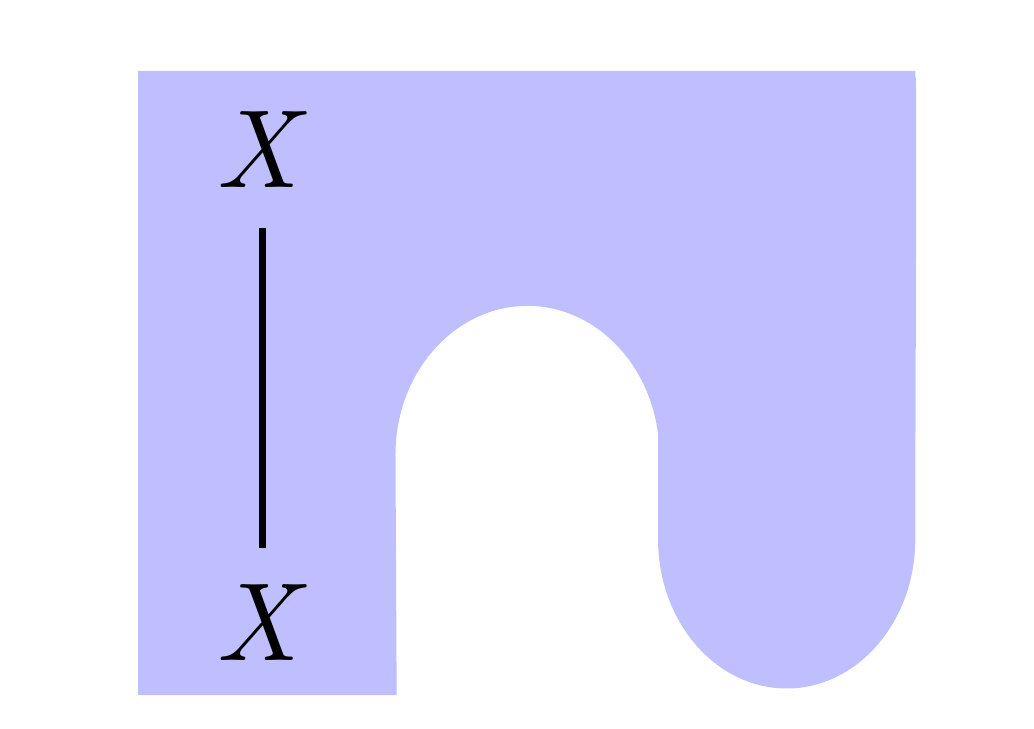}
\end{equation*}
which say that the marginal distribution of some $p\in P(X\otimes 1)$ on the first factor (or of some $p\in P(1\otimes X)$ on the second factor) is just $p$ again.

The monoidal and opmonoidal structure should interact to form a \emph{bimonoidal structure}~\cite{monoidal} for the functor $P$. To have that, we have first of all some unit-counit conditions, which in our setting are trivially satisfied, since they only involve maps to $1$. 
But more importantly, the following bimonoidality (or distributivity) condition needs to hold:
\begin{equation}\label{braidcond}
 \includegraphics[align=c,scale=0.3,keepaspectratio=true]{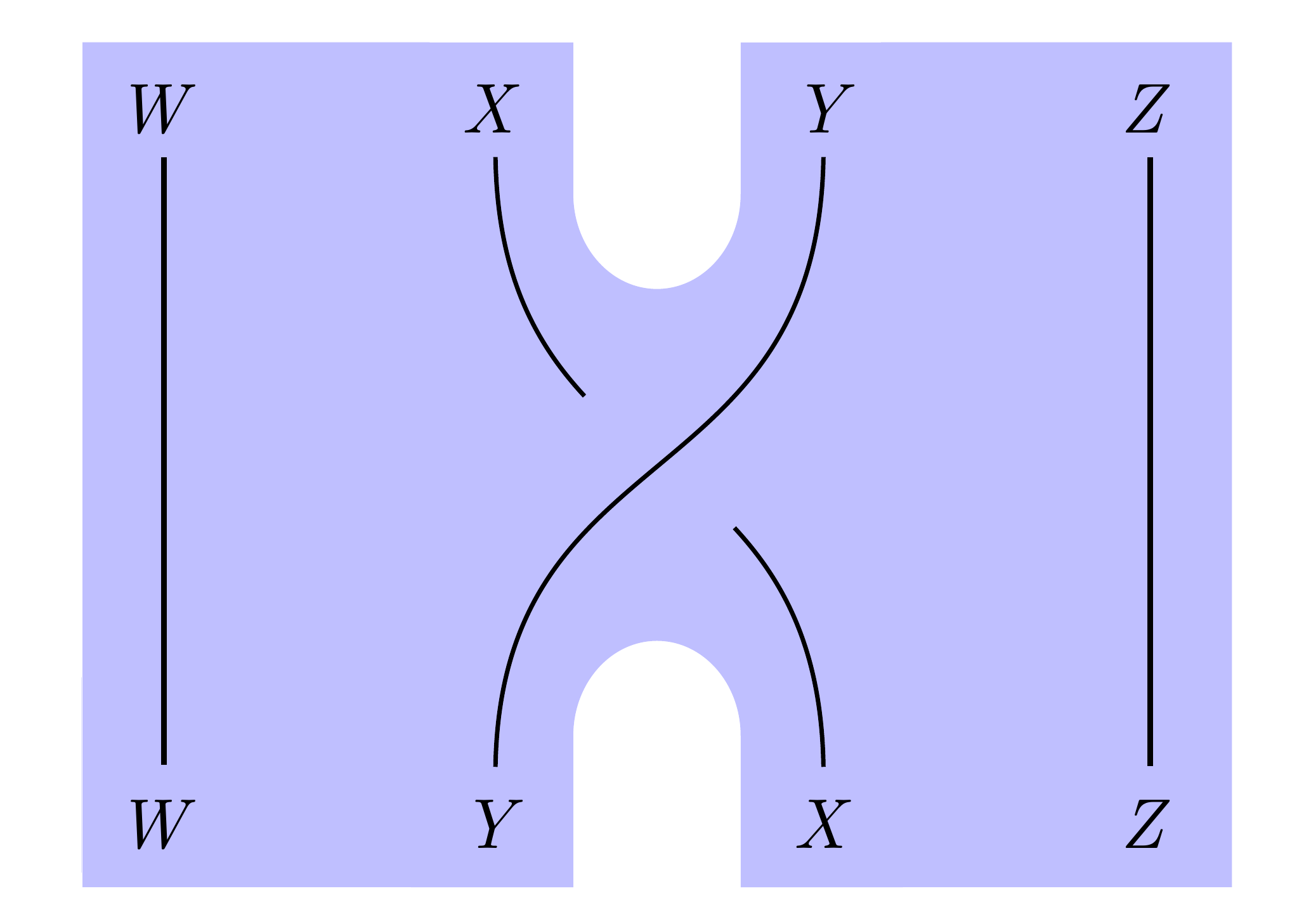}
 =
 \includegraphics[align=c,scale=0.3,keepaspectratio=true]{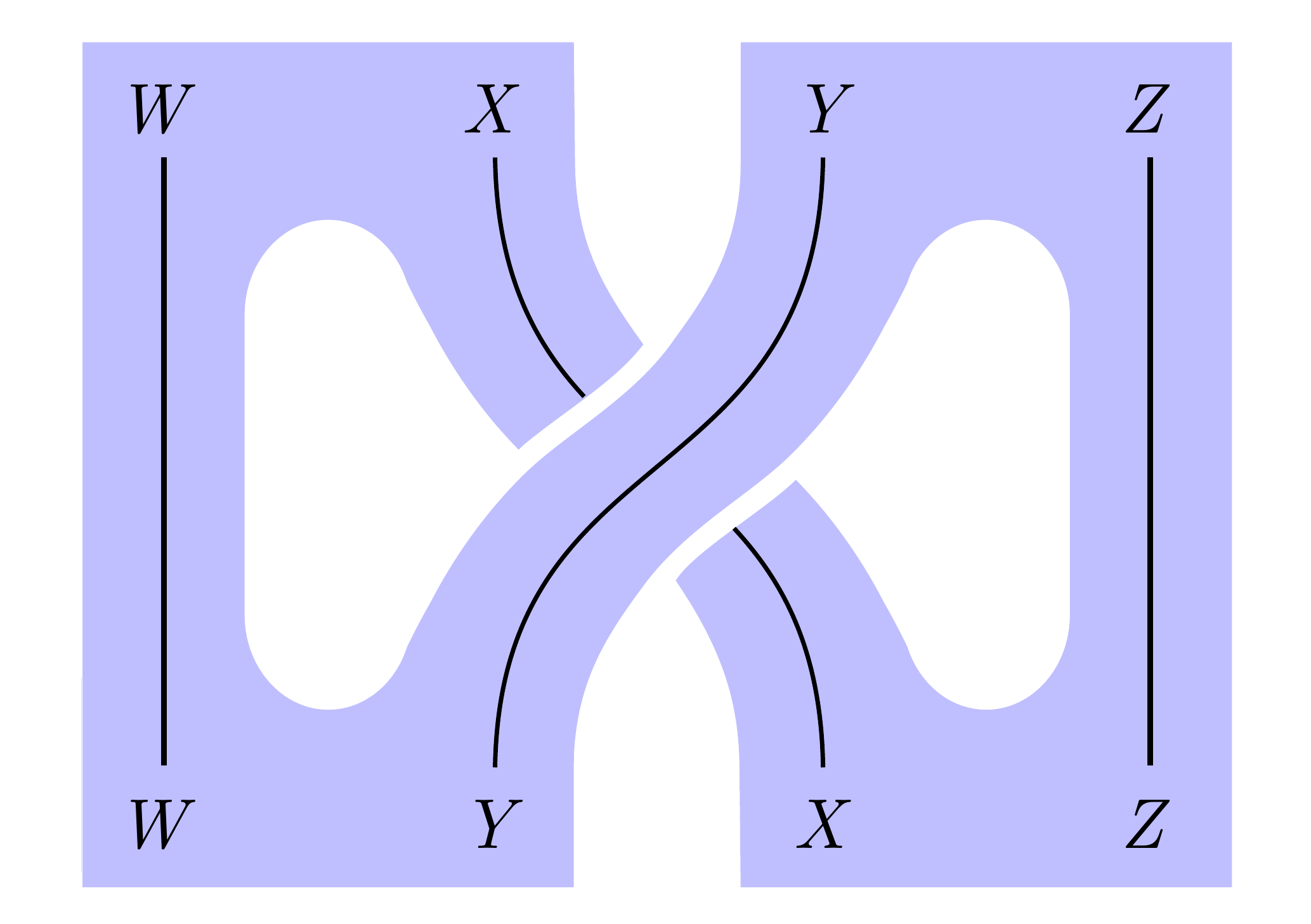}
\end{equation}
where the center of the diagram on the right is a swap of $PX$ and $PY$.
The probabilistic interpretation is a bit involved, and it has to do with stochastic independence. We will analyze it separately in Section \ref{independence}.

We can say even more about the structure of joints and marginals: the whole monad structure should respect the bimonoidal structure of $P$, i.e.\ $\delta:X\to PX$ and $E:PPX\to PX$ commute with taking joint and marginals. In other words, we are saying that $\delta$ and $E$ should be \emph{bimonoidal natural transformations}. 
In terms of diagrams, we are saying that first of all, $\delta$ commutes with the monoidal multiplication and comultiplication:
\begin{equation*}
 \includegraphics[align=c,scale=0.35,keepaspectratio=true]{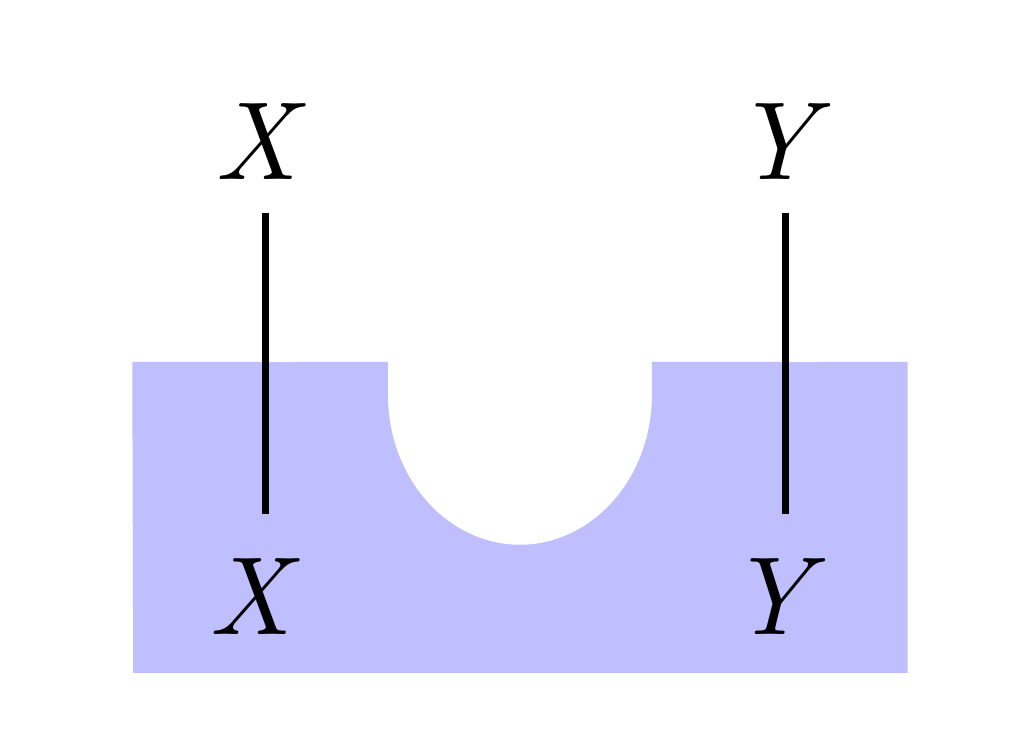}
 = 
 \includegraphics[align=c,scale=0.35,keepaspectratio=true]{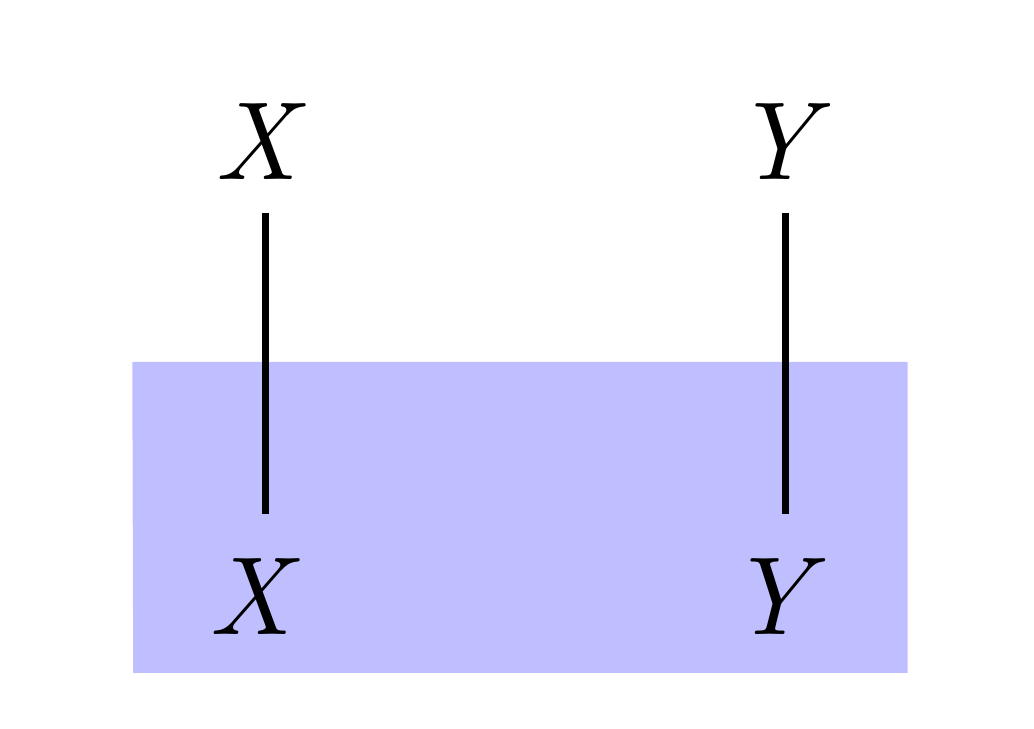}
\end{equation*}
and
\begin{equation*}
 \includegraphics[align=c,scale=0.35,keepaspectratio=true]{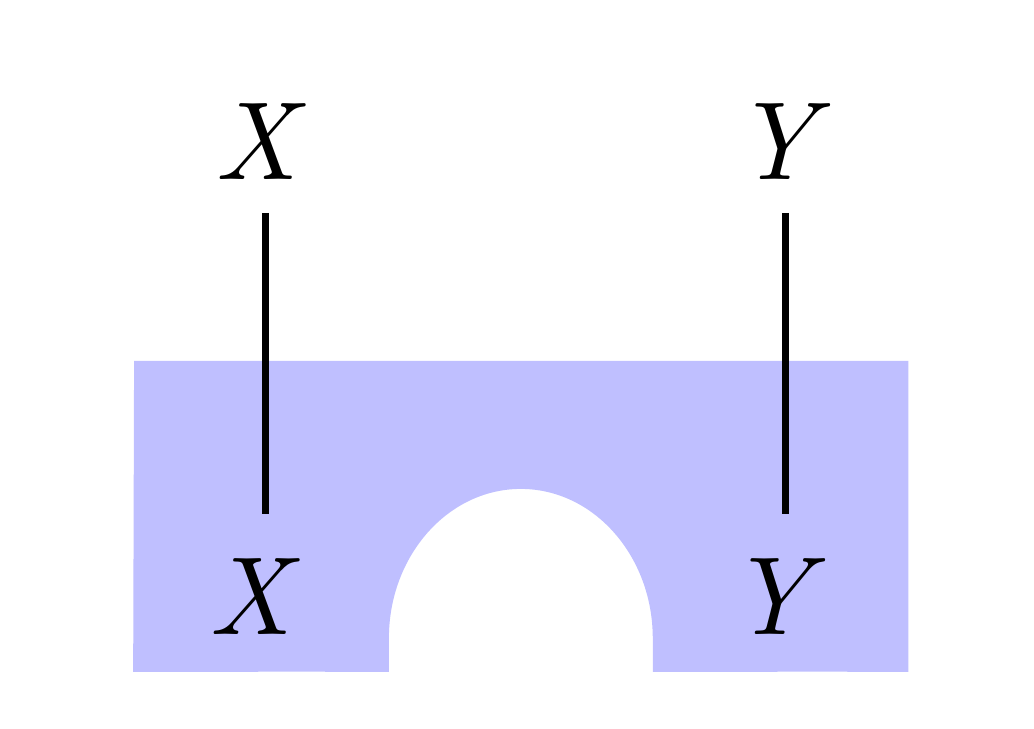}
 =
 \includegraphics[align=c,scale=0.35,keepaspectratio=true]{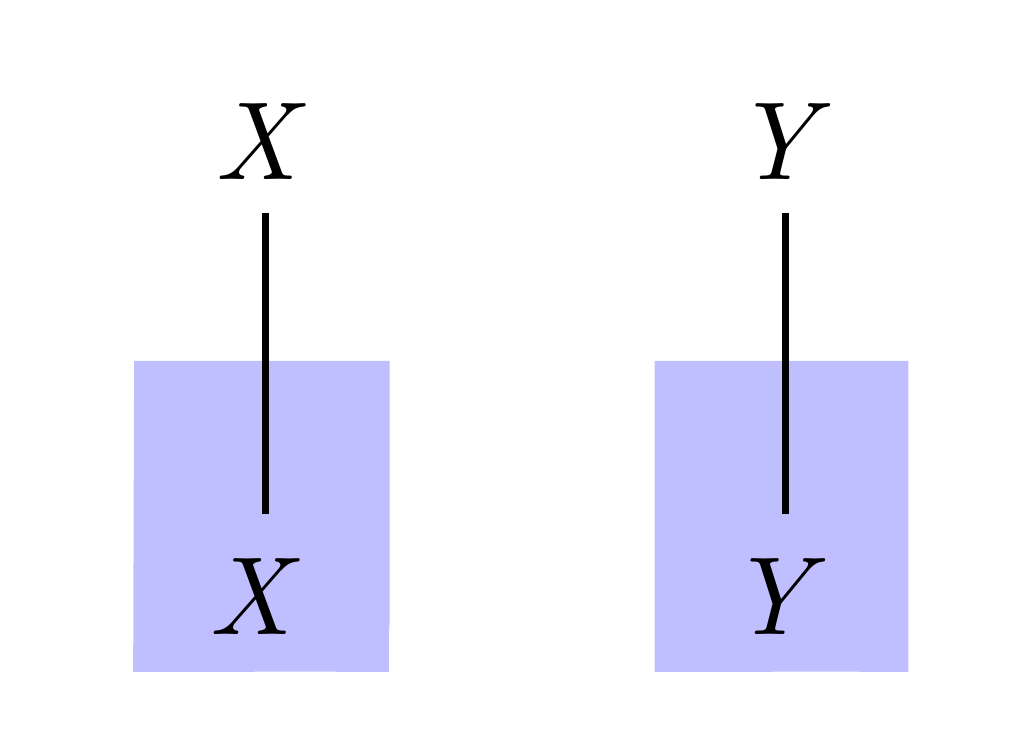}
\end{equation*}
Probabilistically, this means that the delta over the product is the product of the deltas, and that a delta over a product space has as marginals a pair of deltas over the projections.

The same can be said about the average map $E$. It commutes with the multiplication and with the comultiplication:
\begin{equation*}
 \includegraphics[align=c,scale=0.35,keepaspectratio=true]{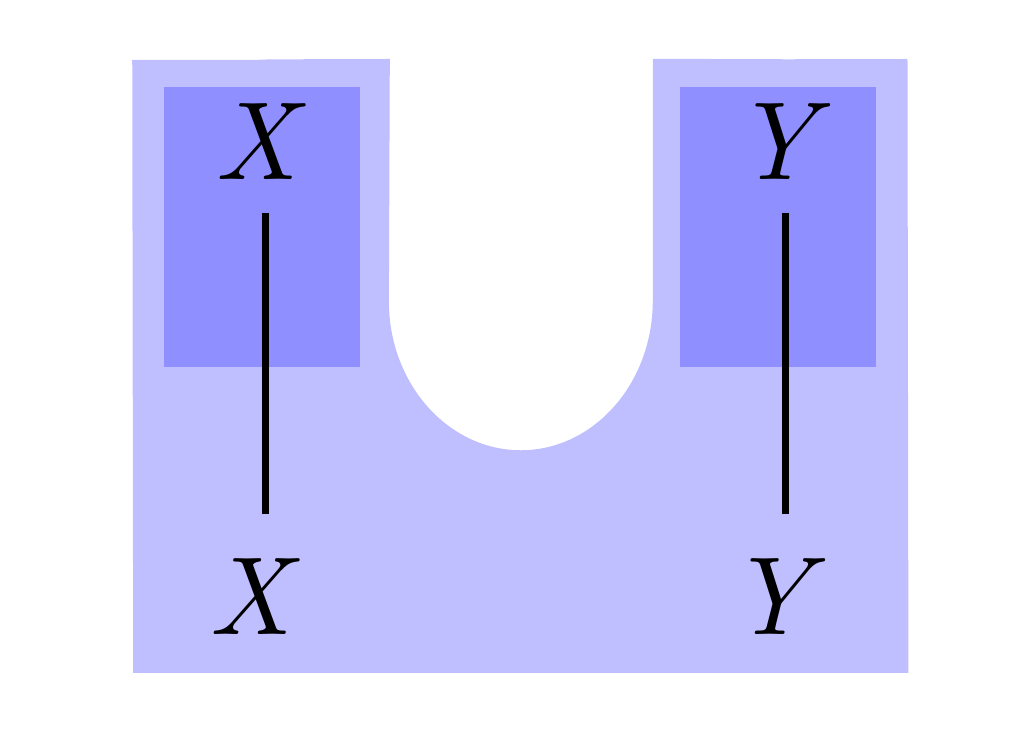}
 =
 \includegraphics[align=c,scale=0.35,keepaspectratio=true]{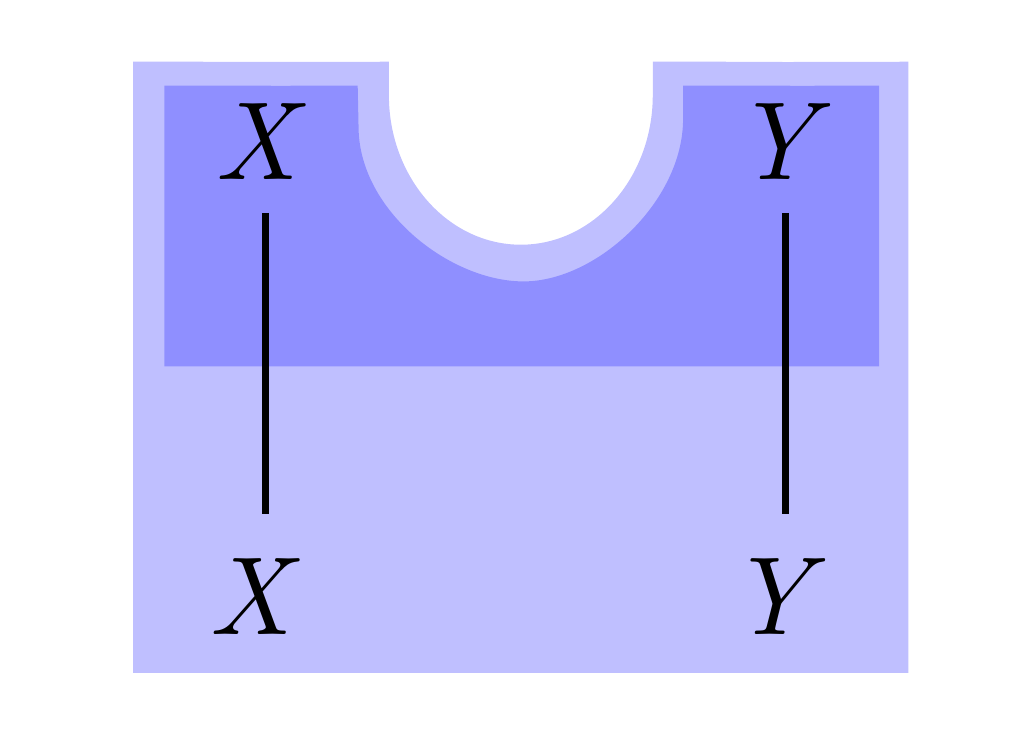}
\end{equation*}
and
\begin{equation*}
 \includegraphics[align=c,scale=0.35,keepaspectratio=true]{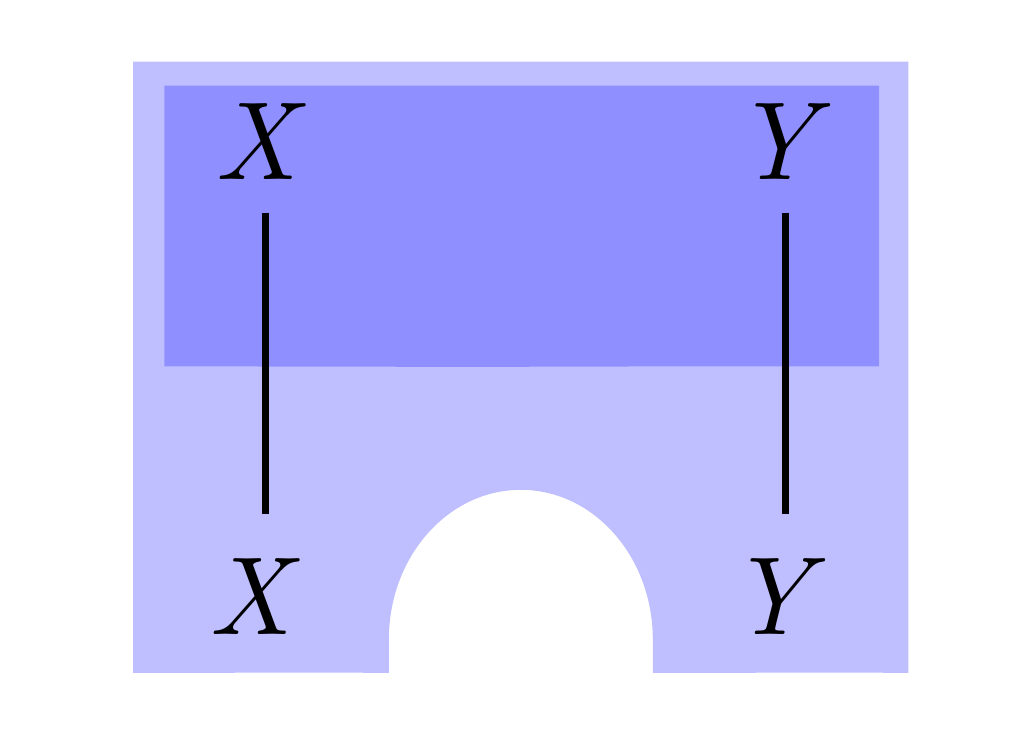}
 = 
 \includegraphics[align=c,scale=0.35,keepaspectratio=true]{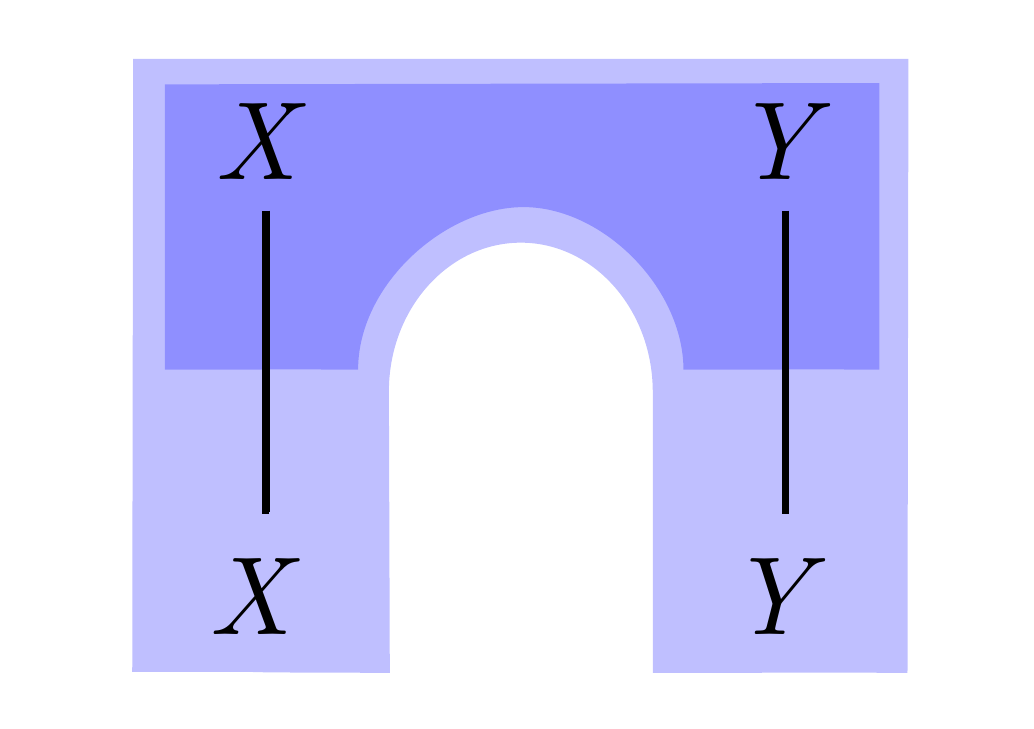}
\end{equation*}
which means that the product of the average is the average of the product, and that the marginals of an average are the averages of the marginals.
These last conditions may seem a bit obscure, but they come up naturally in probability: see as an example the case of the Kantorovich monad (Section \ref{bmp} and its proofs in Appendix \ref{proofs}).

We are in other words requiring that $P$ is a bimonoidal monad.

\begin{deph}
 A \emph{bimonoidal monad} $(P,\delta,E)$ is a monad whose functor is a bimonoidal functor, and whose unit and multiplication are bimonoidal natural transformations.
\end{deph}

The definition above works in general, however the particular conditions for the monoidal and opmonoidal structure which have been given here suffice only in the specific context of a semicartesian monoidal category with an affine monad. In Appendix \ref{monoidalstuff} there is a more general definition, for generic symmetric monoidal categories. The definition given there specializes to the one given above in this context.

As far as we know, this kind of structure has not been considered before in this exact form. Monads in a general bicategory are a standard concept, however to the best of our knowledge the bicategory of monoidal categories, bimonoidal functors, and bimonoidal natural transformations has not been used explicitly. In particular, it has not been used in categorical probability. To avoid possible confusion, let us also point out that the notion of a bimonoidal monad is a distinct concept from that of a \emph{bimonad} \cite{willerton}.

Most probability monads in the literature have an additional symmetry: the multiplication and comultiplication commute with the braiding, i.e.\ they are equivariant with respect to permutations of random variables. This means in diagrams that
\begin{equation*}
 \includegraphics[align=c,scale=0.35,keepaspectratio=true,clip=true,trim=160pt 0pt 160pt 0pt]{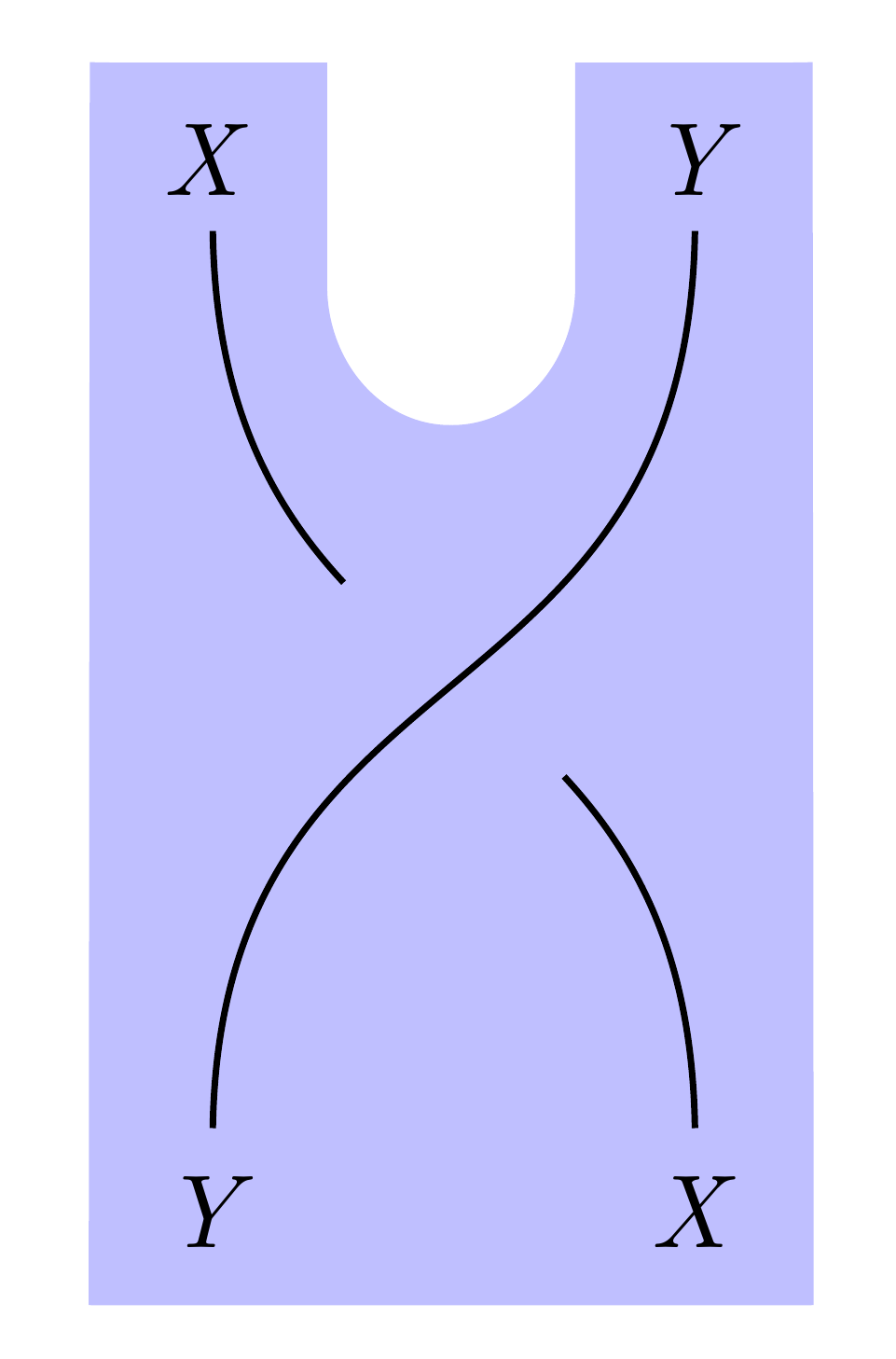}
 =
 \includegraphics[align=c,scale=0.35,keepaspectratio=true,clip=true,trim=160pt 0pt 160pt 0pt]{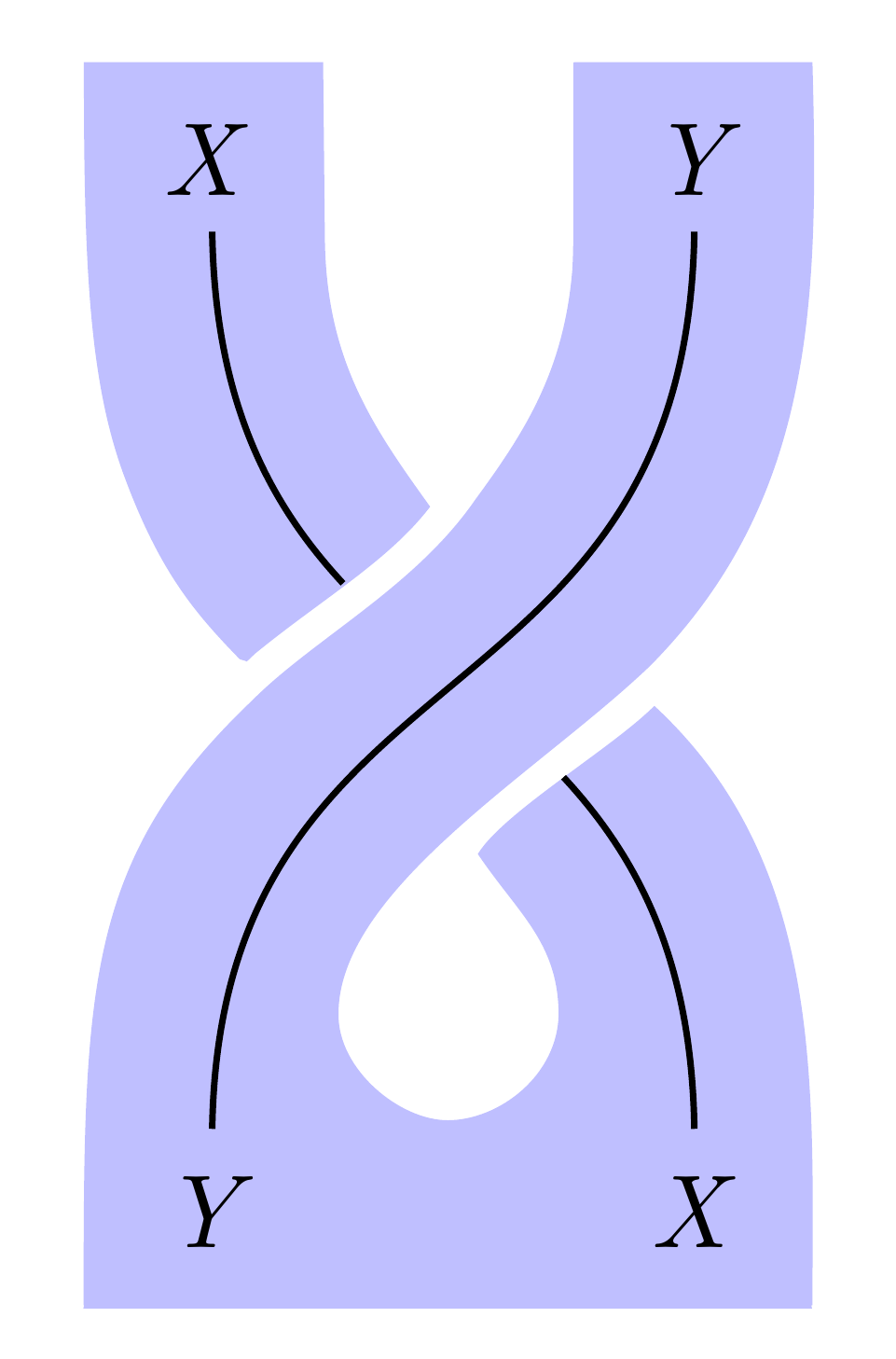}
\end{equation*}
and
\begin{equation*}
 \includegraphics[align=c,scale=0.35,keepaspectratio=true,clip=true,trim=160pt 0pt 160pt 0pt]{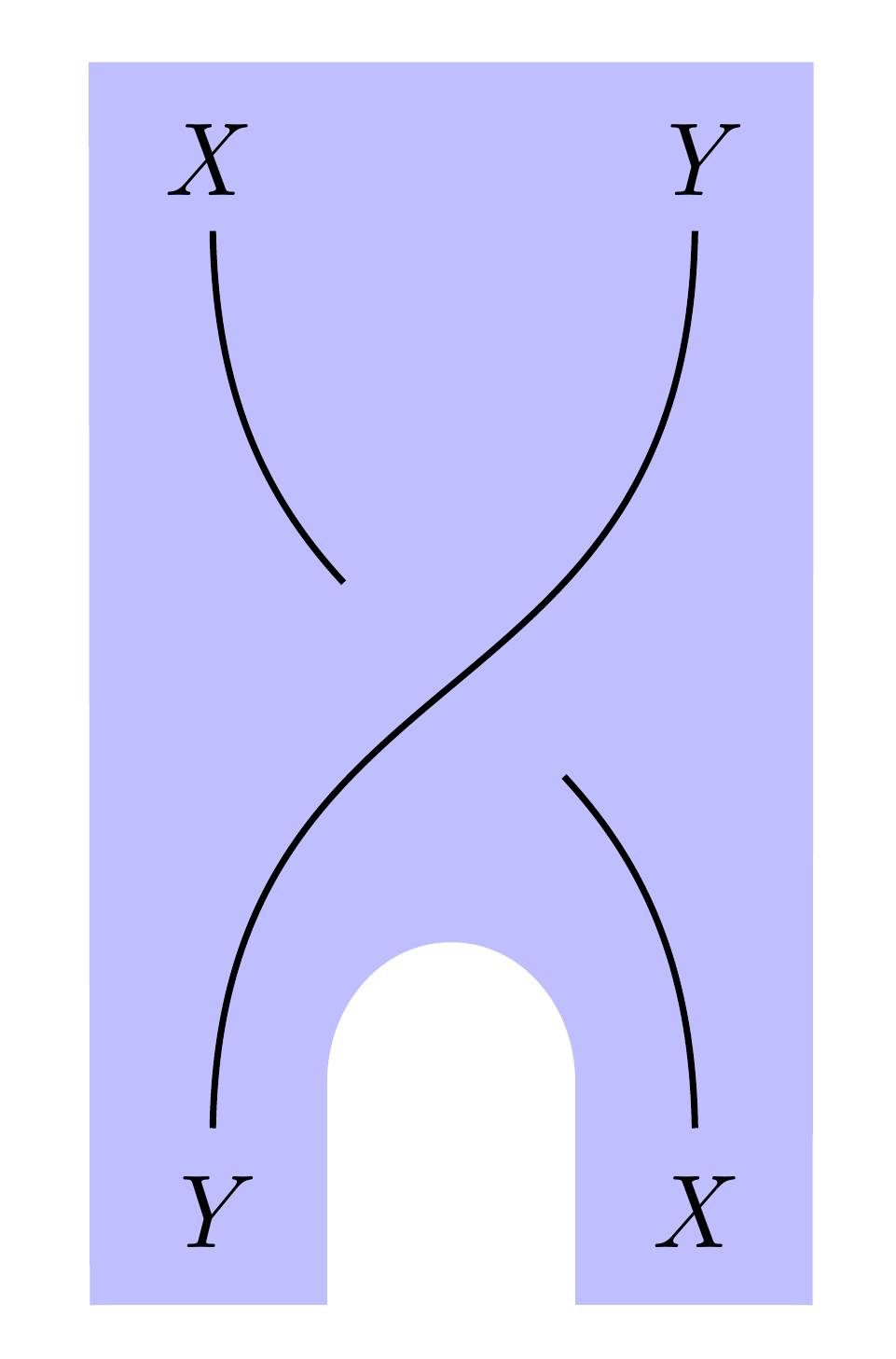}
 = 
 \includegraphics[align=c,scale=0.35,keepaspectratio=true,clip=true,trim=160pt 0pt 160pt 0pt]{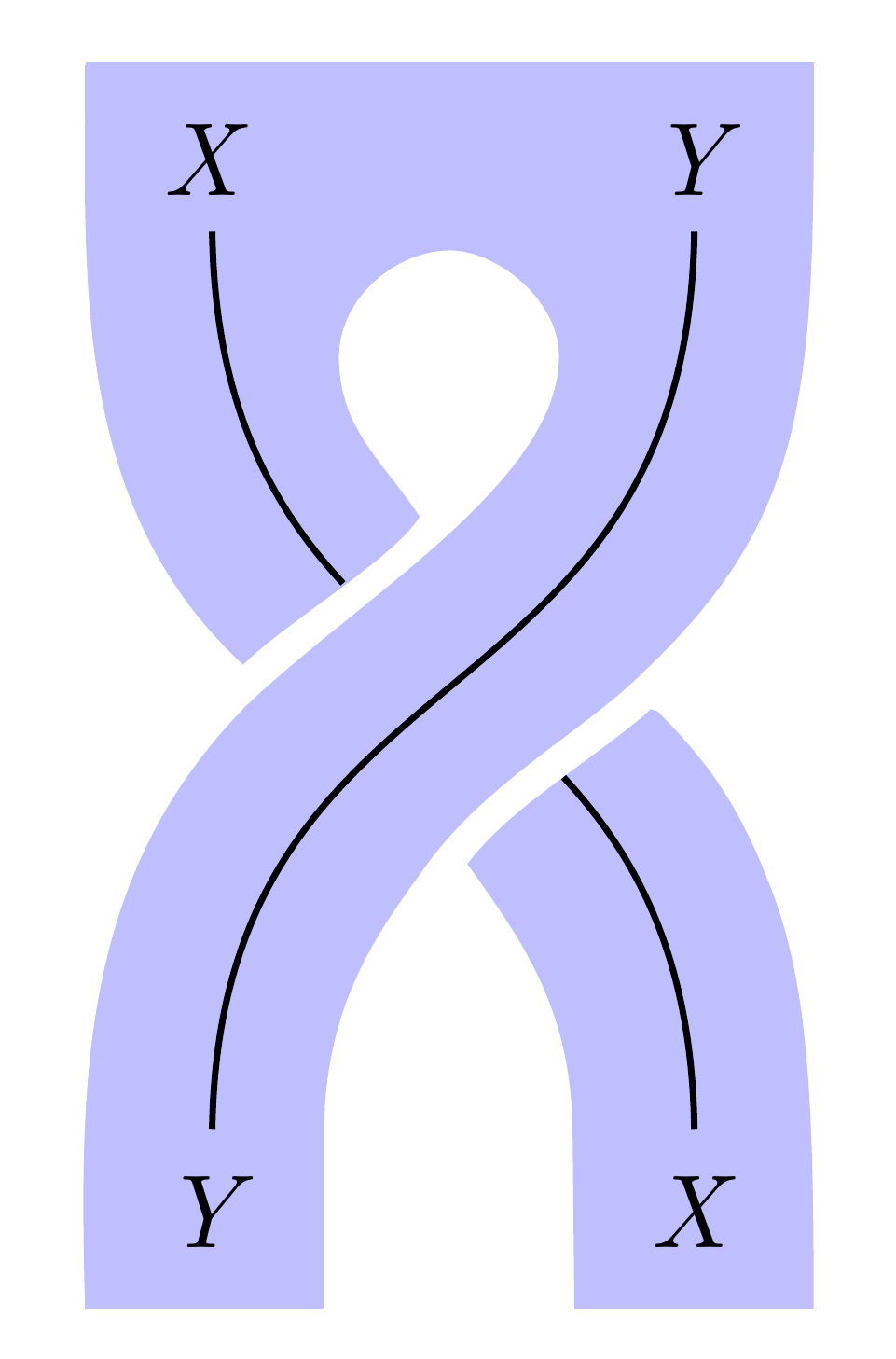}
\end{equation*}
Such a functor (and such a monad) is called \emph{braided} or \emph{symmetric}.
A definition in terms of traditional commutative diagrams can again be found in Appendix \ref{monoidalstuff}.

\subsection{Algebra and coalgebra of random variables}\label{rv}

The so-called ``law of the unconscious statistician'' says that given a function $f:X\to Y$ and a random variable on $X$ with law $p\in PX$, the law of the image random variable under $f$ will be the push-forward of $p$ along $f$. In categorical terms, this simply means that $P$ is a functor, and that the image random variable has law $(Pf)(p)$, where $Pf:PX\to PY$ is given by the push-forward. 

The bimonoidal structure of $P$ comes into play whenever we have functions to and from product spaces. Consider a morphism $f:X\otimes Y \to Z$, which we represent as:
\begin{equation*}
 \includegraphics[align=c,scale=0.35,keepaspectratio=true]{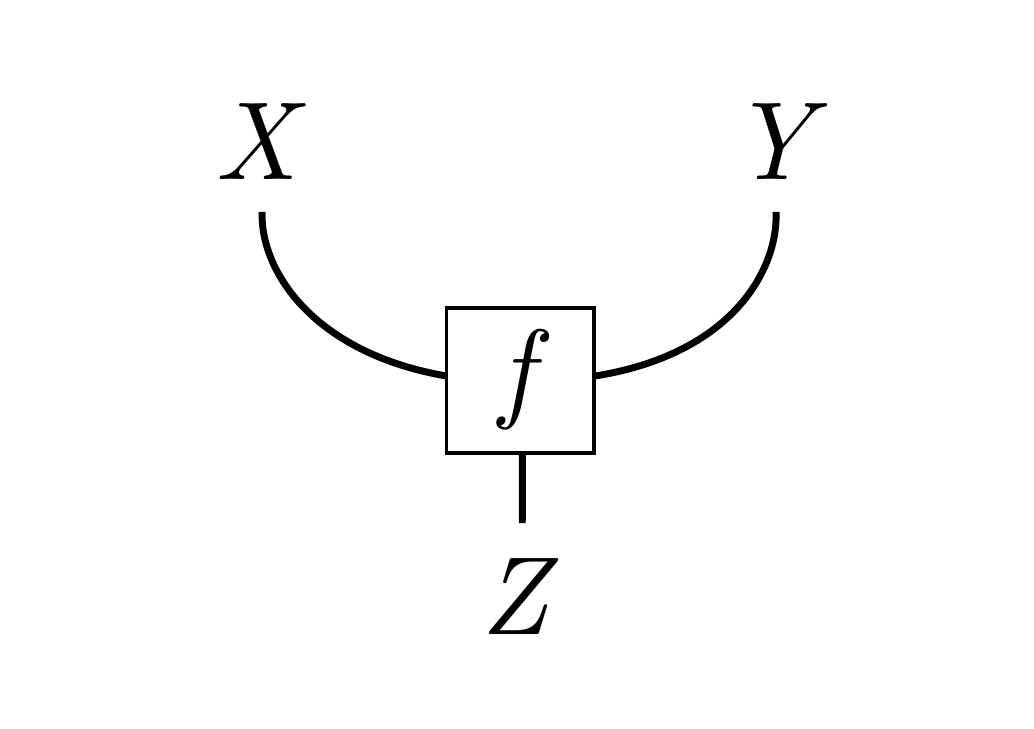}
\end{equation*}
Given random variables $X$ and $Y$, we can form an image random variable on $Z$ in the following way: first we form the joint on $X\otimes Y$ using the monoidal structure, and then we form the image under $f$. In other words, in terms of laws we perform the following composition:
\begin{equation*}
 \includegraphics[align=c,scale=0.35,keepaspectratio=true]{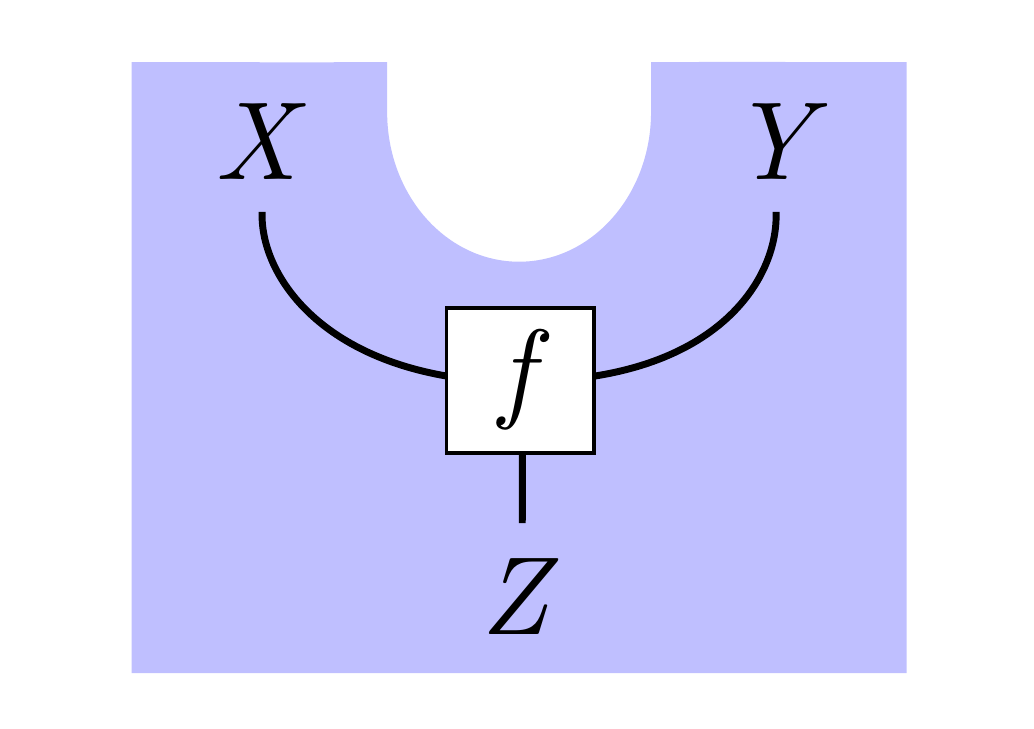}
\end{equation*}

 
For maps in the form $g:X \to Y \otimes Z$ we can proceed analogously by forming the marginals, using the opmonoidal structure:
\begin{equation*}
 \includegraphics[align=c,scale=0.35,keepaspectratio=true]{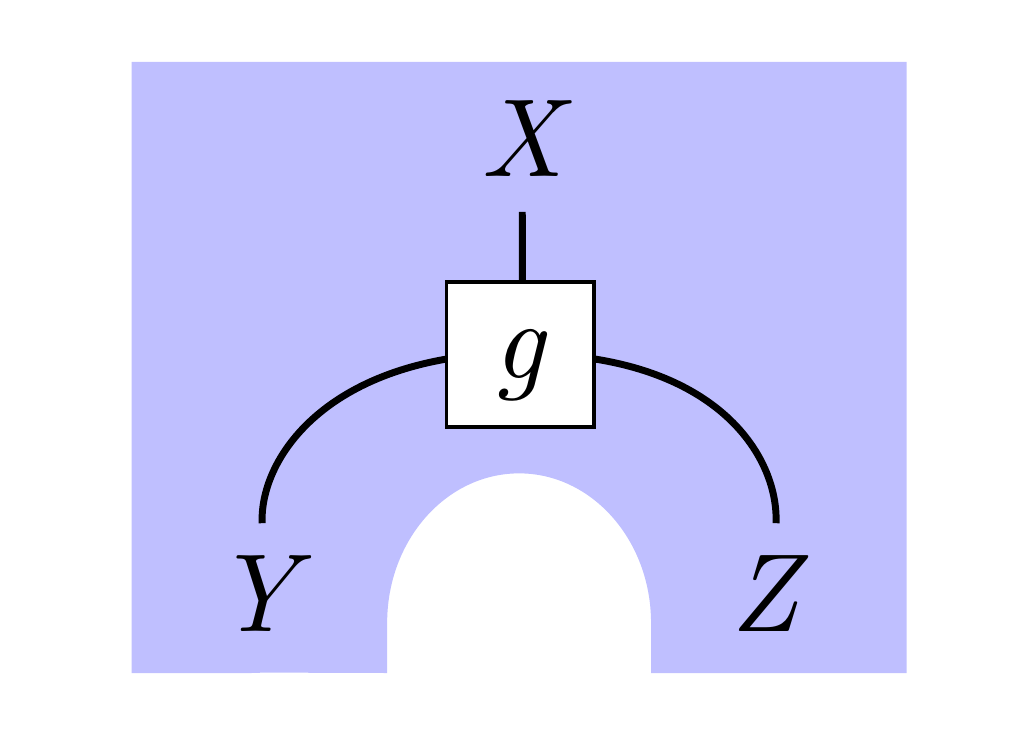}
\end{equation*}
 
This way, together with associativity and coassociativity, one can form functions to and from arbitrary products of random variables. 
 

Whenever we have an internal structure, like an internal monoid or group, this way we can extend the operations on the random elements, via convolution. For example, if $X$ is a monoid, then also $PX$ becomes a monoid, using $PX\otimes PX\to P(X\otimes X)\to PX$ for the multiplication. The analogous statements apply for coalgebraic structures. In other words, the bimonoidal structure allows to have an \emph{algebra and coalgebra of random variables} whenever the deterministic variables form an internal algebraic structure. 
For a concrete example, if as monoid we take the real line with addition, as convolution algebra we get the usual convolution of probability measures. We notice that such a convolution algebra is a monoid (with the neutral element given by the Dirac delta at zero), but \emph{not} a group: only the \emph{monoid} structure is inherited, in general.

\subsection{The category of random elements}\label{probc}

In the literature, many categorical treatments of statistical dependence work in categories whose objects are probability spaces, or fixed probability measures on a space, rather than categories with a probability monad \cite{franz,simpson-independence}. One can form probability spaces from a probability monad in a canonical way:

\begin{deph}
 Let $\cat{C}$ be a category with terminal object $1$ and $P$ a probability monad on $\cat{C}$. Then the category $\cat{Prob(C)}$ is defined to be the co-slice category $1/P$. In other words:
 \begin{itemize}
  \item Objects of $\cat{Prob(C)}$ are objects $X$ of $\cat{C}$ together with arrows $1\to PX$ of $\cat{C}$;
  \item Morphisms of $\cat{Prob(C)}$ are maps $f:X\to Y$ of $\cat{C}$ which make the diagram
  \begin{equation*}
   \begin{tikzcd}[column sep=tiny]
    & 1 \ar{dl} \ar{dr} \\
    PX \ar{rr}{Pf} && PY 
   \end{tikzcd}
  \end{equation*}
  commute.
 \end{itemize}
\end{deph}

In analogy with the category of elements, we can interpret $\cat{Prob(C)}$ as a \emph{category of random elements}, or of probability spaces. The objects can be interpreted as elements of $PX$, i.e.~probability measures on $X$, and the morphisms can be interpreted as maps preserving the selected element in the space of measures, i.e.~measure-preserving maps.

Under some mild assumptions, if $\cat{C}$ has a semicartesian monoidal structure we can transfer that structure to the category of random elements, with a construction analogous to that of Section \ref{rv}.

\begin{deph}\label{tensprobc}
 Let $\cat{C}$ be a semicartesian monoidal category and $P$ an affine probability monad on $\cat{C}$ with monoidal structure $\nabla$. We define the following monoidal structure on $\cat{Prob(C)}$: given $p:1\to PX$ and $q:1\to PY$, we define $p\otimes_\nabla q:1\to P(X\otimes Y)$ to be the composition:
 \begin{equation*}
  \begin{tikzcd}
   1\cong 1\otimes 1 \ar{r}{p\otimes q} & PX \otimes PY \ar{r}{\nabla} & P(X\otimes Y) .
  \end{tikzcd}
 \end{equation*}
 and for morphisms we proceed analogously. 
\end{deph}

This way $(\cat{Prob(C)},\otimes_\nabla)$ is a semicartesian monoidal category, with the unit $1\to 1$ isomorphic to the terminal object. In particular, it is always a tensor category with projections in the sense of \cite{franz}, generalizing the construction given in Section 3.1 therein (in which the base category $\cat{Meas}$ is cartesian monoidal).
In general (and in all interesting cases in the literature), $\cat{Prob(C)}$ equipped with this monoidal structure is \emph{not cartesian monoidal}, not even if $\cat{C}$ is: the product probability does not satisfy the universal property of a categorical product (see for example~\cite{franz} for a discussion on this).\footnote{The intuitive idea that ``the product probability has the same information as the pair of marginals'' can be made rigorous in a different manner, see Section~\ref{independence}.}

Some of the upcoming results will refer to $\cat{Prob(C)}$, whose objects we also call \emph{laws}, as they generalize laws of random variables. In particular we will use the notation $p\otimes_\nabla q$ for the product probability.

\subsection{Bimonoidal monads on a cartesian monoidal category}\label{strength}

Suppose now that the monoidal structure of $\cat{C}$ is \emph{cartesian} monoidal, i.e.\ that the monoidal product is given by the categorical product (so, in particular, $\cat{C}$ is semicartesian). The projection maps $\pi_1:X\times Y \to X$ and $\pi_2:X\times Y \to Y$ now satisfy a universal property. Let's now apply $P$, so that we get maps $P\pi_1:P(X\times Y) \to PX$ and $P\pi_2:P(X\times Y) \to PY$. By the universal property of the product, there is then a \emph{unique} map $P(X\times Y)\to PX\times PY$ compatible with the projections, i.e.\ making the following diagram commute:
\begin{equation*}
 \begin{tikzcd}
  & P(X\times Y) \ar{dl}[swap]{P\pi_1} \ar{dr}{P\pi_2} \uni{d} \\
  PX & PX \times PY \ar{l}{\pi_1} \ar{r}[swap]{\pi_2} & PY
 \end{tikzcd}
\end{equation*}
This gives a natural map $\Delta:P(X\times Y)\to PX\times PY$. Such a map exists and is unique for any (finite) number of factors, so it is automatically associative. Therefore $P$ has a canonical opmonoidal structure. This is true for all functors $P$ between cartesian monoidal categories. Moreover, this opmonoidal structure is unique, due to naturality,
\begin{equation*}
\begin{tikzcd}
	P(X\times Y) \ar{d} \ar{r}{\Delta_{X,Y}} & PX \times PY \ar{d} \\
	P(X\times 1) \ar{r}{\Delta_{X,1} = 1_X} & PX
\end{tikzcd}
\end{equation*}
Suppose now that $P$ in addition has a (given) monoidal structure $\nabla$. By the universal property of the product, it is straightforward to see that the bimonoid diagram \eqref{braidcond} commutes automatically. Therefore, whenever $\cat{C}$ is cartesian monoidal, it suffices to have a monoidal structure to obtain a bimonoidal structure:

\begin{prop}
In a cartesian monoidal category, a bimonoidal monad is the same structure as a monoidal monad.
\end{prop}

In particular, since a monoidal structure is equivalent to a commutative strength (see \cite{kock} for the closed monoidal case, and \cite[Appendix A4]{logrelations} for the general case), a commutative strong monad on a cartesian monoidal category is automatically bimonoidal in a unique way.
This is what happens, for example, for the probabilistic powerdomain on the category of domains. However, not all bimonoidal probability monads arise in this way.
In Section \ref{bmp}, we will give an example of a bimonoidal probability monad on a non-cartesian monoidal category, the Kantorovich monad on complete metric spaces.

\section{Stochastic independence}\label{independence}

Our framework allows to give a formal definition of stochastic dependence and independence in categorical terms, closely related to other notions appearing in the literature \cite{franz,simpson-independence}.

First of all, we look at an important consequence of the bimonoidality condition \eqref{braidcond}: stochastic \emph{dependence} can only be forgotten, not created. 
Consider two spaces $X$ and $Y$. Then given a joint distribution $r\in P(X\otimes Y)$, we can form the marginals $r_X\in PX$ and $r_Y\in PY$. If we try to form a joint again, via the product, the correlation is lost. Vice versa, instead, if we have two marginals, form their joint, and then divide them again into marginals, we expect to get our initial random variables back. Graphically:
\begin{equation}\label{corrforget}
 \includegraphics[align=c,scale=0.35,keepaspectratio=true]{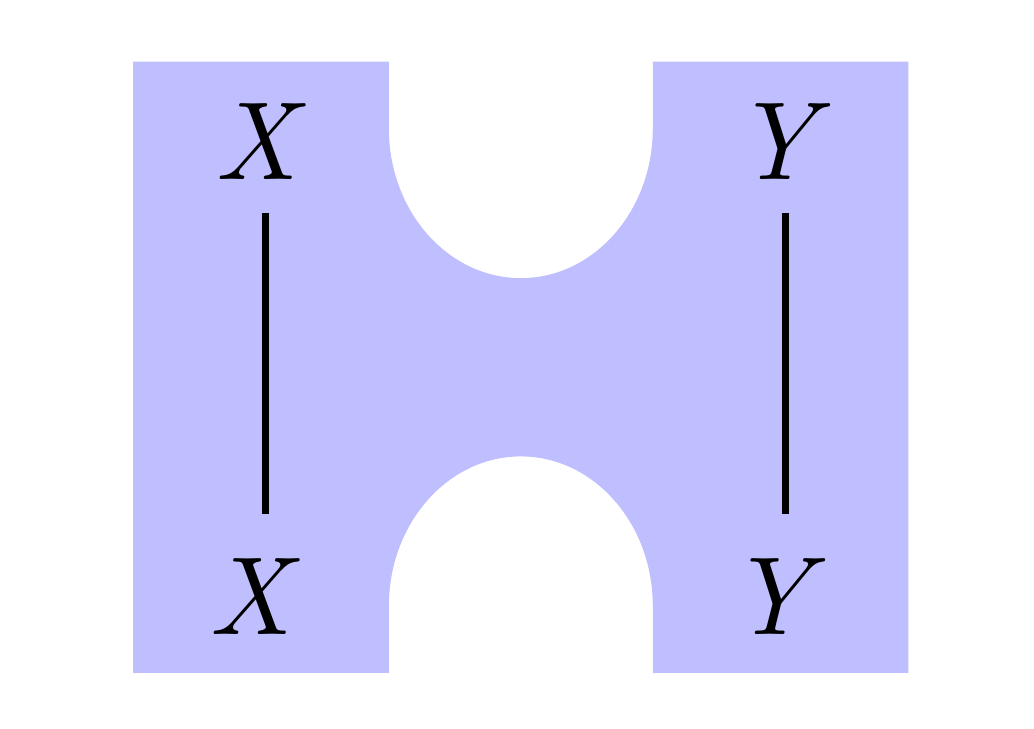}
 = 
 \includegraphics[align=c,scale=0.35,keepaspectratio=true]{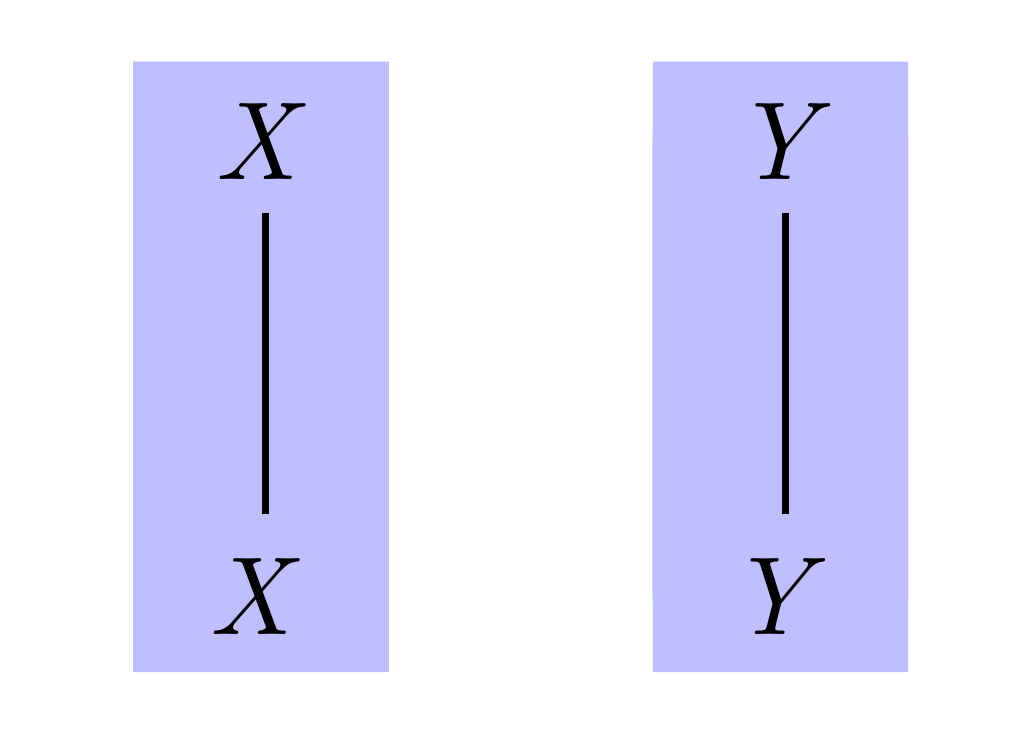}
\end{equation}
This is indeed the case under the assumptions that we've made so far:

\begin{prop}\label{correlationloss}
 Let $X,Y$ be objects of a symmetric semicartesian monoidal category $\cat{C}$. Let $P:\cat{C}\to \cat{C}$ be a bimonoidal endofunctor, with $P(1)\cong 1$. Then $\Delta\circ\nabla= \id_{PX\otimes PY}$. 
In particular, $PX\otimes PY$ is a retract of $P(X\otimes Y)$.
\end{prop}

The proposition above is proved graphically in Appendix \ref{proofcl}. It is a special case of a standard result about the so-called \emph{normal} bimonoidal functors, which can be found for example in \cite[Section 3.5]{monoidal}. 

In general we do \emph{not} get any condition $\nabla\circ\Delta= \id_{P(X\otimes Y)}$, i.e.\ in general
\begin{equation}\label{bubble}
 \includegraphics[align=c,scale=0.35,keepaspectratio=true]{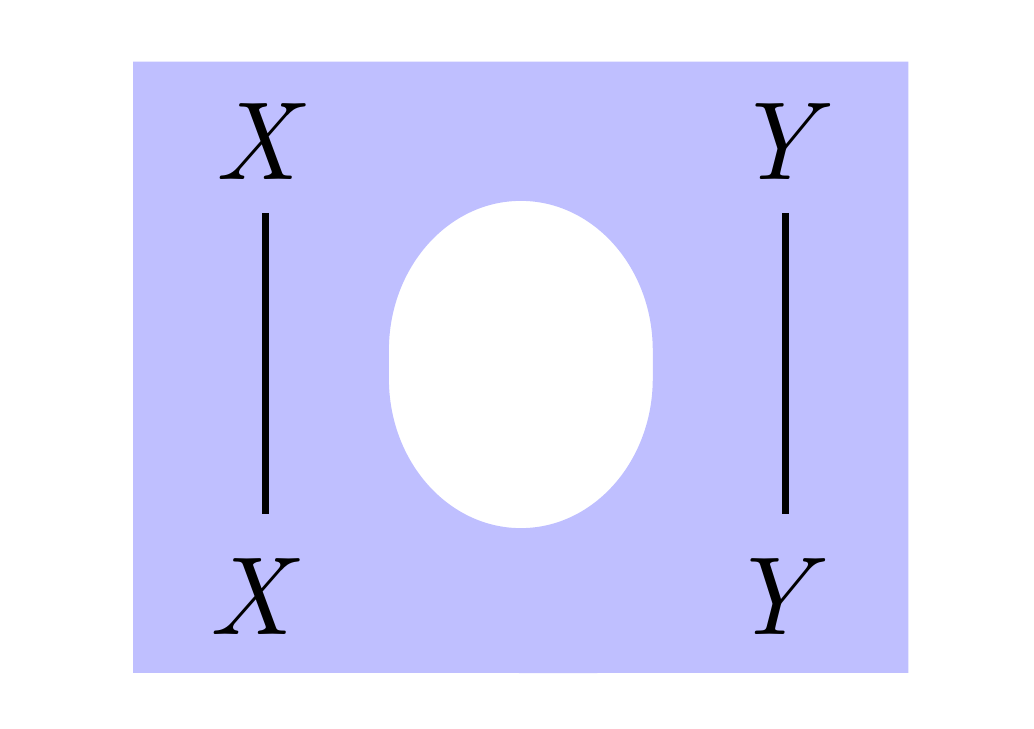}
 \ne
 \includegraphics[align=c,scale=0.35,keepaspectratio=true]{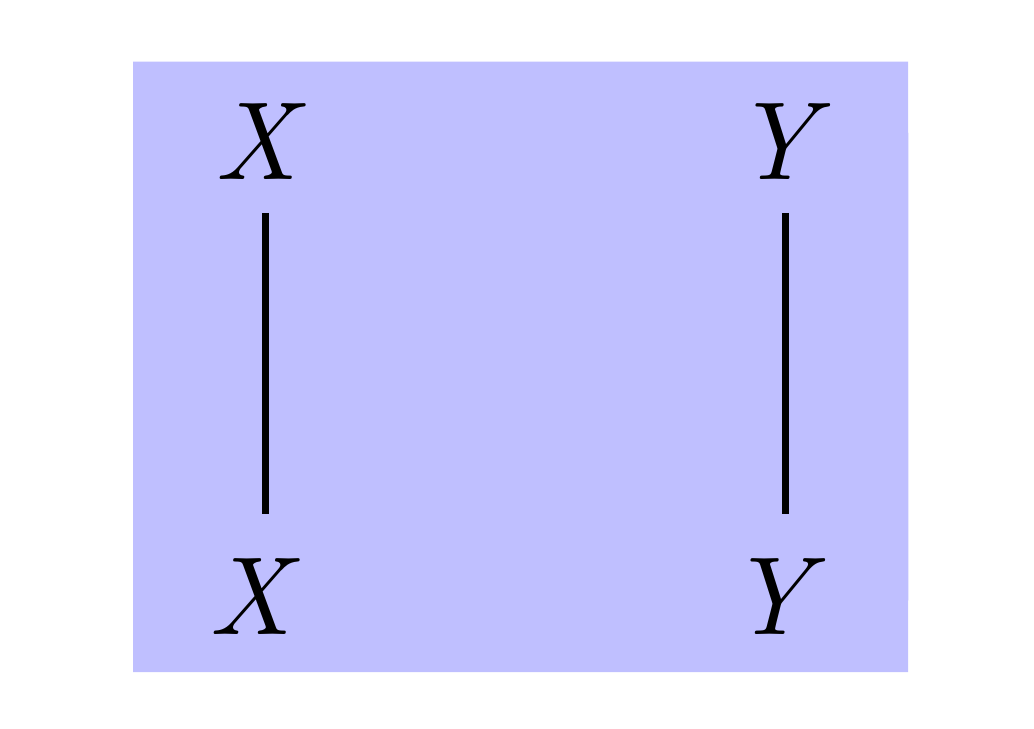}
\end{equation}
An example is given by $X=Y=\{0,1\}$, with a perfectly correlated and uniform distribution. So correlation can be forgotten, but not created, by the bimonoidal structure maps.

Going further, we can use these structures in order to talk about probabilistic independence:

\begin{deph}\label{defindep}
$X$ and $Y$ are \emph{independent} for the law $r:1\to P(X\otimes Y)$ if and only if $\nabla\circ\Delta\circ r=r$.
\end{deph}

That is, applying the left-hand side of \eqref{bubble} gives the same as applying the right-hand side if and only if we have independence. 

We are now ready for the probabilistic interpretation of the bimonoidality condition \eqref{braidcond}, which gives its main motivation: Consider any joints $WX$ and $YZ$, and form their product. In the resulting distribution, $W$ will be independent of $Y$, and $X$ will be independent of $Z$. More rigorously:

\begin{prop}\label{probinterp}
 Let $W,X,Y,Z$ be objects of a symmetric semicartesian monoidal category $\cat{C}$. Let $P:\cat{C}\to \cat{C}$ be a bimonoidal functor, with $P(1)\cong 1$. Let $r:1\to P(W\otimes X)$ and $s:1\to P(Y\otimes Z)$, and consider the law $r\otimes_\nabla s:=\nabla\circ(r\otimes s)$ on $W\otimes X\otimes Y\otimes Z$. Then after forgetting $X$ and $Z$, for the resulting law $W$ and $Y$ are independent. Just as well, after forgetting $W$ and $Y$, for the resulting law $X$ and $Z$ are independent.
\end{prop} 

A graphical proof in terms of Definition \ref{defindep} is given as well in Appendix \ref{proofcl}.

This result forms part of the semi-graphoid axioms \cite{graphoids} which axiomatize properties of conditional independence, namely in the case where the conditioning is trivial. Concretely, Proposition~\ref{probinterp} corresponds to the axiom of decomposition, stating that if $X$ is independent from $(Y,Z)$, then $X$ is also independent from $Y$. The semi-graphoid axiom of symmetry (if $X$ is independent of $Y$, then $Y$ is independent of $X$) is also satisfied  whenever we have a symmetric bimonoidal monad.

\subsection{Comparison with other notions of independence}\label{comparison}

Franz \cite{franz} defines stochastic independence in a semicartesian monoidal category in the following way: given objects $A,B_1,B_2$ (which one can think of as probability spaces), and arrows $f_1:A\to B_1$ and $f_2:A\to B_2$ (which one can think of measure-preserving maps), then $f_1$ and $f_2$ are independent if and only if there exists $h:A\to B_1\otimes B_2$ making this diagram commute:
\begin{equation}
 \begin{tikzcd}
  & A \ar[swap]{dl}{f_1} \ar{d}{h} \ar{dr}{f_2}\\
  B_1 & B_1\otimes B_2 \ar{l}{\pi_1} \ar[swap]{r}{\pi_2} & B_2
 \end{tikzcd}
\end{equation}
where $\pi_1,\pi_2$ are the projections of the tensor product. He then proves \cite[Proposition 3.5]{franz} that in the category $\cat{Prob}$ of (traditional) probability spaces, this notion of independence is equivalent to the standard one of probability theory. 
We propose a generalization of that result, which holds for categories of random elements obtained by generic \emph{cartesian} monoidal categories.

\begin{prop}\label{equivalence}
 Let $\cat{C}$ be a cartesian monoidal category and $P$ an affine bimonoidal probability monad.
 Consider a law $s:1\to PA$ and maps $f_1:A\to B_1$ and $f_2:A\to B_2$. Then $f_1$ and $f_2$ are independent in the sense of Franz \cite{franz} if and only if $B_1$ and $B_2$ are independent for the law $P(f_1,f_2)\circ s$ in the sense of Definition \ref{defindep}.
\end{prop}

So in the case of cartesian monoidal base categories, the two approaches agree. The proof can be found in Appendix \ref{proofeq}, and goes along the lines of the proof of \cite[Proposition 3.5]{franz}.

Simpson \cite{simpson-independence} defines an \emph{independence structure} as a certain collection of multispans that contains the singleton families. Given again a \emph{cartesian} monoidal category $\cat{C}$ and an affine monad $P$ on $C$, and given a finite multispan $\{f_i:A\to B_i\}_{i\in I}$ in $\cat{C}$, we can form a multispan in the category $\cat{Prob(C)}$ by precomposing with a law $r:1\to PA$. We can call such a resulting multispan \emph{independent}, in analogy with Definition \ref{defindep}, iff
\begin{equation*}
 \nabla_I \circ \Delta_I\circ P \left( (f_i)_{i\in I}  \right) \circ r = P \left( (f_i)_{i\in I} \right) \circ r ,
\end{equation*}
where $(f_i)_{i\in I}:A\to \prod_{i\in I} B_i$ is the tupling of the $f_i$ given by the cartesian monoidal structure, and $\nabla_I$ and $\Delta_I$ are the maps respectively $\prod_i (PB_i) \to P(\prod_i B_i)$ and $P(\prod_i B_i) \to \prod_i (PB_i)$ obtained by iterating respectively $\nabla$ and $\Delta$ (by associativity and coassociativity, the resulting maps are unique). 
Independent multispans defined in this way form then an independence structure in the sense of \cite[Definition 2.1]{simpson-independence}, in a way analogous to Examples 2.1 and 2.2 therein: they are closed with respect to multispan composition, and to forming subfamilies. Therefore, again in the case of a cartesian monoidal base category, our definition is compatible with Simpson's approach.

\section{Bimonoidal structure of the Kantorovich monad}\label{bmp}

The Kantorovich monad is a probability monad on complete metric spaces. It was first defined by van Breugel for compact and for complete 1-bounded metric spaces \cite{breugel}. We will use here the definitions and results of \cite{ours_kantorovich}, which work for all complete metric spaces. 

Consider the category $\cat{CMet}$ whose:
\begin{itemize}
 \item Objects are complete metric spaces;
 \item Morphisms are short maps, i.e.\ functions $f:X\to Y$ such that 
 \begin{equation*}
  d\big(f(x),f(x')\big) \le d(x,x') 
 \end{equation*}
 for all $x,x'\in X$;
 \item As monoidal structure, we define $X\otimes Y$ to be the set $X\times Y$, with the metric:
 \begin{equation*}
  d\big( (x,y) , (x',y') \big) := d(x,x') + d(y,y') .
 \end{equation*}
\end{itemize}

This category can be thought of as a category of enriched categories and functors \cite[Section 2]{lawvere}, and the monoidal structure is closed but not cartesian.
Further motivation for the choice of this category is given in \cite{ours_kantorovich}. In particular, by choosing as morphisms the short maps, one can obtain $PX$ as a colimit of spaces of empirical distributions of finite sequences \cite[Section 3]{ours_kantorovich}, which would not be possible if one allowed for more general morphisms (like continuous or Lipschitz functions). 

We recall the basic definitions of \cite{ours_kantorovich}. 

\begin{deph}
 Let $X$ be a complete metric space. 
 \begin{itemize}
  \item A Radon probability measure $p$ on $X$ is said to have \emph{finite first moment} if for every short map $f:X\to\R$,
  \begin{equation*}
   \int_X f \, dp < \infty .
  \end{equation*}
  Every such probability measure can be specified uniquely by its integration against short maps to $\R$: the set of such measures can be identified with the set of positive, Scott-continuous linear functionals on the space of Lipschitz functions on $X$. Hence, in the following, we explicitly construct such measures by specifying their action on short maps.
  \item The \emph{Kantorovich-Wasserstein} space $PX$ is the space of all Radon probability measures on $X$ with finite first moment, equipped with the metric:
  \begin{equation*}
   d(p,q) := \sup_{f:X\to\R} \left| \int_X f \, dp - \int_X f \, dq \right|,
  \end{equation*}
  where the supremum ranges over all short maps $X\to\R$.
  With this metric, $PX$ is itself a complete metric space.
  \item Given $f:X\to Y$, we define $Pf:PX\to PY$ as the map assigning to $p\in PX$ its push-forward measure $(Pf)(p):=f_*p\in PY$. The latter is defined by saying that for all $g:Y\to\R$ short,
  \begin{equation*}
   \int_Y g \, d(f_*p) := \int_X g\circ f \, dp .
  \end{equation*}
 $f_*p$ also has finite first moment, and this assignment makes $P$ into a functor.
 \end{itemize}
\end{deph}

A concise treatment of Wasserstein spaces can be found in \cite{hitch} and a more comprehensive one in \cite{villani}. For the basic measure-theoretic setting, we refer the reader to \cite{bogachev,edgar}. 

The functor $P$ admits a monad structure, with the unit $\delta:X\to PX$ given by the Dirac distributions
\begin{equation*}
 \int_X f(y) \, d (\delta(x))(y) := f(x) ,
\end{equation*}
and the multiplication $E:PPX\to PX$ given by forming the expected or average distribution,
\begin{equation*}
 \int_X f \, d (E\mu) := \int_{PX} \left( \int_X f(x) \, dp(x) \right) d\mu(p) .
\end{equation*}

We can now define product joints and marginals, which will equip $P$ with a bimonoidal structure.

\begin{deph}\label{jm}
 Let $p\in PX, q\in PY$. We denote $p\otimes_\nabla q$ the joint probability measure on $X\otimes Y$ defined by:
\begin{equation*}
\int_{X\otimes Y} f(x,y) \, d(p\otimes_\nabla q) (x,y) := \int_{X\otimes Y} f(x,y) \, dp(x)\,dq(y).
\end{equation*}

 Let now $r\in P(X\otimes Y)$. We denote $(r_X)$ the marginal probability on $X$ defined by:
\begin{equation*}
\int_{X} f(x) \, dr_X (x) := \int_{X\otimes Y} f(x) \, dr(x,y).
\end{equation*}
The marginal on $Y$ is defined analogously.
\end{deph}
It is straightforward to check that the functionals defined in Definition \ref{jm} are positive, linear, and Scott-continuous, therefore they specify uniquely Radon probability measures of finite first moment. 

In the rest of this section we will show that the joints and marginals in Definition \ref{jm} equip the Kantorovich monad on $\cat{CMet}$ with a bimonoidal monad structure (Theorem \ref{bimonoidal}). The proofs with the actual calculations are in Appendix \ref{proofs}.

We will prove now that the product joint construction equips $P$ with a monoidal structure.

\begin{deph}
 Let $X,Y\in\cat{CMet}$. We define the map $\nabla:PX\otimes PY \to P(X\otimes Y)$ as mapping $(p,q)\in PX\otimes PY$ to the joint $p\otimes_\nabla q\in P(X\otimes Y)$.
\end{deph}

\begin{prop}\label{nablasm}
 $\nabla:PX\otimes PY \to P(X\otimes Y)$ is short. 
\end{prop}

Therefore, $\nabla$ is a morphism of $\cat{CMet}$. This would not be the case if we took as monoidal structure for $\cat{CMet}$ the cartesian product: for the product metric, $\nabla$ is Lipschitz, but in general not short.
The fact that $\nabla$ equips $P$ with a monoidal structure now follows directly from the naturality and associativity of the product probability construction (as sketched in Section \ref{pmon}). In other words, the proofs of the next three statements (see Appendix \ref{proofsmonoidal}) can be adapted to most other categorical contexts in which the map $\nabla$ is of a similar form.

\begin{prop}\label{nablanat}
 $\nabla:PX\otimes PY \to P(X\otimes Y)$ is natural in $X$ and $Y$.
\end{prop}

\begin{prop}\label{lmf}
 $(P,\id_1,\nabla)$ is a symmetric lax monoidal functor $\cat{CMet}\to\cat{CMet}$.
\end{prop}

\begin{prop}\label{mmonad}
 $(P,\delta,E)$ is a symmetric monoidal monad.
\end{prop}

We know that a monoidal monad is the same as a commutative monad, and therefore obtain:

\begin{cor}
 $P$ is a commutative strong monad, with strength $X\otimes PY\to P(X\otimes Y)$ given by:
\begin{equation*}
(x,q) \mapsto \delta_x \otimes_\nabla q \in P(X\otimes Y) .
\end{equation*}
\end{cor}


We now turn to the analogous statements for the marginals, and show that they equip $P$ with an opmonoidal structure. 

\begin{deph}
 Let $X,Y\in\cat{CMet}$. We define the map $\Delta: P(X\otimes Y) \to PX\otimes PY$ as mapping $r\in P(X\otimes Y)$ to the pair of marginals $(r_X,r_Y)\in PX\otimes PY$.
\end{deph}

\begin{prop}\label{deltasm}
 $\Delta:P(X\otimes Y)\to PX \otimes PY$ is short.
\end{prop}

Therefore $\Delta$ is a morphism of $\cat{CMet}$. Again, the following statements follow just from the properties of marginals, and their proofs (see Appendix \ref{proofsopmonoidal}) can be adapted to most other categorical contexts provided that $\Delta$ is of a similar form.

\begin{prop}\label{deltanat}
 $\Delta:P(X\otimes Y)\to PX \otimes PY$ is natural in $X,Y$.
\end{prop}

\begin{prop}\label{olmf}
 The marginal map together with the trivial counitor defines a symmetric oplax monoidal functor $(P,\id_1,\Delta)$.
\end{prop}

\begin{prop}\label{ommonad}
 $(P,\delta,E)$ is a symmetric opmonoidal monad.
\end{prop}


The lax and oplax monoidal structure interact to give a bimonoidal structure. The following statements also follow just from the properties of joints and marginals.

\begin{prop}\label{blmf}
 $P$ is a symmetric bilax monoidal functor.
\end{prop}

The main result then just follows as a corollary:

\begin{thm}\label{bimonoidal}
 The Kantorovich monad is a symmetric bimonoidal monad, with monoidal structure given by the product joint, and opmonoidal structure given by the marginals.
\end{thm}

By Proposition~\ref{correlationloss}, we therefore have:

\begin{cor}
 $\Delta_{X,Y} \circ \nabla_{X,Y} =\id_{PX\otimes PY}$. Therefore, the inclusion $\nabla$ of product measures into general joints, is an isometric embedding for the Kantorovich metric, and its image is a retract of the space of all joints.
\end{cor}

\appendix

\section{Monoidal, opmonoidal and bimonoidal monads}\label{monoidalstuff}

We recall the definition of the different monoidal structures for a functor, for the case of braided  (including symmetric) monoidal categories. For more results and more general definitions, we refer to \cite{monoidal}.

Let $(\cat{C},\otimes)$ and $(\cat{D},\otimes)$ be braided monoidal categories. 

\begin{deph}\label{mfunc}
 A \emph{lax monoidal functor} $(\cat{C},\otimes) \to (\cat{D},\otimes)$ is a triple $(F,\eta,\nabla)$, such that:
 \begin{enumerate}
  \item $F:C\to D$ is a functor;
  \item The ``unit'' $\eta:1_\cat{D} \to F(1_\cat{C})$ is a morphism of $\cat{D}$;
  \item The ``multiplication'' $\nabla:F(-) \otimes F(-) \Rightarrow F(- \otimes -)$ is a natural transformation of functors $\cat{C}\times \cat{C} \to \cat{D}$;
  \item The following ``associativity'' diagram commutes for every $X,Y,Z$ in $\cat{C}$:
  \begin{equation*}
   \begin{tikzcd}
    (FX \otimes FY) \otimes FZ \ar{r}{\cong} 	\ar{d}{\nabla_{X,Y}\otimes \id}	& FX \otimes (FY \otimes FZ)	\ar{d}{\id \otimes \nabla_{Y,Z}} \\
    F(X\otimes Y) \otimes FZ 			\ar{d}{\nabla_{X\otimes Y, Z}}	& FX \otimes F(Y\otimes Z)	\ar{d}{\nabla_{X,Y\otimes Z}} \\
    F((X \otimes Y) \otimes Z) \ar{r}{\cong} 					& F(X \otimes (Y \otimes Z)) 
   \end{tikzcd}
  \end{equation*}
  \item The following ``unitality'' diagrams commute for every $X$ in $\cat{C}$:
  \begin{equation*}
   \begin{tikzcd}
    1_\cat{D} \otimes FX 	\ar{d}{\cong}	\ar{r}{\eta\otimes \id}	& F(1_\cat{C}) \otimes FX 	\ar{d}{\nabla_{1_\cat{C},X}} \\
    FX 									& F(1_\cat{C} \otimes X) 		\ar{l}{\cong}
   \end{tikzcd}
   \qquad
   \begin{tikzcd}
    FX \otimes 1_\cat{D} 	\ar{d}{\cong}	\ar{r}{\id\otimes\eta}	& FX \otimes F(1_\cat{C}) 	\ar{d}{\nabla_{X,1_\cat{C}}} \\
    FX 									& F(X \otimes 1_\cat{C}) 		\ar{l}{\cong}
   \end{tikzcd}
  \end{equation*}
 \end{enumerate}
 We say that $(F,\eta,\nabla)$ is also \emph{braided}, or \emph{symmetric} if $\cat{C}$ is symmetric, if in addition the multiplication commutes with the braiding:
 \begin{equation*}
 \begin{tikzcd}
  FX \otimes FY \ar{d}{\nabla} \ar{r}{\cong} & FY \otimes FX \ar{d}{\nabla} \\
  F(X\otimes Y) \ar{r}{\cong} & F(Y\otimes X)
 \end{tikzcd}
\end{equation*}
\end{deph}

\begin{deph} 
  Let $(F,\eta_F,\nabla_F)$ and $(G,\eta_G,\nabla_G)$ be lax monoidal functors $(\cat{C},\otimes) \to (\cat{D},\otimes)$. A \emph{lax monoidal natural transformation}, or just \emph{monoidal natural transformation} when it's clear from the context, is a natural transformation $\alpha:F\Rightarrow G$ which is compatible with the unit and multiplication map. In particular, the following diagrams must commute (for all $X,Y\in \cat{C}$):
\begin{equation*}
 \begin{tikzcd}
  1_\cat{D} \ar{r}{\eta_F} \ar[swap]{dr}{\eta_G}  & F(1_\cat{C}) \ar{d}{\alpha_{1_\cat{C}}} \\
  & G(1_\cat{C})
 \end{tikzcd}
 \qquad
 \begin{tikzcd}
  FX \otimes FY \ar{r}{\nabla_F} \ar{d}{\alpha_X\otimes\alpha_Y} & F(X\otimes Y) \ar{d}{\alpha_{X\otimes Y}} \\
  GX \otimes GY \ar{r}{\nabla_G}  & G(X\otimes Y) 
 \end{tikzcd}
\end{equation*}
\end{deph}

\begin{deph}\label{opmfunc}
 An \emph{oplax monoidal functor} $(\cat{C},\otimes) \to (\cat{D},\otimes)$ is a triple $(F,\epsilon,\Delta)$, such that:
 \begin{enumerate}
  \item $F:C\to D$ is a functor;
  \item The ``counit'' $\epsilon: F(1_\cat{C}) \to 1_\cat{D}$ is a morphism of $\cat{D}$;
  \item The ``comultiplication'' $\Delta: F(- \otimes -) \Rightarrow F(-) \otimes F(-)$ is a natural transformation of functors $\cat{C}\times \cat{C} \to \cat{D}$;
  \item The following ``coassociativity'' diagram commutes for every $X,Y,Z$ in $\cat{C}$:
  \begin{equation*}
   \begin{tikzcd}
    F((X \otimes Y) \otimes Z) \ar{r}{\cong} 	\ar{d}{\Delta_{X\otimes Y, Z}} 	& F(X \otimes (Y \otimes Z))	\ar{d}{\Delta_{X,Y\otimes Z}} \\
    F(X\otimes Y) \otimes FZ 			\ar{d}{\Delta_{X,Y}\otimes \id}	& FX \otimes F(Y\otimes Z)	\ar{d}{\id \otimes \Delta_{Y,Z}} \\ 
    (FX \otimes FY) \otimes FZ \ar{r}{\cong} 		& FX \otimes (FY \otimes FZ)
   \end{tikzcd}
  \end{equation*}
  \item The following ``counitality'' diagrams commute for every $X$ in $\cat{C}$:
  \begin{equation*}
   \begin{tikzcd}
   F(1_\cat{C} \otimes X)	\ar{d}{\cong} 	\ar{r}{\Delta_{1_\cat{C},X}}	& F(1_\cat{C}) \otimes FX	\ar{d}{\epsilon\otimes \id}\\
   FX 										& 1_\cat{D} \otimes FX \ar{l}{\cong}
   \end{tikzcd}
   \qquad
   \begin{tikzcd}
   F(X \otimes 1_\cat{C})	\ar{d}{\cong} 	\ar{r}{\Delta_{X,1_\cat{C}}}	& FX \otimes F(1_\cat{C})	\ar{d}{\id\otimes\epsilon}\\
   FX 										& FX \otimes 1_\cat{D}  \ar{l}{\cong}
   \end{tikzcd}
  \end{equation*}
 \end{enumerate}
  We say that $(F,\epsilon,\Delta)$ is also \emph{braided}, or \emph{symmetric} if $\cat{C}$ is symmetric, if in addition the comultiplication commutes with the braiding:
\begin{equation*}
 \begin{tikzcd}
  F(X\otimes Y) \ar{d}{\Delta} \ar{r}{\cong} & F(Y\otimes X) \ar{d}{\Delta} \\
  FX \otimes FY \ar{r}{\cong} & FY \otimes FX
 \end{tikzcd}
\end{equation*}
\end{deph}

\begin{deph} 
 Let $(F,\epsilon_F,\Delta_F)$ and $(G,\epsilon_G,\Delta_G)$ be oplax monoidal functors $(\cat{C},\otimes) \to (\cat{D},\otimes)$. An \emph{oplax monoidal natural transformation}, or just \emph{monoidal natural transformation} when it's clear from the context, is a natural transformation $\alpha:F\Rightarrow G$ which is compatible with the counit and comultiplication map. In particular, the following diagrams must commute (for all $X,Y\in \cat{C}$):
\begin{equation*}
 \begin{tikzcd}
  1_\cat{D} \ar[leftarrow]{r}{\epsilon_F} \ar[swap,leftarrow]{dr}{\epsilon_G}  & F(1_\cat{C}) \ar{d}{\alpha_{1_\cat{C}}} \\
  & G(1_\cat{C})
 \end{tikzcd}
 \qquad
 \begin{tikzcd}
  FX \otimes FY \ar[leftarrow]{r}{\Delta_F} \ar{d}{\alpha_X\otimes\alpha_Y} & F(X\otimes Y) \ar{d}{\alpha_{X\otimes Y}} \\
  GX \otimes GY \ar[leftarrow]{r}{\Delta_G}  & G(X\otimes Y) 
 \end{tikzcd}
\end{equation*}
\end{deph}

\begin{deph}\label{bimfunc}
 A \emph{bilax monoidal functor} $(\cat{C},\otimes) \to (\cat{D},\otimes)$ is a ``quintuplet'' $(F,\eta,\nabla,\epsilon,\Delta)$ such that:
 \begin{enumerate}
  \item $(F,\eta,\nabla): (\cat{C},\otimes) \to (\cat{D},\otimes)$ is a lax monoidal functor;
  \item $(F,\epsilon,\Delta): (\cat{C},\otimes) \to (\cat{D},\otimes)$ is an oplax monoidal functor;
  \item The following ``bimonoidality'' diagram commutes:
  \begin{equation}\label{braiding}
    \begin{tikzcd}
    & F(W\otimes X) \otimes F(Y\otimes Z) \ar[swap]{dl}{\nabla_{W\otimes X, Y\otimes Z}} \ar{dr}{\Delta_{W,X}\otimes \Delta_{Y,Z}} \\
    F(W\otimes X \otimes Y \otimes Z) \ar[swap]{d}{\cong} & & F(W) \otimes F(X) \otimes F(Y) \otimes F(Z) \ar{d}{\cong} \\
    F(W\otimes Y \otimes X \otimes Z) \ar[swap]{dr}{\Delta_{W\otimes Y,X\otimes Z}} & & F(W) \otimes F(Y) \otimes F(X) \otimes F(Z) \ar{dl}{\nabla_{W,Y}\otimes\nabla_{X,Z}} \\
    & F(W\otimes Y) \otimes F(X\otimes Z)
    \end{tikzcd}
  \end{equation}
  \item The following three ``unit/counit'' diagrams commute:
    \begin{equation*}\begin{tikzcd}
    1 \ar{r}{\eta} \idar{dr} & F(1) \ar{d}{\epsilon} \\
    & 1
    \end{tikzcd}
    \qquad
    \begin{tikzcd}
    1 \ar[swap]{d}{\cong} \ar{r}{\eta} & F(1) \ar{r}{\cong} & F(1\otimes 1) \ar{d}{\Delta_{1,1}} \\
    1\otimes 1 \ar[swap]{rr}{\eta\otimes\eta} & & F(1) \otimes F(1)
    \end{tikzcd}
   \end{equation*}
   \begin{equation*} 
    \begin{tikzcd}
    1 & F(1) \ar[swap]{l}{\epsilon} & F(1\otimes 1) \ar[swap]{l}{\cong} \\
    1\otimes 1 \ar{u}{\cong} & & F(1) \otimes F(1) \ar{ll}{\epsilon\otimes\epsilon} \ar[swap]{u}{\nabla_{1,1}}
    \end{tikzcd}\end{equation*}
 \end{enumerate}
\end{deph}

\begin{deph}
  Let $(F,\epsilon_F,\Delta_F)$ and $(G,\epsilon_G,\Delta_G)$ be bilax monoidal functors $(\cat{C},\otimes) \to (\cat{D},\otimes)$. A \emph{bilax monoidal natural transformation}, or just \emph{monoidal natural transformation} when it's clear from the context, is a natural transformation $\alpha:F\Rightarrow G$ which is a lax and oplax natural transformation. 
\end{deph}

\begin{deph}
 Now, we define:
 \begin{itemize}
  \item A \emph{monoidal monad} is a monad in the bicategory of monoidal categories, lax monoidal functors, and monoidal natural transformations;
  \item An \emph{opmonoidal monad} is a monad in the bicategory of monoidal categories, oplax monoidal functors, and monoidal natural transformations;
  \item A \emph{bimonoidal monad} is a monad in the bicategory of braided monoidal categories, bilax monoidal functors, and monoidal natural transformations.
 \end{itemize}
 In the third definition, we need the symmetry (or at least a braiding) in order to express the bimonoid equation that is part of the definition of bilax monoidal functor~\cite{monoidal}, even if the functor itseld if not braided. If the functor is braided, we can define in addition:
 \begin{itemize}
  \item A \emph{braided} (resp.\ \emph{symmetric}) \emph{monoidal monad} is a monad in the bicategory of braided (resp.\ symmetric) monoidal categories, braided lax monoidal functors, and monoidal natural transformations;
  \item An \emph{braided} (resp.\ \emph{symmetric}) \emph{opmonoidal monad} is a monad in the bicategory of braided (resp.\ symmetric) monoidal categories, braided oplax monoidal functors, and monoidal natural transformations;
  \item A \emph{braided} (resp.\ \emph{symmetric}) \emph{bimonoidal monad} is a monad in the bicategory of braided (resp.\ symmetric) monoidal categories, braided bilax monoidal functors, and monoidal natural transformations.
 \end{itemize}
\end{deph}

\section{Proofs}\label{proofs}

Here are the detailed proofs of the statements in the main text.

\subsection{Graphical proofs}\label{proofcl}

\begin{proof}[Proof of Proposition \ref{correlationloss}]
Let $X=Y=1_\cat{C}$ in the bimonoidality diagram \eqref{braidcond}, and rename $W$ to $X$ and $Z$ to $Y$ for convenience. 
Then we get:
\begin{equation*}
 \includegraphics[align=c,scale=0.35,keepaspectratio=true]{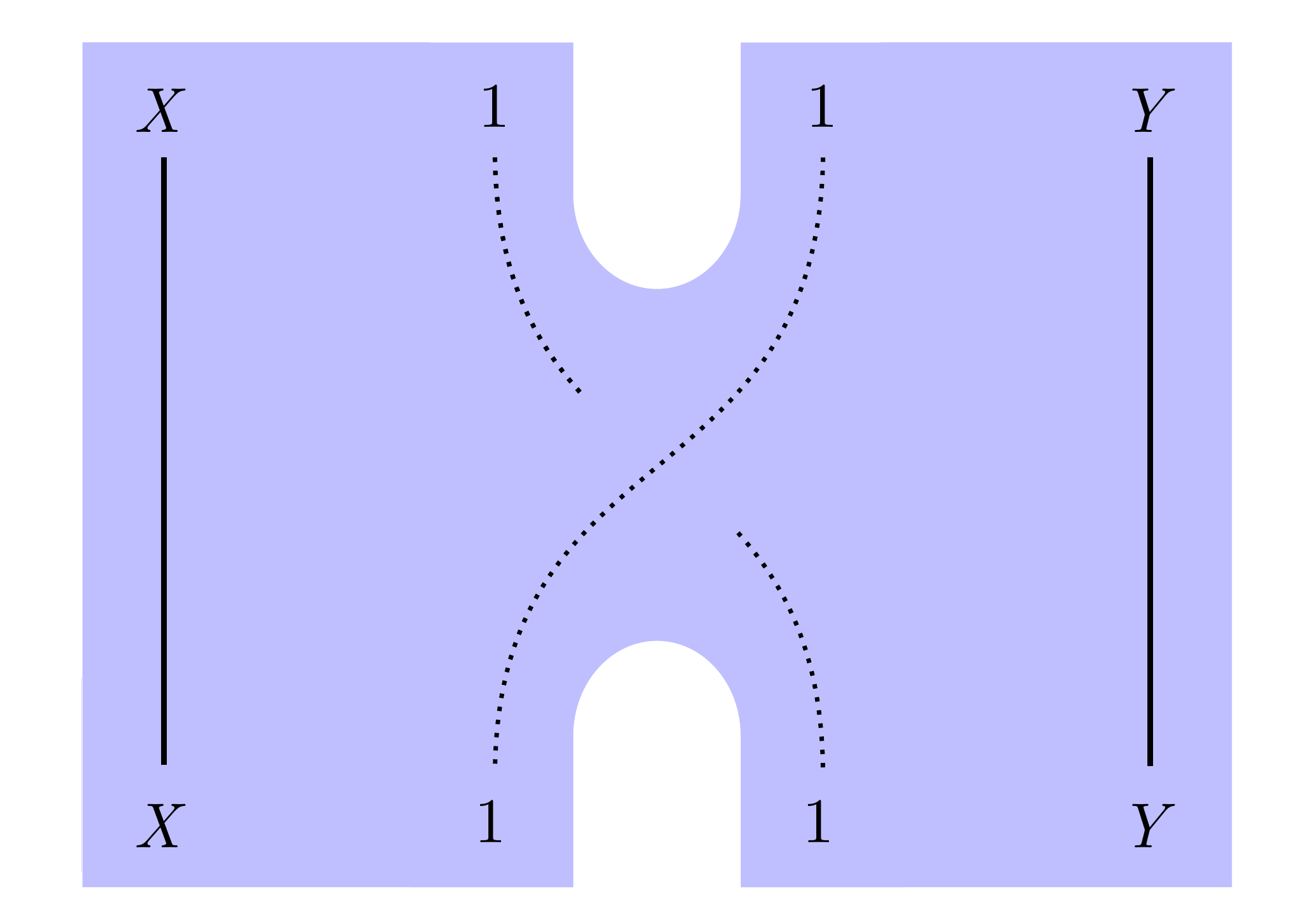}
 =
 \includegraphics[align=c,scale=0.35,keepaspectratio=true]{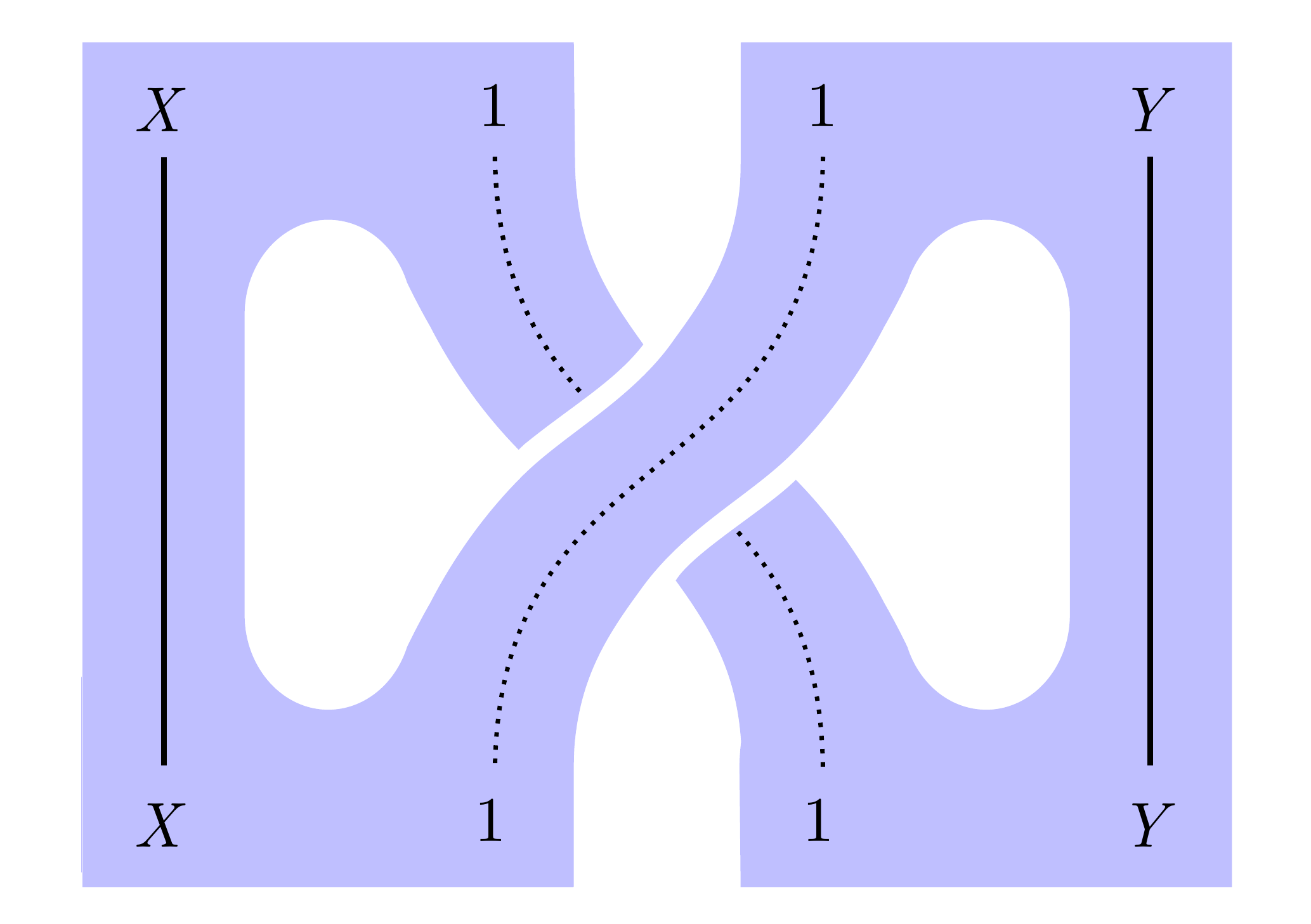}
\end{equation*}
Now since the braiding at $1\otimes 1$ is just the identity, we can even simplify the condition to:
\begin{equation*}
 \includegraphics[align=c,scale=0.35,keepaspectratio=true]{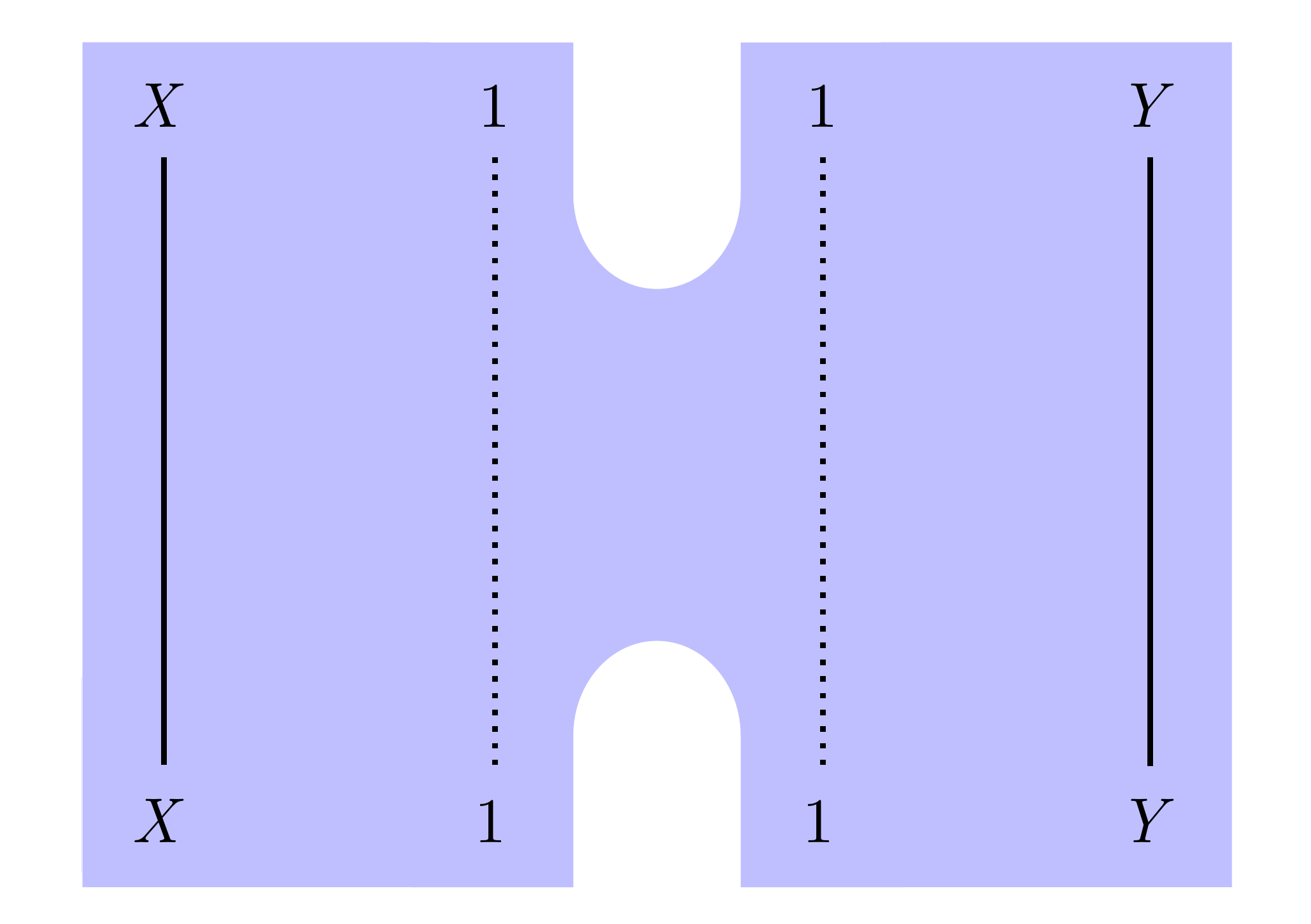}
 =
 \includegraphics[align=c,scale=0.35,keepaspectratio=true]{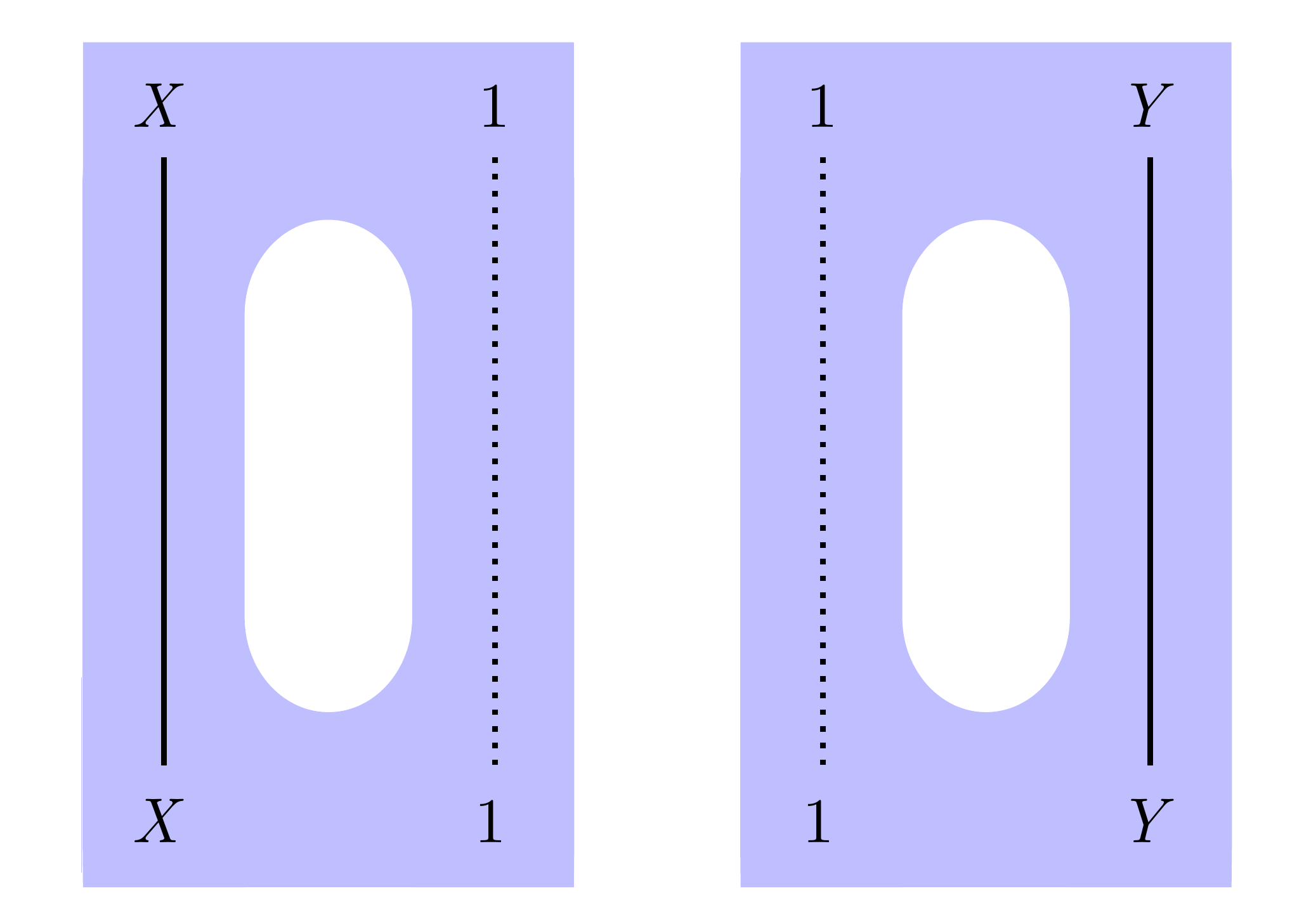}
\end{equation*}
or more concisely:
\begin{equation*}
 \includegraphics[align=c,scale=0.3,keepaspectratio=true]{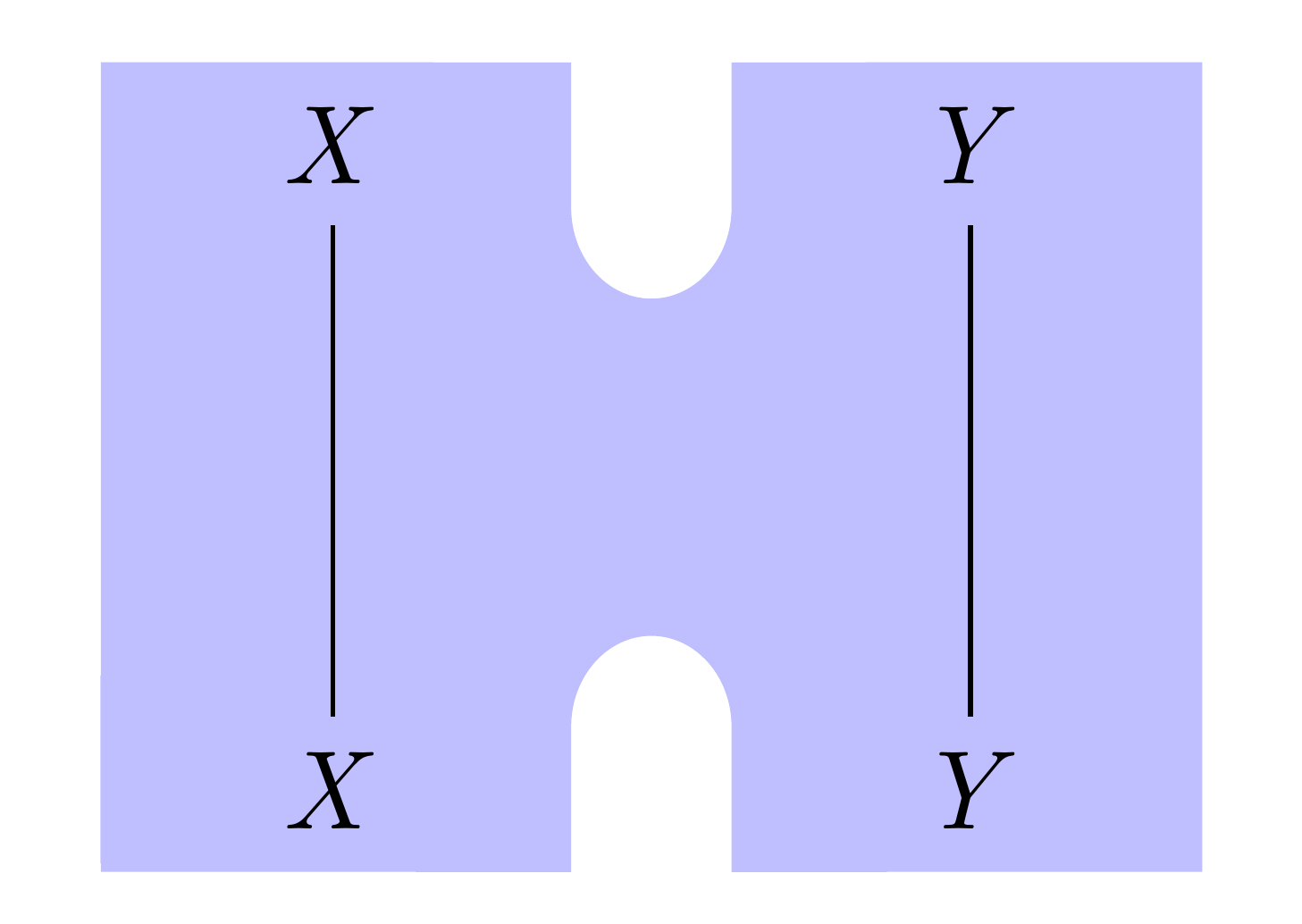}
 =
 \includegraphics[align=c,scale=0.3,keepaspectratio=true]{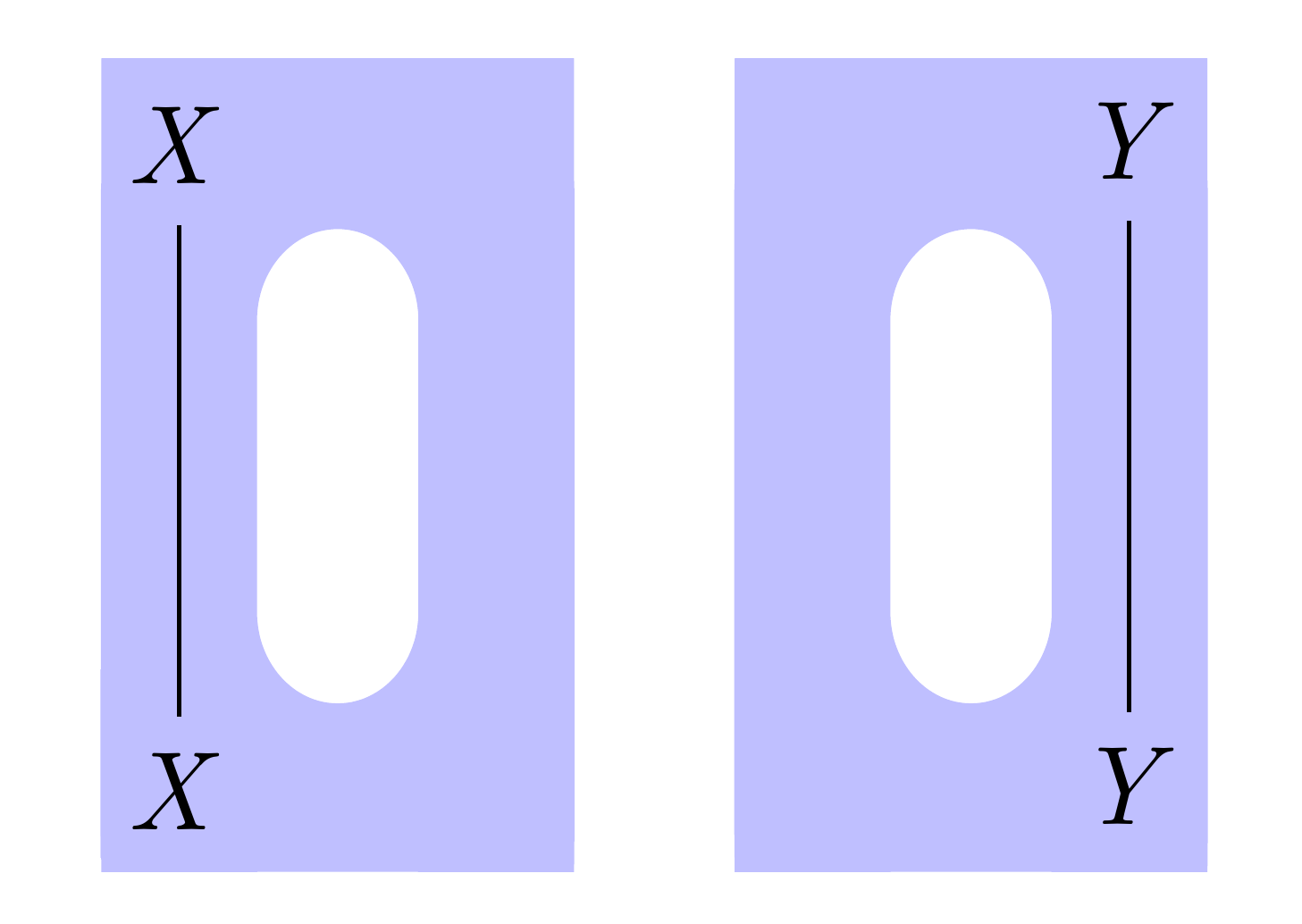}
\end{equation*}
We now notice that:
\begin{equation*}
 \includegraphics[align=c,scale=0.35,keepaspectratio=true,clip=true,trim=50pt 0pt 50pt 0pt]{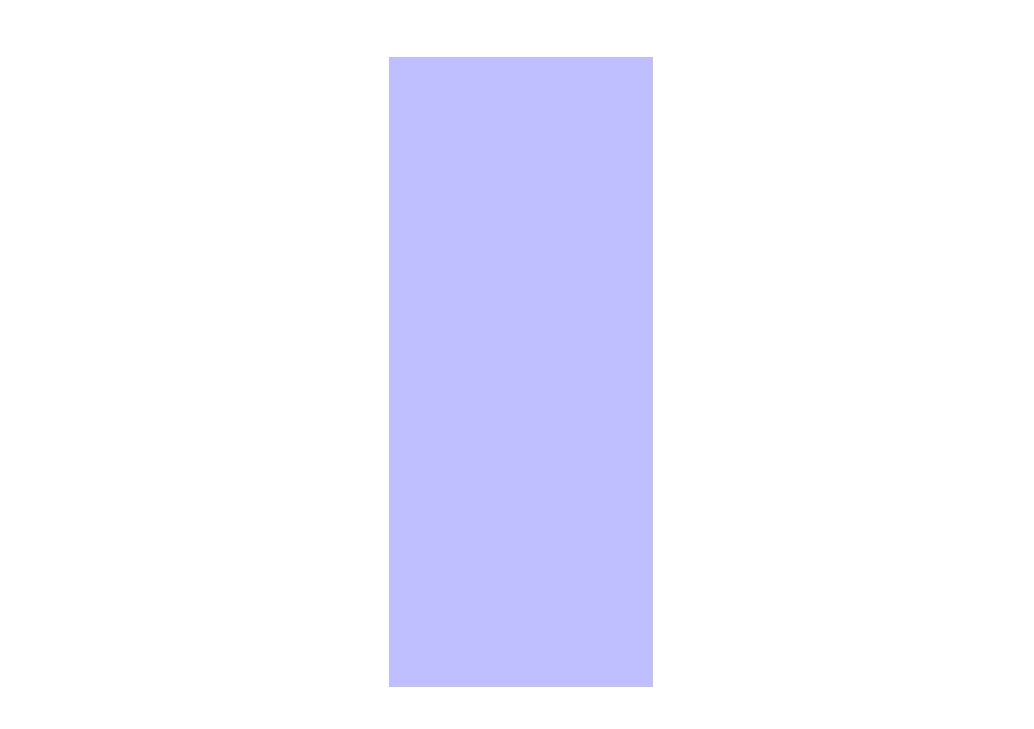}
 =
 \includegraphics[align=c,scale=0.35,keepaspectratio=true,clip=true,trim=50pt 0pt 50pt 0pt]{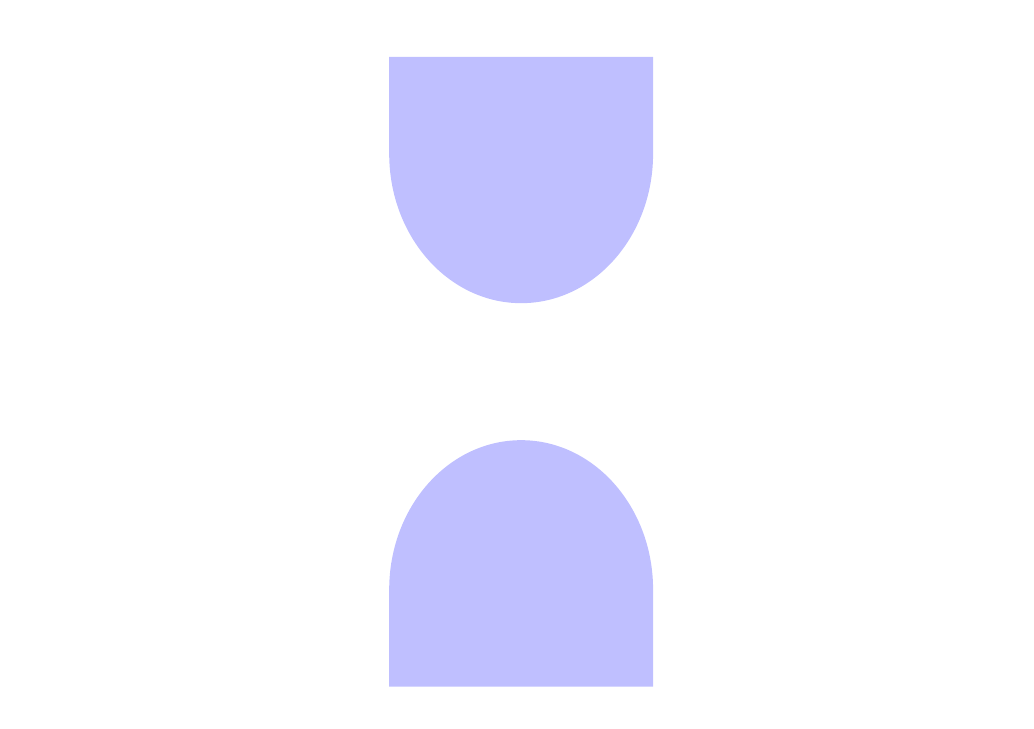}
\end{equation*}
since both maps are just the identities at $1$. This is the crucial step. We are left with:
\begin{equation*}
 \includegraphics[align=c,scale=0.3,keepaspectratio=true]{.//proof5.pdf}
 = 
 \includegraphics[align=c,scale=0.3,keepaspectratio=true]{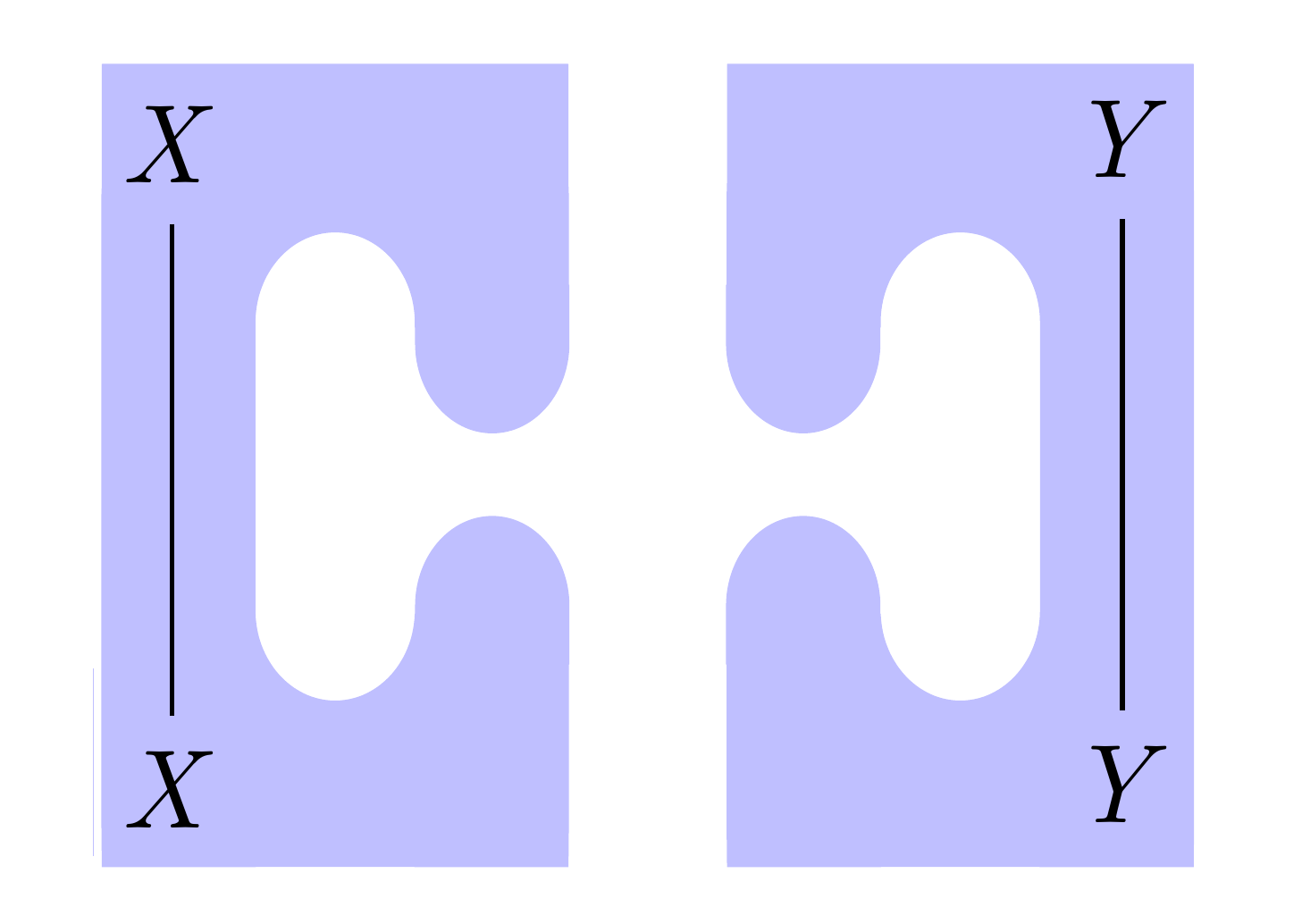}
\end{equation*}
which because of all the unit and counit conditions is equivalent to
\begin{equation*}
 \includegraphics[align=c,scale=0.35,keepaspectratio=true]{.//corr1.pdf}
 = 
 \includegraphics[align=c,scale=0.35,keepaspectratio=true]{.//corr2.pdf}
\end{equation*}
i.e. equation \eqref{corrforget}.
\end{proof}

\begin{proof}[Proof of Proposition \ref{probinterp}]
Consider the left side of \eqref{braidcond} and forget $X$ and $Z$ using the unique maps to $1$, and compose at the remaining $W\otimes Y$ with the left-hand side of \eqref{bubble}. We get:
\begin{equation*}
 \includegraphics[align=c,scale=0.35,keepaspectratio=true,clip=true,trim=0pt 300pt 0pt 0pt]{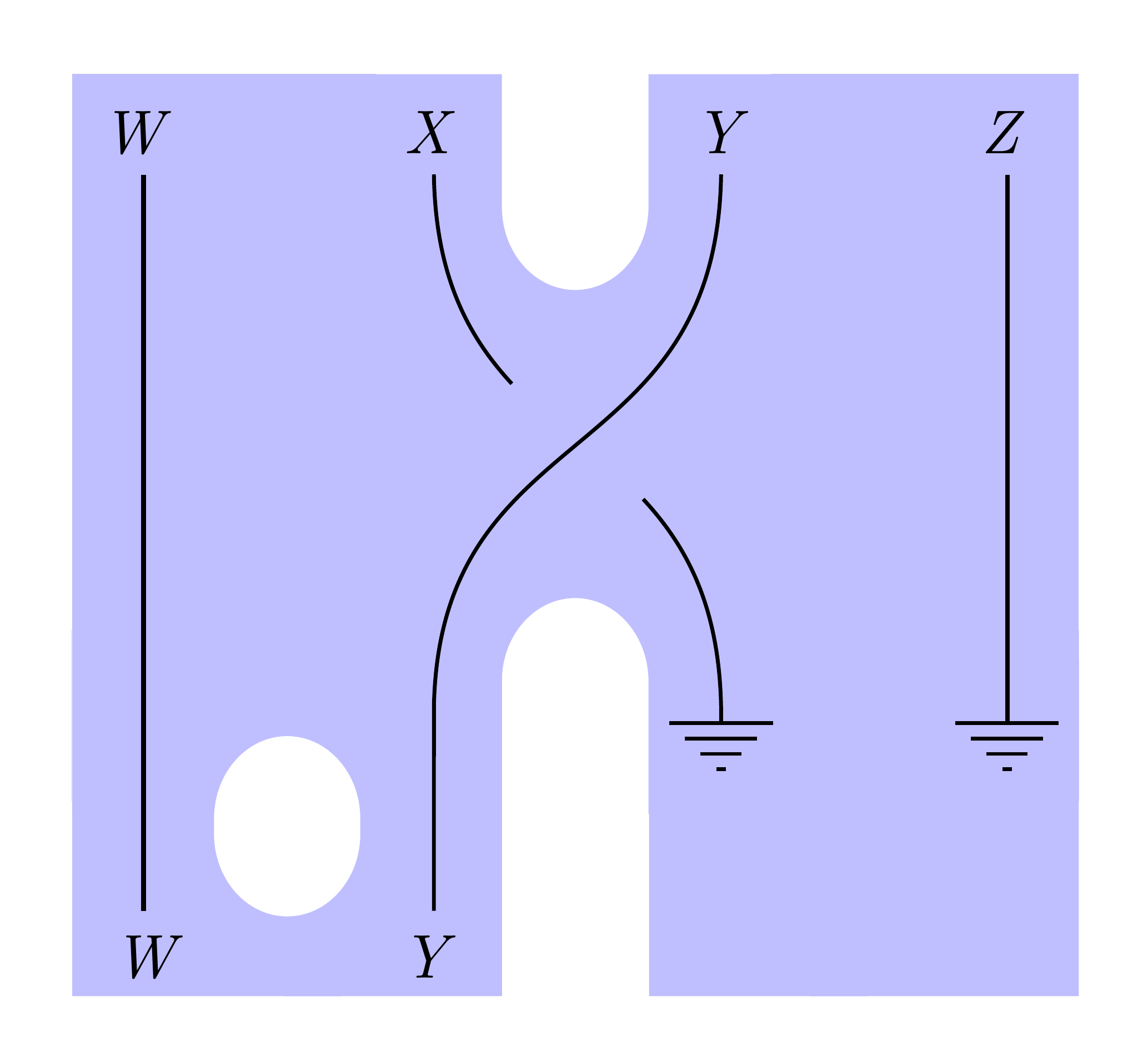}
\end{equation*}
Applying \eqref{braidcond} on the left and affinity of $P$ on the right we get:
\begin{equation*}
 \includegraphics[align=c,scale=0.35,keepaspectratio=true,clip=true,trim=0pt 300pt 0pt 0pt]{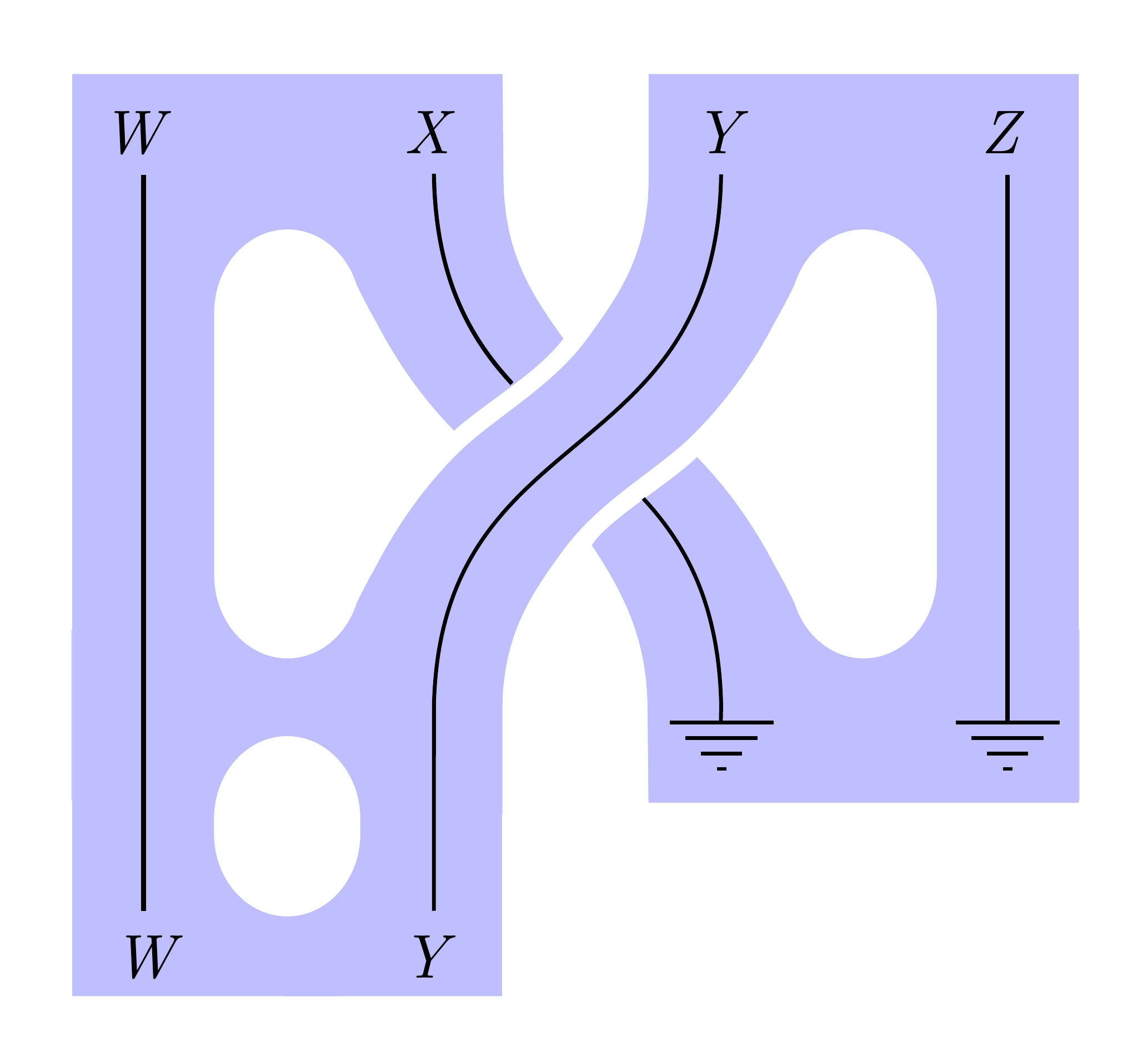}
\end{equation*}
and applying \eqref{corrforget} on the left we now get:
\begin{equation*}
 \includegraphics[align=c,scale=0.35,keepaspectratio=true,clip=true,trim=0pt 300pt 0pt 0pt]{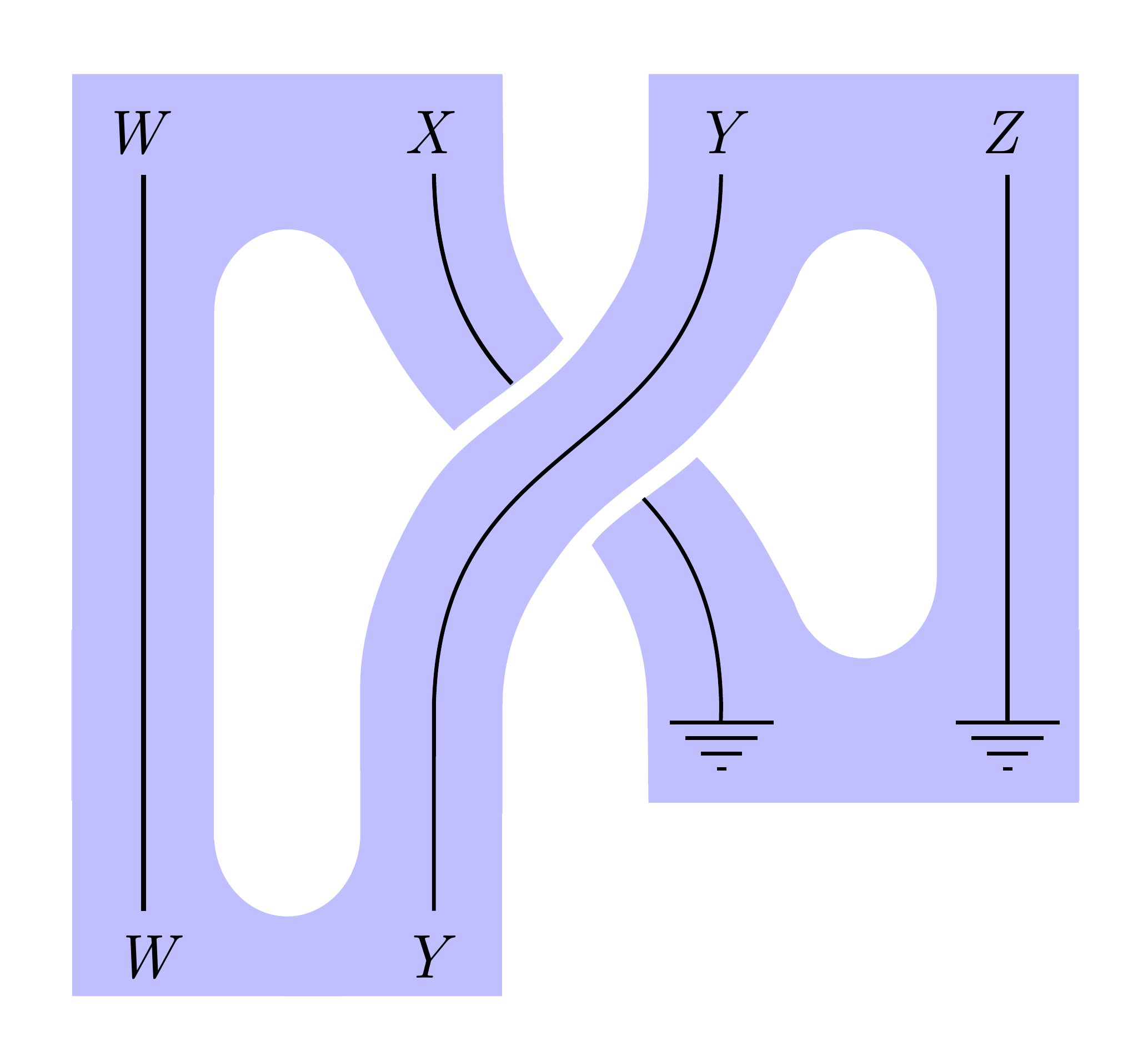}
\end{equation*}
which, before applying the ground wire maps, is the right-hand side of \eqref{braidcond}, which is therefore equal to its left-hand side. Hence by Definition \ref{defindep}, $W$ is independent of $Y$ for any law in the form given in the hypothesis. For $X$ and $Z$ we can proceed analogously.
\end{proof}

\subsection{Proof of equivalence of the notions of independence}\label{proofeq}

\begin{proof}[Proof of Proposition \ref{equivalence}]
 In $\cat{Prob(C)}$, $f_1$ and $f_2$ are independent in the sense of Franz with respect to the law $s : 1 \to PA$ if and only if there exists $h:A\to B_1\times B_2$ such that the following diagram commutes:
 \begin{equation}
 \begin{tikzcd}
  1 \ar[dotted,swap,bend right=50]{dd}{r_1} \ar[dotted,swap,pos=0.35]{ddr}{r_1\otimes_\nabla r_2} \ar[dotted]{dr}{s} \ar[dotted,bend left=40]{ddrr}{r_2} \\
  & A \ar[swap]{dl}{f_1} \ar{d}{h} \ar{dr}{f_2}\\
  B_1 & B_1\times B_2 \ar{l}{\pi_1} \ar[swap]{r}{\pi_2} & B_2
 \end{tikzcd}
\end{equation}
where $\pi_1$ and $\pi_2$ are the projections of $\cat{C}$, where the dotted arrows from $1$, with a slight abuse of notation, denote Kleisli morphisms ($s:1\to PA$, etcetera), and where $r_1$ and $r_2$ denote the resulting laws on $B_1$ and $B_2$.

Now suppose that such an $h$ exists. By the universal property of the product, it must necessarily be equal to $(f_1,f_2)$. Therefore $P(f_1,f_2) \circ s = r_1\otimes_\nabla r_2 = \nabla\circ (r_1\otimes r_2)$. Now using Proposition \ref{correlationloss},
\begin{align*}
 \nabla\circ\Delta\circ P(f_1,f_2) \circ s = \nabla\circ\Delta\circ \nabla\circ (r_1\otimes r_2) = \nabla\circ (r_1\otimes r_2) = P(f_1,f_2) \circ s ,
\end{align*}
so $B_1$ and $B_2$ are independent in the sense of Definition \ref{defindep}.

Conversely, suppose that $\nabla\circ\Delta\circ P(f_1,f_2) \circ s = P(f_1,f_2) \circ s$. Then we have
\[
	P(f_1,f_2) \circ s = \nabla\circ\Delta\circ P(f_1,f_2)\circ s = \nabla\circ(r_1,r_2) = r_1\otimes_\nabla r_2,
\]
as was to be shown.
\end{proof}

\subsection{Monoidal structure of the Kantorovich monad}\label{proofsmonoidal}

In order to prove Proposition \ref{nablasm}, first a useful result:
\begin{prop}\label{partintsm}
 Let $f:X\otimes Y\to\R$ be short. Let $p\in PX$. Then the function
 \begin{equation*}
 \left(  \int_X f(x,-)\,dp(x) \right) : Y\to \R
 \end{equation*}
 is short as well.
\end{prop}

\begin{proof}[Proof of Proposition \ref{partintsm}]
 First of all, $f:X\otimes Y \to \R$ being short means that for every $x,x'\in X, y,y'\in Y$:
\begin{equation*}
|f(x,y) - f(x',y')| \le d(x,x') + d(y,y').
\end{equation*}
Now:
\begin{align*}
&\left| \int_X f(x,y)\,dp(x) - \int_X f(x,y')\,dp(x) \right| \\
&= \left| \int_X \big( f(x,y)- f(x,y') \big) \,dp(x) \right| \\
&\le  \int_X \left| f(x,y)- f(x,y') \right| \,dp(x) \\
&\le \int_X  \big( d(x,x) + d(y,y') \big) \,dp(x) \\
&= \int_X d(y,y') \,dp(x) \\
&= d(y,y') . 
\end{align*}
\end{proof}

\begin{proof}[Proof of Proposition \ref{nablasm}]
 To prove that $\nabla$ it is short, let $p,p'\in PX, q,q'\in PY$. Then
\begin{align*}
&d\big( \nabla(p,q), \nabla(p',q') \big) \\
&= d\big( p\otimes_\nabla q, p'\otimes_\nabla q' \big) \\
&= \sup_{f: X\otimes Y\to \R} \int_{X\otimes Y} f(x,y)\, d(p\otimes_\nabla q - p'\otimes_\nabla q')(x,y) \\
&= \sup_{f: X\otimes Y\to \R} \int_{X\otimes Y} f(x,y)\, d\big(p\otimes_\nabla q - p'\otimes_\nabla q + p'\otimes_\nabla q - p'\otimes_\nabla q'\big)(x,y) \\
&= \sup_{f: X\otimes Y\to \R} \int_{X\otimes Y} f(x,y)\, d\big((p-p')\otimes q + p'\otimes_\nabla (q-q')\big)(x,y)\\
&= \sup_{f: X\otimes Y\to \R} \int_X \left\{ \int_Y f(x,y) \,dq(y) \right\} d(p-p')(x)  \\
&\qquad + \int_Y \left\{ \int_X f(x,y) \,dp'(x) \right\} d(q-q')(y) \\
&\le \sup_{g: X\to \R} \int_X g(x) d(p-p')(x) +  \sup_{h: Y\to \R} \int_Y h(y) d(q-q')(y) \\
&= d(p,p') + d(q,q')\\
&= d\big( (p,q), (p',q') \big),
\end{align*}
 where by replacing the partial integral of $f$ by $g$ we have used Proposition \ref{partintsm}. 
\end{proof}

\begin{proof}[Proof of Proposition \ref{nablanat}]
 By symmetry, it suffices to show naturality in $X$. Let $f:X\to Z$. We need to show that this diagram commutes:
\begin{equation*}\begin{tikzcd}
PX\otimes PY \ar{d}{f_*\otimes \id} \ar{r}{\nabla_{X,Y}} & P(X\otimes Y) \ar{d}{(f\otimes \id)_*} \\
PZ\otimes PY \ar{r}{\nabla_{Z,Y}} & P(Z\otimes Y)
\end{tikzcd}\end{equation*}
Now let $p\in PX,q\in PY$, and $g:Z\otimes Y\to \R$. Then
\begin{align*}
\int_{Z\otimes Y} f(z,y)\, d\big( (f\otimes \id)_* \nabla_{X,Y}(p,q) \big)(z,y) &= \int_{X\otimes Y} g(f(x),y)\,d(\nabla_{X,Y}(p,q)) (x,y)\\
&= \int_{X\otimes Y} g(f(x),y)\,dp(x)\,dq(y)\\
&= \int_{Z\otimes Y} g(z,y) \, d(f_*p)(z)\,dq(y) \\
&= \int_{Z\otimes Y} g(z,y) \, d\big( (f_*p)\otimes q \big)(z,y) \\
&= \int_{Z\otimes Y} g(z,y) \, d\big( \nabla_{Z,Y}\circ (f_*\otimes \id) (p,q) \big)(z,y) . 
\end{align*}
\end{proof}

\begin{proof}[Proof of Proposition \ref{lmf}]
 Since both maps are natural, we only need to check the coherence diagrams. Since the unitor is just the identity at the terminal object, the unit diagrams commute.
The associativity diagram at each $X,Y,Z$
\begin{equation*}
\begin{tikzcd}
PX \otimes PY \otimes PZ \ar{d}{\nabla_{X,Y}\otimes \id} \ar{r}{\id \otimes \nabla_{Y,Z}} & PX \otimes P(Y\otimes Z) \ar{d}{\nabla_{X,Y\otimes Z}} \\
P(X\otimes Y) \otimes PZ \ar{r}{\nabla_{X\otimes Y, Z}} & P(X\otimes Y\otimes Z)
\end{tikzcd}\end{equation*}
gives for $(p,q,r)\in PX \otimes PY \otimes PZ$ on one path
\begin{equation*}
(p,q,r) \mapsto (p\otimes_\nabla q, r) \mapsto (p\otimes_\nabla q) \otimes_\nabla r ,
\end{equation*}
and on the other path
\begin{equation*}
(p,q,r) \mapsto (p, q\otimes_\nabla r) \mapsto p\otimes_\nabla (q \otimes_\nabla r) . 
\end{equation*}
The product of probability distributions is now associative, as a simple calculation can show.

The symmetry condition is straightforward.
\end{proof}

\begin{proof}[Proof of Proposition \ref{mmonad}]
 We know that $(P,\id_1,\nabla)$ is a lax monoidal functor. We need to check now that $\delta$ and $E$ are monoidal natural transformations. Again we only need to show the commutativity with the multiplication, since the unitor is trivial.
For $\delta:\id_{\cat{CMet}}\Rightarrow P$ we need to check that this diagram commute for each $X,Y$:
\begin{equation*}
\begin{tikzcd}
X\otimes Y \ar[swap]{dr}{\delta} \ar{r}{\delta\otimes \delta} & PX \otimes PY \ar{d}{\nabla_{X,Y}} \\
& P(X\otimes Y)
\end{tikzcd}\end{equation*}
which means that for each $x\in X,y\in Y$ $\delta_x\otimes_\nabla\delta_y = \delta_{(x,y)}$, which is easy to check (the delta over the product is the product of the deltas).
For $E:PP\Rightarrow P$ we first need to find the multiplication map $\nabla^2_{X,Y}:PPX\otimes PPY\to PP(X\otimes Y)$ (the unit is just twice the deltas, and the unit diagram again trivially commutes). This map is given by
\begin{equation*}\begin{tikzcd}
P(PX) \otimes P(PY) \ar{r}{\nabla_{PX,PY}} & P(PX\otimes PY) \ar{r}{(\nabla_{X,Y})_*} & P(P(X\otimes Y))
\end{tikzcd}\end{equation*}
and more explicitly, if $\mu\in PPX, \nu\in PPY$, and $f:P(X\times Y)\to \R$,
\begin{align*}
\int_{P(X\otimes Y)} f(r)\,d\big( \nabla^2_{X,Y} (\mu,\nu) \big)(r) &= \int_{P(X\otimes Y)} f(r)\,d\big( (\nabla_{X,Y})_* \circ \nabla_{PX,PY} (\mu,\nu) \big)(r) \\
&= \int_{P(X\otimes Y)} f(r)\,d\big( (\nabla_{X,Y})_* (\mu\otimes_\nabla\nu) \big)(r) \\
&= \int_{PX\otimes PY} f(\nabla_{X,Y}(p,q))\,d(\mu\otimes_\nabla\nu)(p,q)\\
&=\int_{PX\otimes PY} f(p\otimes_\nabla q)\,d\mu(p)\,d\nu(q) .
\end{align*}
Now we have to check that this map makes this multiplication diagram commute:
\begin{equation*}
\begin{tikzcd}
PPX \otimes PPY \ar{d}{\nabla^2_{X,Y}} \ar{rr}{E_X\otimes E_Y} &&PX \otimes PY \ar{d}{\nabla_{X,Y}} \\
PP(X\otimes Y) \ar{rr}{E_{X\otimes Y}} && P(X\otimes Y)
\end{tikzcd}\end{equation*}
Now let $\mu\in PPX, \nu\in PPY$, and $g:X\times Y\to \R$. We have, using the formula for $\nabla^2$ found above,
\begin{align*}
&\int_{X\otimes Y} g(x,y)\, d\big( \nabla_{X,Y}\circ (E_X,E_Y)(\mu,\nu) \big)(x,y) =\\
&= \int_{X\otimes Y} g(x,y)\, d\big( \nabla_{X,Y} (E\mu,E\nu)\big)(x,y)\\
&= \int_{X\otimes Y} g(x,y)\, d\big( E\mu\otimes_\nabla E\nu\big)(x,y) \\
&= \int_{PX\otimes PY} \left\{ \int_{X\otimes Y} g(x,y)\,dp(x)\,dq(y) \right\} d\mu(p)\,d\nu(q) \\
&= \int_{PX\otimes PY} \left\{ \int_{X\otimes Y} g(x,y)\,d(p\otimes_\nabla q)(x,y) \right\} d\mu(p)\,d\nu(q) \\
&= \int_{P(X\times Y)} \left\{ \int_{X\otimes Y} g(x,y)\,dr(x,y) \right\} d(\nabla^2_{X,Y}(\mu,\nu))(r) \\
&= \int_{X\otimes Y} g(x,y)\, d\big( E_{X\otimes Y} \circ \nabla^2_{X,Y}(\mu,\nu) \big)(x,y) .
\end{align*}
Therefore the diagram commutes, and $(P,\delta,E)$ is a monoidal monad.
\end{proof}

\subsection{Opmonoidal structure of the Kantorovich monad}\label{proofsopmonoidal}

Just as in the case of joints, to prove the Proposition \ref{deltasm} we first prove the following useful result.

\begin{prop}\label{sumsm}
 Let $f:X\to\R$ and $g:Y\to\R$ be short. Then $(f+g):X\otimes Y\to\R$ given by $(x,y) \mapsto f(x) + g(y)$ is short.
\end{prop}

\begin{proof}[Proof of Proposition \ref{sumsm}]
 Let $x,x'\in X$ and $y,y\in Y$. Then
 \begin{align*}
  | f(x) + g(y) - f(x') - f(y') | &\le |f(x) - f(x')| + |g(y) - g(y')| \\
  &\le d(x,x') + d(y,y') = d\big( (x,y), (x',y') \big) .
 \end{align*}
\end{proof}

\begin{proof}[Proof of Proposition \ref{deltasm}]
To prove that $\Delta$ is short, let $p,q\in P(X\otimes Y)$, and denote $p_X,p_Y,q_X,q_Y$ their marginals. Then:
\begin{align*}
d &\big( \Delta(p) , \Delta(q) \big) = d \big( (p_X,p_Y) , (q_X,q_Y) \big) = d (p_X,q_X) + d (p_Y,q_Y) \\
&= \sup_{f:X\to\R} \int_X f(x) \, d(p_X-q_X)(x) + \sup_{g:Y\to\R} \int_Y g(y) \, d(p_Y-q_Y)(y) \\
&= \sup_{f:X\to\R} \int_{X\otimes Y} f(x) \, d(p-q)(x,y) + \sup_{g:Y\to\R} \int_{X\otimes Y} g(y) \, d(p-q)(x,y)  \\
&= \sup_{f:X\to\R} \sup_{g:Y\to\R} \int_{X\otimes Y} \bigg( f(x) + g(y) \bigg) \, d(p-q)(x,y) \\
&\le \sup_{h:X\otimes Y\to\R} h(x,y)\, d(p-q)(x,y)\\
&= d_{P(X\otimes Y)} (p,q) ,
\end{align*}
where by replacing $f+g$ with $h$ we have used Proposition \ref{sumsm}. 
\end{proof}

\begin{proof}[Proof of Proposition \ref{deltanat}]
 By symmetry, it suffices to show naturality in $X$. Let $f:X\to Z$. We need to show that this diagram commutes:
\begin{equation*}\begin{tikzcd}
P(X\otimes Y) \ar{d}{(f\otimes \id)_*} \ar{r}{\Delta_{X,Y}} & PX\otimes PY \ar{d}{f_*\otimes \id}   \\
P(Z\otimes Y) \ar{r}{\Delta_{Z,Y}} & PZ\otimes PY
\end{tikzcd}\end{equation*}
Let now $p\in P(X\otimes Y)$. We have to prove that:
\begin{equation*}
\Delta_{Z,Y} \circ (f\otimes \id)_* p = (f_*\otimes \id) \circ \Delta_{X,Y} (p) .
\end{equation*}
On one hand:
\begin{align*}
(f_*\otimes \id) \circ \Delta_{X,Y} (p) &= (f_*\otimes \id)(p_X,p_Y) \\
&= (f_*p_X,p_Y) .
\end{align*}
On the other hand, let $h:Z\to \R$ and $g:Y\to \R$ be short. Then:
\begin{align*}
\int_Z h(z) \, d(((f\otimes \id)_* p)_Z)(z) &= \int_{Z\otimes Y} h(z) \, d((f\otimes \id)_* p)(z,y) \\
&= \int_{X\otimes Y} h(f(x)) \, dp(x,y) \\
&= \int_{X} h(f(x)) \, dp_X(x) \\
&= \int_{Z} h(z) \, d(f_*p_X)(x)  ,
\end{align*}
and:
\begin{align*}
\int_Y g(y) \, d(((f\otimes id)_* p)_Y)(y) &= \int_{Z\otimes Y} g(y) \, d((f\otimes id)_* p)(z,y) \\
&= \int_{X\otimes Y} g(y) \, dp(x,y) \\
&= \int_{Y} g(y) \, dp_Y(y) ,
\end{align*}
so the two components are again $(f_*p_X,p_Y)$. 
\end{proof}

\begin{proof}[Proof of Proposition \ref{olmf}]
 We already have naturality of the maps, and the counitor is trivial, we just have to check coassociativity. Namely, that the following diagrams commutes for each $X,Y,Z$:
\begin{equation*}
\begin{tikzcd}
P(X\otimes Y \otimes Z) \ar[swap]{d}{\Delta_{X,Y\otimes Z}} \ar{rr}{\Delta_{X\otimes Y, Z}} && P(X\otimes Y) \otimes P(Z) \ar{d}{\Delta_{X\otimes Y} \otimes \id} \\
P(X) \otimes P(Y\otimes Z) \ar[swap]{rr}{\id\otimes \Delta_{Y\otimes Z}} && P(X) \otimes P(Y) \otimes P(Z)
\end{tikzcd}\end{equation*}
Now given $p\in P(X\otimes Y \otimes Z)$, we get:
\begin{equation*}
(\Delta_{X\otimes Y} \otimes \id) \circ \Delta_{X\otimes Y, Z} (p) = (\Delta_{X\otimes Y} \otimes \id) (p_{XY},p_Z) = (p_X,p_Y,p_Z) ,
\end{equation*}
and:
\begin{equation*}
(\id\otimes \Delta_{Y\otimes Z}) \circ \Delta_{X,Y\otimes Z} (p) = (\id\otimes \Delta_{Y\otimes Z}) (p_{X},p_{YZ}) = (p_X,p_Y,p_Z) ,
\end{equation*}
since there is only one way of forming marginals.

The symmetry condition is again straightforward.
\end{proof}

\begin{proof}[Proof of Proposition \ref{ommonad}]
 We know that $(P,\id_1,\Delta)$ is an oplax monoidal functor. We need to check now that $\delta$ and $E$ are comonoidal natural transformations. Again we only need to show the commutativity with the comultiplication, since the counitor is trivial.
For $\delta:\id_{\cat{CMet}}\Rightarrow P$ we need to check that this diagram commute for each $X,Y$:
\begin{equation*}
\begin{tikzcd}
X\otimes Y \ar[swap]{dr}{\delta\otimes \delta} \ar{r}{\delta} & P(X\otimes Y) \ar{d}{\Delta_{X,Y}} \\
& PX \otimes PY
\end{tikzcd}\end{equation*}
which means that for each $x\in X,y\in Y$, $(\delta_{(x,y)})_X=\delta_x$ and $(\delta_{(x,y)})_Y=\delta_y$, which is again easy to check (the marginals of a delta are the deltas at the projections).
For $E:PP\Rightarrow P$ we first need to find the comultiplication map $\Delta^2_{X,Y}:PP(X\otimes Y)\to PPX\otimes PPY$ (the unit is just twice the deltas, and the unit diagram again trivially commutes). This map is given by:
\begin{equation*}\begin{tikzcd}
P(P(X\otimes Y)) \ar{r}{(\Delta_{XY})_*} & P(PX\otimes PY) \ar{r}{\Delta_{PX,PY}} & P(PX) \otimes P(PY)
\end{tikzcd}\end{equation*}
and more explicitly, if $\mu\in P(P(X\otimes Y))$, and $f:PX\to\R$ and $g:PY\to\R$ are short:
\begin{align*}
\int_{PX} f(p)\, d\big( ((\Delta_{XY})_*\mu)_{PX} \big) (p) &= \int_{PX\otimes PY} f(p)\, d\big( ((\Delta_{XY})_*\mu)_{PX} \big) (p,q)\\
&= \int_{P(X\otimes Y)} f(r_X)\, d\mu(r)
\end{align*}
since $g$ only depends on $PX$, and analogously:
\begin{align*}
\int_{PY} g(q)\, d\big( ((\Delta_{XY})_*\mu)_{PY} \big) (q) &= \int_{P(X\otimes Y)} f(r_Y)\, d\mu(r) .
\end{align*}
We have to check that this map makes this multiplication diagram commute:
\begin{equation*}
\begin{tikzcd}
PP(X\otimes Y)  \ar{d}{\Delta^2_{X,Y}} \ar{rr}{E_{X\otimes Y}} && P(X\otimes Y) \ar{d}{\Delta_{X,Y}} \\
PPX \otimes PPY \ar{rr}{E_X\otimes E_Y} && PX \otimes PY
\end{tikzcd}\end{equation*}
Now let $\mu\in P(P(X\otimes Y))$, and $f:X\to \R$ and $g:Y\to\R$ short. We have, using the formula for $\Delta^2$ found above:
\begin{align*}
\int_X f(x)\, d\left( (E_{X\otimes Y} \mu )_X \right) (x) &= \int_{X\otimes Y} f(x)\, d\left( E_{X\otimes Y} \mu \right) (x,y) \\
&= \int_{P(X\otimes Y)} \left\{ \int_{X\otimes Y} f(x) \, dr(x,y) \right\} d\mu(r) \\
&= \int_{P(X\otimes Y)} \left\{ \int_{X} f(x) \, d(r_X)(x) \right\}\, d \mu  (r) \\
&= \int_{PX\otimes PY} \left\{ \int_{X} f(x) \, dp(x) \right\}\, d\big( (\Delta_{XY})_*\mu \big) (p,q) \\
&= \int_{PX} \left\{ \int_{X} f(x) \, dp(x) \right\}\, d\big( ((\Delta_{XY})_*\mu)_{PX} \big) (p) \\
&= \int_{X} f(x) \, d\big( E_X ((\Delta_{XY})_*\mu)_{PX} \big)(x),
\end{align*}
and analogously:
\begin{align*}
\int_Y g(y)\, d\left( (E_{X\otimes Y} \mu )_Y \right) (y) &= \int_{Y} f(y) \, d\big( E_Y ((\Delta_{XY})_*\mu)_{PY} \big)(y),
\end{align*}
which means:
\begin{align*}
\Delta_{X,Y} \circ E_{X\otimes Y} \mu &= (E_X\otimes E_Y) \circ \Delta_{PX,PY} (\Delta_{XY})_*\mu) \\
&= (E_X\otimes E_Y) \circ (\Delta_{PX,PY} \circ (\Delta_{XY})_*) \mu \\
&= (E_X\otimes E_Y) \circ \Delta^2_{X,Y} \mu .
\end{align*}
Therefore the diagram commutes, and $(P,\delta,E)$ is an opmonoidal monad.
\end{proof}

\subsection{Bimonoidal structure of the Kantorovich monad}\label{proofbimonoidal}

\begin{proof}[Proof of Proposition \ref{blmf}]
 We already know that $P$ is lax and oplax. We only need to check the compatibility diagrams between the two structures. The unit diagrams are trivial, because the unitors are trivial. The bimonoidality diagram:
\begin{equation*}
\begin{tikzcd}
& P(W\otimes X) \otimes P(Y\otimes Z) \ar[swap]{dl}{\nabla_{W\otimes X, Y\otimes Z}} \ar{dr}{\Delta_{W,X}\otimes \Delta_{Y,Z}} \\
P(W\otimes X \otimes Y \otimes Z) \ar[swap]{d}{\cong} & & P(W) \otimes P(X) \otimes P(Y) \otimes P(Z) \ar{d}{\cong} \\
P(W\otimes Y \otimes X \otimes Z) \ar[swap]{dr}{\Delta_{W\otimes Y,X\otimes Z}} & & P(W) \otimes P(Y) \otimes P(X) \otimes P(Z) \ar{dl}{\nabla_{W,Y}\otimes \nabla_{X,Z}} \\
& P(W\otimes Y) \otimes P(X\otimes Z)
\end{tikzcd}
\end{equation*}
says that given $p\in P(W\otimes X),q\in P(Y\otimes Z)$:
\begin{equation*}
\Delta_{W\otimes Y,X\otimes Z} \circ \nabla_{W\otimes X, Y\otimes Z} (p,q) = (\nabla_{W,Y}\otimes \nabla_{X,Z}) \circ (\Delta_{W,X}\otimes \Delta_{Y,Z}) (p,q)
\end{equation*}
Now on one hand:
\begin{align*}
(\nabla_{W,Y}\otimes \nabla_{X,Z}) \circ (\Delta_{W,X}\otimes \Delta_{Y,Z}) (p,q) &= (\nabla_{W,Y}\otimes \nabla_{X,Z}) (p_W,p_X,q_Y,q_Z) \\
&= (p_W\otimes_\nabla q_Y, p_X\otimes_\nabla q_Z) .
\end{align*}
On the other hand:
\begin{align*}
\Delta_{W\otimes Y,X\otimes Z} \circ \nabla_{W\otimes X, Y\otimes Z} (p,q) &= \Delta_{W\otimes Y,X\otimes Z} (p \otimes_\nabla q) .
\end{align*}
The marginal of $p \otimes q$ on $W\otimes Y$ is, by Fubini's theorem, let $f:W\otimes Y\to\R$:
\begin{align*}
\int_{W\otimes Y} f(w,y) \,d((p\otimes_\nabla q)_{WY})(w,y) &= \int_{W\otimes X\otimes Y\otimes Z} f(w,y) \,d(p\otimes_\nabla q)(w,x,y,z) \\
&= \int_{W\otimes X\otimes Y\otimes Z} f(w,y) \,dp(w,x)\,dq(y,z) \\
&= \int_{W\otimes Y} f(w,y) \,dp_W(w)\,dq_Y(y) \\
&= \int_{W\otimes Y} f(w,y) \,d(p_W\otimes_\nabla q_Y)(w,y) ,
\end{align*}
and similarly the marginal on $X\otimes Z$ is given by $p_X\otimes_\nabla q_Z$. In other words, if the pairs are independent, the components from different pairs are also independent.
It follows that $P$ is bilax monoidal.
\end{proof}

\section*{Acknowledgements}
 The authors would like to thank the anonymous reviewers for the very helpful comments.
\addcontentsline{toc}{section}{Acknowledgements}

\bibliographystyle{alpha}
\bibliography{catprob}
\addcontentsline{toc}{section}{\bibname}

\end{document}